\documentclass[11pt]{article}
\RequirePackage[colorlinks,citecolor=blue,urlcolor=blue,linkcolor=blue]{hyperref}
\usepackage{url}
\usepackage{svg}
\DeclareSymbolFont{slenderlargesymbols}{OMX}{ccex}{m}{n}
\DeclareMathSymbol{\prod}{\mathop}{slenderlargesymbols}{"51}
\usepackage{latexsym,amssymb,amsmath,amsthm,graphics,graphicx,float,psfrag, epsfig, color, authblk, enumerate}

\usepackage{pgfplots}
\usepackage{url,amssymb,amsmath,amsthm,amscd,paralist,bbm,wrapfig}
\usepackage{bm}
\usepackage{algorithm}
\usepackage{algorithmic} 

\usepackage{mathtools}

\usepackage[T1]{fontenc}
\usepackage[numbers]{natbib}
\tikzstyle{vertex}=[circle,black, fill=black, draw, inner sep=0pt, minimum size=6pt]

\usepackage{tikz}
\usetikzlibrary{decorations.pathreplacing}
\usepackage{xcolor}
\usepackage{mathtools}
\definecolor{cof}{RGB}{219,144,71}
\definecolor{pur}{RGB}{186,146,162}
\definecolor{greeo}{RGB}{91,173,69}
\definecolor{greet}{RGB}{52,111,72}
\pgfplotsset{compat=1.14}

\newcommand*{\email}[1]{\texttt{#1}}

\providecommand{\keywords}[1]
{
  \small	
  \textbf{Keywords:} #1
}

\providecommand{\classificationnumbers}[1]
{
  \small	
  \textbf{2020 AMS Mathematics Subject Classification:} #1
}

\newtheorem{innercustomgeneric}{\customgenericname}
\providecommand{\customgenericname}{}
\newcommand{\newcustomtheorem}[2]{%
  \newenvironment{#1}[1]
  {%
   \renewcommand\customgenericname{#2}%
   \renewcommand\theinnercustomgeneric{##1}%
   \innercustomgeneric
  }
  {\endinnercustomgeneric}
}

\newcustomtheorem{customasmp}{Assumptions}

\usepackage{latexsym,amssymb,amsmath,amsthm,graphics,graphicx,float,psfrag, epsfig, color, authblk, enumerate}

\usepackage{url,amssymb,amsmath,amsthm,amscd,paralist,bbm,wrapfig}
\usepackage{bm}
\usepackage{algorithm}
\usepackage{algorithmic} 

\usepackage{mathtools}

\usepackage[numbers]{natbib}

\newtheorem{thm}{Theorem}

\newtheorem{cor}[thm]{Corollary}

\newtheorem{ex}[thm]{Example}

\newcommand{\R}{\mathbb{R}}

\newcommand{\E}{\mathbb{E}}

\newcommand{\Tcal}{{\mathcal{T}}}
\newcommand{\Ccal}{\mathcal{C}}
\newcommand{\Rcal}{\mathcal{R}}
\newcommand{\Fcal}{\mathcal{F}}
\newcommand{\Scal}{\mathcal{S}}

\newcommand{\Bcal}{\mathcal{B}}
\newcommand{\Dcal}{\mathcal{D}}
\newcommand{\Ncal}{\mathcal{N}}
\newcommand{\Acal}{\mathcal{A}}
\newcommand{\Pcal}{\mathcal{P}}
\newcommand{\Lcal}{\mathcal{L}}

\newcommand{\vol}{\mathrm{vol}}

\newcommand{\Cfrak}{\mathfrak{C}}

\DeclareMathOperator*{\argmax}{arg\,max}
\DeclareMathOperator*{\argmin}{arg\,min}
\DeclareMathOperator*{\conv}{conv}

\DeclareMathOperator*{\rank}{rank}

\DeclarePairedDelimiterX{\infdivx}[2]{(}{)}{%
  #1\;\delimsize\|\;#2%
}

\newtheorem{theorem}{Theorem}
\newtheorem{lemma}{Lemma}
\newtheorem{corollary}{Corollary}
\newtheorem{proposition}{Proposition}

\newtheorem{definition}{Definition}

\makeatletter
\newcommand\ShortTitle{Optimal Regularization}
\newcommand\Author{O. Leong, E. O'Reilly, Y.S. Soh, V. Chandrasekaran}
\let\Title\@title
\def\ps@mystyle{%
      \def\@oddfoot{\hfill\thepage\hfill\makebox[0pt][l]{}}
      \def\@evenfoot{\hfill\thepage\hfill\makebox[0pt][l]{}}
      \def\@oddhead{%
       \ifodd\value{page}\relax
          \hfill\Author\hfill\makebox[0pt][l]{}%
      \else
          \makebox[0pt][l]{}\hfill\ShortTitle\hfill%
      \fi%
      }%
      \let\@mkboth\markboth}
\makeatother
\begin{document}
\title{Optimal Regularization for a Data Source}

\author[1]{Oscar Leong\thanks{Correspondence email: \email{oleong@stat.ucla.edu}}}
\author[2]{Eliza O'Reilly}
\author[3]{Yong Sheng Soh}
\author[4]{Venkat Chandrasekaran}
\affil[1]{Department of Statistics and Data Science, UCLA}
\affil[2]{Department of Applied Mathematics and Statistics, Johns Hopkins University}
            
\affil[3]{Department of Mathematics, National University of Singapore}

\affil[4]{Department of Computing and Mathematical Sciences and Electrical Engineering, Caltech}
\date{December 29, 2022,\ Revised: September 10, 2024}
\maketitle

\begin{abstract}
In optimization-based approaches to inverse problems and to statistical estimation, it is common to augment criteria that enforce data fidelity with a regularizer that promotes desired structural properties in the solution.  The choice of a suitable regularizer is typically driven by a combination of prior domain information and computational considerations.  Convex regularizers are attractive computationally but they are limited in the types of structure they can promote.  On the other hand, nonconvex regularizers are more flexible in the forms of structure they can promote and they have showcased strong empirical performance in some applications, but they come with the computational challenge of solving the associated optimization problems.  In this paper, we seek a systematic understanding of the power and the limitations of convex regularization by investigating the following questions: Given a distribution, what is the optimal regularizer for data drawn from the distribution? What properties of a data source govern whether the optimal regularizer is convex?  We address these questions for the class of regularizers specified by functionals that are continuous, positively homogeneous, and positive away from the origin.  We say that a regularizer is optimal for a data distribution if the Gibbs density with energy given by the regularizer maximizes the population likelihood (or equivalently, minimizes cross-entropy loss) over all regularizer-induced Gibbs densities.  As the regularizers we consider are in one-to-one correspondence with star bodies, we leverage dual Brunn-Minkowski theory to show that a radial function derived from a data distribution is akin to a ``computational sufficient statistic'' as it is the key quantity for identifying optimal regularizers and for assessing the amenability of a data source to convex regularization.  Using tools such as $\Gamma$-convergence from variational analysis, we show that our results are robust in the sense that the optimal regularizers for a sample drawn from a distribution converge to their population counterparts as the sample size grows large. Finally, we give generalization guarantees for various families of star bodies that recover previous results for polyhedral regularizers (i.e., dictionary learning) and lead to new ones for a variety of classes of star bodies.
\end{abstract}


\classificationnumbers{49, 52, 90}

\section{Introduction}
\label{sec:intro}


Balancing computational considerations and accurate modeling of various phenomena of interest is a fundamental and long-standing challenge in data science. In the fields of statistical estimation and inverse problems, this challenge manifests itself most prominently in the choice of an appropriate regularization functional for a data source.  Regularizers are commonly employed to augment data fidelity criteria in optimization-based methods for inverse problems, and they are used to promote desired structural properties in the solution.  Popular forms of structure for which regularization is routinely employed include smoothness, sparsity, matrices or tensors with small rank, and many others.  As inverse problems are often ill-posed due to the available data being noisy or incomplete, the selection of a suitable regularizer is crucial to obtaining high-quality solutions.  In deriving a regularizer that promotes a desired structure, a central goal is that of obtaining a computationally tractable optimization formulation.

The core property of an optimization problem that governs its computational tractability is whether the problem is convex or not.  Indeed, as articulated by Rockafellar \cite{Rockafellar93} some decades ago ``... the great watershed in optimization isn't between linearity and nonlinearity, but convexity and nonconvexity.''  Convex optimization problems can be solved reliably to global optimality, and they come with the rich toolkit of convex analysis that provide convergence guarantees and bounds on sensitivity to perturbations of the underlying problem data.  These attributes have significant consequences in the context of inverse problems; for example, in medical imaging applications solution reliability and robustness are of critical importance.  As such, the design of convex regularizers\footnote{Nonlinear inverse problems can potentially give rise to nonconvex data fidelity terms.  We assume for simplicity that such data fidelity terms are convex, and that they yield convex optimization problems when augmented with convex regularizers.} that promote a desired structure is a question that has been studied for many decades \cite{StatLearnSparsityBook}.  However, convex regularizers are limited in the types of structure they can promote.  As a result, there has been a lot of interest in recent years in the incorporation of nonconvex regularizers in solution methods for inverse problems, and the resulting nonconvex optimization formulations have showcased strong empirical performance in some applications \cite{Arridgeetal19}.  Nonetheless, these approaches come with few guarantees about solution quality, robustness to data perturbations, and convergence of numerical solution schemes.

Despite substantial literature devoted to the topic of regularizer selection, we still do not have a systematic understanding of the power and the limitations of convex regularization for general data sources.  For example, it is unclear whether there exist data sources for which a convex regularizer is optimal and one need not sacrifice computational efficiency.  In this paper, we take first steps towards addressing these larger objectives by formulating and answering the following specific questions: Given a data source specified as a probability distribution, what is the optimal regularizer that promotes the structure contained in that data?  Are there data sources for which this optimal regularizer is convex?  What are the ``sufficient statistics'' associated with a data source that govern the choice of the optimal regularizer?

\subsection{A Maximum-Likelihood Criterion for Optimal Regularizer Selection} \label{sec:MLE-criterion}

To make the preceding conceptual questions precise, we begin by fixing a family of functionals over which we identify optimal regularizers.  Whether convex or not, a regularizer ought to satisfy basic regularity properties that lead to well-behaved optimization formulations for inverse problems.  To this end, we consider the collection of functionals on $\R^d$ that are $(i)$ positive except at the origin; $(ii)$ positively homogeneous { of some degree $\alpha > 0$}; and $(iii)$ continuous.  These assumptions are commonplace in the context of inverse problems; in particular, positivity is natural as a regularizer serves as a penalty function to promote a desired structure, positive homogeneity aligns with the idea that we often wish to solve inverse problems at varying levels of signal-to-noise ratio using the same regularizer, and finally, continuity is appealing computationally as it yields well-behaved optimization formulations.  Most regularizers employed in the literature in inverse problems satisfy these three conditions. { For ease of exposition and simplicity of results, we will focus on positive homogeneity of degree $\alpha = 1$, but will discuss how our results extend to the general $\alpha > 0$ case in Section \ref{sec:pos-hom-general-degree}.}


A simple but important consequence of these three assumptions is that the collection of functionals satisfying these conditions is in one-to-one correspondence with the collection of star bodies in $\R^d$; we call a set $K \subseteq \R^d$ a \emph{star body} if it is compact with non-empty interior and for each $x \in \R^d \backslash \{0\}$ the ray $\{\lambda x : \lambda > 0\}$ intersects the boundary of $K$ exactly once. Given a star body $K \subseteq \R^d$, the regularizer associated with this star body is given by its gauge function (also called the Minkowski functional of $K$):
\begin{equation}
\|x\|_K = \inf\{t > 0 ~:~ x \in t \cdot K\}. \label{eq:gauge}
\end{equation}
This regularizer is nonconvex for general star bodies $K$ and is convex if and only if $K$ is convex.  The $1$-sublevel set of the regularizer is $K$.  See Section~\ref{sec:preliminaries} for background on star bodies.


With this formalization, we describe next what we mean by a regularizer being optimal for a given data source.  As data sources in this paper are characterized by probability distributions, a natural approach to obtaining an optimal regularizer is to map our collection of regularizers to a family of densities and to identify the regularizer for which the associated density maximizes the population likelihood.  Specifically, we associate a regularizer $\|\cdot\|_K$ given by the gauge of a star body $K \subseteq \R^d$ to the Gibbs density $p_K(x) \propto \exp\left(-\|x\|_K\right)$, which yields the following family of densities:
\begin{align*}
    \mathcal{D} := \left\{p_K : \R^d \rightarrow \R ~|~ p_K(x) := e^{-\|x\|_K}/Z_K,\ Z_K := \int_{\R^d} e^{-\|x\|_K} \mathrm{d} x < \infty,\ K\ \text{is a star body} \right\}.
\end{align*}
Note that this class is quite expressive, as it includes log-concave densities with convex gauges as energy functions as well as a wide class of nonconvex energies.  For a data distribution $P$ on $\R^d$, the negative population log-likelihood for any $p_K \in \mathcal{D}$ is given by\begin{align*}
    \E_P[-\log p_K(x)] = \E_P[\|x\|_K] + \log Z_K.
\end{align*}
The first term is the average value that the star body gauge assigns to points drawn according to the distribution $P$, and the second term corresponds to the normalizing constant.  This normalization is typically intractable to evaluate for general Gibbs densities \cite{GelfandSmith98}, but in our case it has an appealing geometric characterization in terms of the volume of the star body $K$ as described in the following result.



\begin{proposition}
    For any star body $K \subset \R^d$, we have \begin{align*}
        \int_{\R^d} e^{-\|x\|_K} \mathrm{d} x = \vol_d(K) \cdot \Gamma(d+1)
    \end{align*} where $\Gamma(z) := \int_0^{\infty} t^{z-1} e^{-t} \mathrm{d} t$ is the Gamma function.
\end{proposition}
\begin{proof}
    The result follows from elementary integration. Let $\mathbf{1}_{E}$ denote the indicator function on the event $E$. Recall that for any star body $K \subseteq \R^d$, we have $\int_{\R^d} \mathbf{1}_{\{x \in K\}} \mathrm{d} x = \vol_d(K)$ and $\vol_d(tK) = t^d\vol_d(K)$ for $t \geqslant 0$. Then observe that \begin{align*}
        \int_{\R^d} e^{-\|x\|_K} \mathrm{d} x  & = \int_{\R^d} \left(\int_{\|x\|_K}^{\infty} e^{-t}\mathrm{d} t\right) \mathrm{d} x 
        = \int_{\R^d} \left(\int_{0}^{\infty} \mathbf{1}_{\{\|x\|_K \leqslant t\}} e^{-t}\mathrm{d} t\right) \mathrm{d} x \\
        & = \int_{0}^{\infty}e^{-t}\left(\int_{\R^d} \mathbf{1}_{\{\|x\|_K \leqslant t\}} \mathrm{d} x \right)\mathrm{d} t =\int_0^{\infty}e^{-t} \vol_d(tK) \mathrm{d} t = \vol_d(K) \int_0^{\infty} t^d e^{-t} \mathrm{d} t.
        \end{align*}
\end{proof}

{ This result shows that minimizing the negative population log-likelihood objective $\E_P[-\log p_K(x)]$ over $p_K \in \mathcal{D}$ is equivalent to minimizing the following functional over all star bodies $K \subseteq \R^d$
\begin{align}
    K \mapsto \E_P[\|x\|_K] + \log \vol_d(K). \label{eq:MLE-objective}
\end{align}}
{There are additional observations one can make regarding this functional to arrive at a simpler, yet equivalent, formulation to maximum likelihood. For any fixed star body $K$, we consider the optimal positive dilate $\lambda$ that minimizes $\mathbb{E}_P[\|x\|_K] + \log\vol_d(K)$, i.e., we wish to solve $\min_{\lambda > 0} ~ \mathbb{E}_P[\|x\|_{\lambda K}] + \log\vol_d(\lambda K)$.  By noting that $\mathbb{E}_P[\|x\|_{\lambda K}] = \lambda^{-1} \mathbb{E}_P[\|x\|_K]$ and that $\vol_d(\lambda K) = \lambda^d \vol_d(K)$, we conclude from single-variable calculus that the expression $\lambda^{-1} \mathbb{E}_P[\|x\|_K] + \log\left(\lambda^d \vol_d(K)\right)$ is minimized when $\lambda_K := \E_P[\|x\|_K] / d$.  One can check that $\mathbb{E}_P[\|x\|_{\lambda_K K}] + \log\vol_d(\lambda_K K)$ is equal to $\log\left(\E_P\left[\|x\|_{K/\vol_d(K)^{1/d}}\right]\right)$ up to additive constants and multiplicative factors that do not depend on $K$. As $K/\vol_d(K)^{1/d}$ has unit volume, we conclude that in the minimization of the functional \eqref{eq:MLE-objective} over all star bodies $K$, it suffices to only consider unit-volume star bodies.  Consequently, we arrive at the following natural optimization problem:
\begin{equation}
\argmin_{\substack{K \subseteq \R^d, ~ \vol_d(K) = 1, \\ K \text{ is a star body}}} ~ \E_P[\|x\|_K] \label{eq:star-body-min}
\end{equation}
which is the main problem of interest in the paper.}

Solving the optimization problem \eqref{eq:star-body-min} entails many challenges as the space of star bodies is infinite-dimensional with its own intricate topology.  In particular, characterizing basic properties about the problem, such as existence and uniqueness of minimizers, requires an analysis of the map $K \mapsto \E_P[\|x\|_K]$ on the space of star bodies. To make progress on these challenges, we leverage concepts and ideas from the field of star geometry, with dual Brunn-Minkowski theory playing a prominent role \cite{Lutwak1975}.  In particular, as we show in Section \ref{sec:Opt_starbody_regularizers}, the expected gauge $\E_P[\|x\|_K]$ can be written as a \textit{dual mixed volume} between the star body $K$ and a special star body induced by the data distribution $P$, which we call $L_P$.  The star body $L_P$ serves as a ``computational sufficient statistic'' that summarizes all the relevant information about $P$ for the selection of a regularizer.  A virtue of this geometric characterization of $\E_P[\|x\|_K]$ is that the dual mixed volume between $K$ and $L_P$ can be bounded below by the product of the volumes of $K$ and $L_P$, with equality holding if and only if $K$ is a (positive) dilate of $L_P$, i.e., $K = \lambda L_P$ for some $\lambda > 0$ \cite{Lutwak1975}.

This problem \eqref{eq:star-body-min} constitutes the central mathematical object of interest in our paper.  We show in Section \ref{sec:Opt_starbody_regularizers} that under appropriate conditions the volume-normalized $L_P$ is the unique solution  to the above optimization problem.  In Sections~\ref{sec:general-existence-results}--\ref{sec:posterior-persp}, we further analyze variational, statistical, and learning-theoretic aspects of this problem; see Section~\ref{sec:our-contributions} for a summary of our contributions.


From an information-geometric perspective, our maximum likelihood criterion for regularizer selection is equivalent to computing a \textit{moment projection} (or $m$-projection) \cite{amari2000methods} of the data distribution $P$ onto the family of distributions $\overline{\mathcal{D}} := \{p_K \in \mathcal{D} : Z_K = \Gamma(d+1)\}.$  Assuming $P$ is absolutely continuous with respect to any $p_K \in \overline{\mathcal{D}}$ (i.e., $P \ll p_K$ for any $p_K \in \overline{\mathcal{D}})$, we have the following interpretation of the problem \eqref{eq:star-body-min}:
\begin{align*}
    \min_{p_K \in \overline{\mathcal{D}}} D_{\mathrm{KL}}(P || p_K) \Leftrightarrow \min_{p_K \in \overline{\mathcal{D}}} \int_{\R^d}\log\left(\frac{\mathrm{d}P}{\mathrm{d}P_K}(x)\right)\mathrm{d}P(x) \Leftrightarrow \max_{p_K \in \overline{\mathcal{D}}} \int_{\R^d} \log p_K(x) \mathrm{d}P(x) \Leftrightarrow \min_{p_K \in \overline{\mathcal{D}}} \mathbb{E}_P[\|x\|_K]
\end{align*}
where $P_K$ is the measure with density $p_K$. The minimizer $p^*_{K} := \argmin_{p_K \in \overline{\mathcal{D}}} D_{\mathrm{KL}}(P || p_K)$ is known as the $m$-projection of $P$ onto the family $\overline{\mathcal{D}}$, and it is given by the Gibbs density induced by the star body that solves the optimization problem \eqref{eq:star-body-min}.



\subsection{Our Contributions} \label{sec:our-contributions}

In Section~\ref{sec:minimizers-star-bodies}, we identify the solution of \eqref{eq:star-body-min} using techniques from dual Brunn-Minkowski theory \cite{Lutwak1975}.  This branch of geometric functional analysis extends classical results in Brunn-Minkowski theory concerning mixed volumes of convex bodies \cite{conv-bodies-Schneider} to the class of star bodies; we show that the objective in \eqref{eq:star-body-min} may be viewed as a dual intrinsic volume, and this interpretation yields a precise characterization of the optimal regularizer.  Given a data distribution that has a density with respect to the Lebesgue measure, we show that a radial function induced by the distribution's density provides a kind of ``summary statistic'' for the data that provides all the information required to obtain the optimal regularizer for that data; in particular, this radial function allows us to characterize when the optimal regularizer is convex.  Next, we establish the existence of minimizers of \eqref{eq:star-body-min} in Section~\ref{sec:general-existence-results} for a broader class of distributions than those considered in Section~\ref{sec:minimizers-star-bodies}.  Along the way, we describe a compactness result akin to Blaschke's Selection Theorem for ``well-conditioned'' star bodies.

In Section~\ref{sec:convergence-of-minimizers}, we show that our results are robust in the sense that the optimal regularizers for a sample drawn from a distribution converge to their population counterparts as the sample size grows large.  Our approach to establishing this result is based on tools from variational analysis, such as $\Gamma$-convergence \cite{Braides-Gamma-Handbook}, which we use to establish uniform convergence of the objective in \eqref{eq:star-body-min} as the sample size grows large.  These tools also allow us to show that star bodies learned on noisy data converge to star bodies learned on uncorrupted data as the noise level goes to zero (see Section \ref{sec:robustness}).

Finally, in Sections \ref{sec:stat-learning} and \ref{sec:posterior-persp}, we describe additional results on the statistical learning of star body regularizers and how our optimal regularizer characterization has consequences for inverse problems. Section \ref{sec:stat-learning} builds upon results in Sections \ref{sec:minimizers-star-bodies} -- \ref{sec:convergence-of-minimizers} for learning regularizers from certain parametrized families of star bodies. We prove a general uniform convergence result that gives a bound on generalization error as a function of the metric entropy of the family of bodies being considered. We then apply this result to several classes of interest, including star bodies given as unions of convex bodies, ellipsoids, polytopes, and linear images of Schatten-$p$ norm balls. Several of these generalization bounds are novel, while others recover known results in the literature for problems such as dictionary learning. Section \ref{sec:posterior-persp} discusses how our results have consequences for posterior sampling, a crucial ingredient for Bayesian inference in inverse problems.

\subsection{Related Work}

\subsubsection{Regularizers Derived from Domain Expertise}

The traditional paradigm in solving inverse problems incorporates the use of hand-crafted regularization functionals to promote structure, a practice dating as far back as Tikhonov regularization \cite{TikReg} and the Nyquist sampling theorem \cite{NyquistShannonSampling}. For instance, natural images have been observed to be compressible, or approximately sparse, under the wavelet or the discrete cosine transform. Subsequently, regularization functionals such as the $\ell_1$-norm in the transformed basis, have shown to be effective at inducing sparse structure and is computationally tractable to optimize over owing to its convexity \cite{Daubechiesetal04, Tao2006, Donoho2006, Tibshirani94}. A prominent convex regularizer promoting smoothness is total variation regularization \cite{Rudinetal92}, which also includes generalizations incorporating higher-order derivatives \cite{ChambolleLions, BrediesKunischPock}. Other examples of hand-crafted convex regularization techniques include the nuclear norm for low-rankness \cite{FazelThesis, CandesRecht09, Rechtetal10} as well as entropic regularization \cite{Eggermont93}. Several of these regularizers also fall under the wider class of atomic norm regularizers \cite{Chandrasekaranetal12a, BhaskarRecht13, Tangetal13, Shahetal12, OymakHassibi}. We also highlight the rich use of Gaussian priors \cite{Stuart-BIP, Knapiketal11, DashtiStuart, GiordanoNickl, Monardetal21}, which have been used extensively in infinite-dimensional inverse problems for regularization and uncertainty quantification.

While utilizing convex regularizers has computational benefits for optimization, there has been a recent surge of interest in the incorporation of nonconvex regularization functionals. For example, returning to sparsity as a model for structure, several approaches have proposed the use of $\ell_p$ regularizers for $p \in [0,1)$ \cite{FoucartLai09, Sun12, Mohimanietal08} or other nonconvex approximations to the $\ell_1$-norm \cite{YaoKwok18, Selesnick17, FanLi2001, Wangetal14, Zhang10, PieperPetrosyan22}. Please see \cite{BenningBurger18} for a comprehensive overview of hand-crafted regularizers for inverse problems. A significant challenge in this paradigm, however, is that choosing an appropriate regularizer for an inverse problem requires precise domain knowledge or expertise. Moreover, for general data sources, it is difficult to make claims regarding the optimality of such regularizers.

\subsubsection{Regularizers Learned from Data}

A dominant paradigm in recent years in designing regularizers has been to take a data-driven approach by directly \textit{learning} a regularization functional from example data. This framework has the advantage that the learned regularizer is particular to the data distribution of interest, and can help ameliorate a lack of domain knowledge or expertise. One prominent instantiation of this is given by learned convex regularizers from dictionary learning or sparse coding (see \cite{OlshausenField96, OlshausenField97, Spielmanetal12, Aharonetal2006, Aroraetal15, Agarwaletal16, Agarwaletal17, Schnass14, Schnass16, Baraketal15} and the surveys \cite{MairaletalSurvey14, EladSurvey10} for more). Given a training set of signals of interest, this class of methods computes a basis such that each example can be sparsely represented in this new basis. Once learned, such a basis can be incorporated in a regularization functional for subsequent tasks by enforcing that reconstructed data should also be sparsely represented in this basis. Geometrically, this regularizer constructs a polytope (whose extreme points are given by the basis elements) such that datapoints of interest should lie on its low-dimensional faces. Extensions of such an approach have also been investigated for particular types of bases, such as convolutional dictionaries \cite{Papyanetal17, Wohlberg15, GarciaCardonaWolhberg18}, and those representable by semidefinite programming using an infinite collection of basis elements \cite{SohChandrasekaran}. Such regularizers are convex, but there have also been many approaches incorporating more complicated nonconvex structures for data, such as learned regularizers parametrized by deep neural networks. Such approaches have showcased strong empirical performance and have been empirically observed to out-perform handcrafted regularizers in a range of tasks \cite{Lunzetal18, venkatakrishnan2013pnp, romano2017RED, Boraetal17, Kobleretal20, HabringHoller22, ReehorstSchniter19}. For a comprehensive survey on data-driven regularizers for inverse problems, please see the survey \cite{Arridgeetal19}. 

Despite this empirical success of data-driven regularizers, there is a lack of an overarching understanding as to why such approaches better capture underlying data geometry than handcrafted approaches. For example, when and why is it the case that certain data distributions can be captured or modeled by sparsity in a learned basis? On the other hand, which data sources should instead be modeled by nonconvex regularizers and how much information is lost by restricting oneself to a family of convex regularizers? These are the types of questions that we investigate in the present paper.

\subsubsection{Characterizing Optimal Regularizers}

To our knowledge, there are few works that consider the optimality of a regularizer for a given dataset. For example, in the area of dictionary learning, most prior work and performance guarantees analyze convergence guarantees of learning a synthetically generated, planted basis \cite{Agarwaletal16, Sunetal16I, Sunetal16II, ChatterjiBartlett17}, but these works do not consider the question of studying why a polyhedral regularizer is the ``best choice'' as a regularizer for a given dataset. A recent work that has tackled the question of identifying an optimal regularizer is \cite{Traonmilinetal-21}. Here, the authors consider signal recovery problems where the underlying signal has intrinsic low-dimensional structure. Given linear measurements of the signal, the goal is to characterize the optimal convex regularizer for the inverse problem. The authors introduce several notions of optimality for convex regularizers (dubbed compliance measures) and establish that canonical low-dimensional models, such as the $\ell_1$-norm for sparsity, are optimal for sparse recovery under such compliance measures. The results in this work, however, focus exclusively on optimality of convex regularizers for a given low-dimensional model of the underlying data source. In contrast, our results can prove optimality of both convex and nonconvex regularizers for general data sources modeled as distributions.

\section{Preliminaries} \label{sec:preliminaries}

In this section, we provide relevant background concerning convex bodies and star bodies. For a more detailed treatment of each subject, we recommend \cite{conv-bodies-Schneider} for convex geometry and \cite{geom-tom-Gardner, Hansen-starset-survey} for star geometry.

First, we review some basic functions. 
Let $K$ be a closed subset of $\R^d$. The \emph{radial function} of $K$ is defined by 
\begin{align}\label{e:radial}
\rho_{K} (x) := \sup \{ t > 0 : t \cdot x \in K \}. 
\end{align}
It follows that the gauge function of $K$ satisfies $\|x\|_K = 1/ \rho_{K} (x)$ for all $x \in \mathbb{R}^{d}$ such that $x \neq 0$.  
The \emph{support function} of $K$ at $x \in \mathbb{R}^{d}$ is given by
$$
h_{K}(x) := \sup \{ \langle x, y \rangle : y \in K \}.
$$
Finally, we note that the gauge function and the support function are homogeneous with degree $1$, while the radial function is homogeneous with degree $-1$.  As such, in what follows, we generally speak of these functions restricted to the unit sphere $\mathbb{S}^{d-1} := \{ u \in \mathbb{R}^d : \|u\|_{\ell_2} = 1 \}$. We also use $\vol_d(K)$ to denote the usual $d$-dimensional volume of $K \subseteq \R^d$.

Recall that a closed set $K \subseteq \mathbb{R}^{d}$ is \emph{convex} if $x , y \in K$ implies that $\theta x + (1-\theta)y \in K$ for all $0 \leqslant \theta \leqslant 1$.  We call a convex set $K$ a \emph{convex body} if it is also compact and contains the origin in its interior.
We say that a closed set $K \subseteq \mathbb{R}^{d}$ is \emph{star-shaped} (with respect to the origin) if  
for all $x \in K$, the line segment $[0,x] := \{tx : t \in [0,1]\}$ is also contained in $K$. An equivalent definition of a star body as described in the introduction is 
a compact set $K \subseteq \R^d$ that is star-shaped and has a positive and continuous radial function $\rho_K$.  
We will denote the class of convex bodies and star bodies in $\R^d$ by $\Ccal^d$ and $\Scal^d$, respectively. Note that $\Ccal^d$ is a strict subclass of $\Scal^d$ for $d > 1$.  Convex bodies are determined uniquely by their support function, and star bodies are determined uniquely by their radial function. 

For $K, L \in \mathcal{S}^d$, 
it is easy to see that $K \subseteq L$ if and only if $\rho_K \leqslant \rho_L$.  We say that $L$ is a \emph{dilate} of $K$ if there exists $\lambda > 0$ such that $L = \lambda K$; that is, $\rho_{L} = \lambda \rho_K$.  In addition, given a linear transformation $\phi \in GL(\R^d)$, one has $\rho_{\phi K}(x) = \rho_{K}(\phi^{-1}x)$ for all $x \in \R^d \backslash \{0\}$.

An additional important aspect of star bodies are their kernels. Specifically, we define the \emph{kernel} of a star body $K$ as the set of all points for which $K$ is star-shaped with respect to, i.e., $\ker(K) := \{x \in K : [x,y] \subseteq K,\ \forall y \in K\}$. Note that $\ker(K)$ is a convex subset of $K$ and $\ker(K) = K$ if and only if $K$ is convex. For a parameter $r > 0$, we define the following subset of $\mathcal{S}^d$ consisting of \emph{well-conditioned} star bodies with nondegenerate kernels:
$$\mathcal{S}^d(r) := \{K \in \mathcal{S}^d : rB^d \subseteq \ker(K)\}.$$ 
Here $B^d$ is the unit (Euclidean) ball in $\R^d$.

\subsection{Metric Spaces} \label{sec:prelims-metric-spaces}

Let $K,L \in \mathcal{S}^d$ 
be star bodies.  The \emph{Minkowski sum} between $K, L$ is given by $K+L := \{x + y: x \in K, y \in L\}$.  Following the definition of the support function, we have that $h_{K+L} = h_{K} + h_{L}$. 
The \emph{Hausdorff distance} is defined by
$$
d_H(K,L) := \inf\{ \varepsilon \geqslant 0  : K \subseteq L + \varepsilon B^d, L \subseteq K + \varepsilon B^d\}.
$$ 

If $K$ and $L$ are also convex, then the Hausdorff distance can be defined in terms of the infinity norm of their support functions $d_H(K,L) = \|h_K - h_L\|_{\infty}$ -- here, both support functions are defined over the unit-sphere $\mathbb{S}^{d-1}$. The Hausdorff distance also defines a metric over the space of nonempty compact sets. The metric space 
$(\Ccal^d, d_H)$ enjoys favorable compactness properties that will aid in establishing several results in our analysis. In particular, it is known that the metric space $(\Ccal^d, d_H)$ is \emph{locally compact}, a seminal result known as Blaschke's Selection Theorem \cite{conv-bodies-Schneider}. Prior to stating this result, we recall that a collection $\Cfrak$ of sets is bounded if there exists a sufficiently large ball $R B^d$ of radius $R < \infty$ such that $K \subseteq R B^d$ for all $K \in \Cfrak$.
\begin{theorem}[Blaschke's Selection Theorem\label{thm:Blaschke}]
The metric space of convex bodies is locally compact. That is, given any bounded collection $\Cfrak \subset \Ccal^d$ in the space of convex bodies and a sequence $(K_n) \subset \Cfrak$, there exists a subsequence $(K_{n_m})$ and a convex body $K \in \Ccal^d$ such that $K_{n_m} \rightarrow K$ in the Hausdorff metric. 
\end{theorem}

While this result plays an important role in the context of our results on convex bodies, we also require an analogous result for star bodies. A generalization is known for star bodies equipped with the Hausdorff metric \cite{Hirose1965}, however this is not the metric for which relevant continuity properties over the space of star bodies hold. We present here an alternative metric -- the radial metric -- over star bodies for which it is possible to derive a Blaschke Selection Theorem.  In the convex geometry literature, the radial metric is often more natural over the space of star bodies than the Hausdorff metric. Concretely, let $K, L \subset \mathbb{R}^{d}$ be star bodies.  We define the \emph{radial sum} $\tilde{+}$ between $K$ and $L$ as $K \tilde{+} L := \{ x + y : x \in K, y \in L, x = \lambda y \}$; that is, unlike the Minkowski sum, we restrict the pair of vectors to be parallel.  The radial sum obeys the relationship $\rho_{K \tilde{+} L}(u) := \rho_{K}(u) + \rho_{L}(u)$.  
We denote the \emph{radial} metric between two star bodies $K,L$ as
$$
\delta(K,L) := \inf\{ \varepsilon \geqslant 0 : K \subseteq L \,\tilde{+}\, \varepsilon B^d, L \subseteq K\, \tilde{+}\, \varepsilon B^d\}.
$$
In a similar fashion, the radial metric satisfies $\delta(K,L) := \|\rho_K - \rho_L\|_{\infty}$.

\subsection{Brunn-Minkowski Theory}

The Brunn-Minkowski theory for convex bodies \cite{conv-bodies-Schneider} combines Minkowski addition and notions of volume to understand the structure of the space of convex bodies and prove important geometric inequalities.  Starting with Lutwak \cite{Lutwak1975}, a dual Brunn-Minkowski theory for star bodies equipped with radial addition was also developed. This theory studies functionals that measure the size of a star body $K$ relative to another $L$ in a particular way, and studies the extremals of such functionals. The importance of such results in the context our work is that we show our objective of interest has an equivalent interpretation as this functional of interest, and such results analyzing the extremals can aid in characterizing minimizers of our optimization problem. More concretely, the \emph{dual mixed volume} between two star bodies $K, L \in \mathcal{S}^d$ for any $i \in \R$ is defined as:
\begin{align}
    \tilde{V}_i(K, L) := \frac{1}{d}\int_{\mathbb{S}^{d-1}} \rho_K(u)^{i}\rho_L(u)^{d-i} \mathrm{d}u. \label{eq:dual-mixed-vol-definition}
\end{align}
Note that for all $i$,  $\tilde{V}_i(K, K) = \frac{1}{d}\int_{\mathbb{S}^{d-1}} \rho_K(u)^{d} \mathrm{d}u = \mathrm{vol}_d(K)$ is the usual $d$-dimensional volume of $K$. Lutwak \cite{Lutwak1975} analyzed upper and lower bounds of $\tilde{V}_i(K,L)$ depending on the volumes of $K$ and $L$, and characterized for which sets $K$ and $L$ would these bounds be attained. An important case for our purposes is the following result when $i = -1$: \begin{theorem}[Special Case of Theorem 2 in \cite{Lutwak1975}] \label{thm:lutwak-dmv} For star bodies $K,L \in \Scal^d$, we have
\[\tilde{V}_{-1}(K, L)^d \geqslant \vol_d(K)^{-1}\vol_d(L)^{d + 1},\]
and equality holds if and only if $K$ and $L$ are dilates.

\end{theorem}

{
\subsection{Notation}

For a distribution $P$, let $F(\cdot; P) : \Scal^d \rightarrow \R$ denote the functional $F(K; P) := \E_P[\|x\|_K]$. The letters $C$ and $c$ will be reserved for numerical constants, with subscripts denoting quantities they may depend on, e.g., $C_{\Cfrak}$ will denote a constant that depends on $\Cfrak \subset \Scal^d$. We say that $a \lesssim b$ if and only if $a \leqslant Cb$ for some positive absolute constant $C$. Note that $a \gtrsim b$ is defined in an analogous fashion. Special constants include $\kappa_{\ell} := \vol_{\ell}(B^{\ell})$ where $B^{\ell}$ is the unit Euclidean ball in dimension $\ell \in [d]$ and $\vol_{\ell}(\cdot)$ is the $\ell$-dimensional volume functional. For a parameter $r > 0$, we set $R_r := \frac{d+1}{r^{d-1}\kappa_{d-1}}$ and $\Scal^d(r) := \{K \in \Scal^d : rB^d \subseteq \ker(K)\}$. We further set $\Scal^d(r,1):=\Scal^d(r) \cap \{K \in \Scal^d : \vol_d(K) = 1\}.$ The corresponding restrictions of these sets to convex bodies are denoted by $\Ccal^d(r)$ and $\Ccal^d(r,1)$, respectively.
}

\section{Minimizers of the Population Risk} \label{sec:minimizers-star-bodies}

In this section, we will describe a setting where the unique optimal regularizer for a data distribution $P$ can be derived and provide a characterization for when this optimal regularizer is convex. 
As described in the introduction, we aim to find the unique star body that solves the following optimization problem:
\begin{align}\label{e:opt_1}
\argmin_{K \in \mathcal{S}^d: \mathrm{vol}_d(K) = 1}    \E_P[\|x\|_K]. 
\end{align}
The optimal regularizer is convex when the solution of \eqref{e:opt_1} is a convex body. We will see in the following that a unique solution can be obtained by interpreting the objective as a dual mixed volume.

In Section \ref{sec:radial_function} we define the fundamental quantity determined by $P$ that is needed to characterize the optimal star body. In Section \ref{sec:Opt_starbody_regularizers} we will state the main result characterizing the solution to \eqref{e:opt_1} and in Section \ref{sec:opt_reg_convex} we give a condition on $P$ that implies the optimal star body regularizer will be convex. Finally in Section \ref{sec:examples} we provide examples of different data distributions $P$ for which we apply this theory to obtain a description of the optimal regularizer.

\subsection{Radial Function Associated to a Data Distribution}\label{sec:radial_function}

Given a data distribution $P$, we would like to find a way to summarize the relevant structure that determines an optimal regularizer. 
Since we have restricted to positively homogeneous functions, the regularizer is characterized by a function on the unit sphere. Thus, given a data distribution, we need a statistic related to the distribution in each unit direction.
Recall that star bodies are uniquely determined by the distance from the origin to the boundary in each direction given by the radial function of the star body. Thus, we can hope to characterize an optimal regularizer through a function on the unit sphere obtained from $P$ that determines the radial function of the optimal star body that solves \eqref{e:opt_1}.

To be precise, we first restrict to data distributions $P$ that have a density with respect to Lebesgue measure on $\R^d$. We then define the following function on the unit sphere for a given $P$ with density $p$:
\begin{align}\label{e:rhoP}
    \rho_{P}(u) := \left(\int_0^{\infty} r^d p(ru) \mathrm{d} r\right)^{1/(d+1)}, \qquad u \in \mathbb{S}^{d-1}.
\end{align} Under certain conditions, the function $\rho_P$ is precisely the radial function of a star body $L_P$ in $\mathbb{R}^d$.
Using this interpretation, we will see that $\rho_P$ captures all the relevant information from the data distribution in order to determine a unique optimal regularizer.

Note that many probability distributions may correspond to the same radial function and associated star body. The following is thus a natural question: What aspects of the distribution $P$ characterize the star body $L_P$? Intuitively, $\rho_P(u)$ captures how much and how far away on average the mass of the distribution is in a given direction $u$. Thus the corresponding star body $L_P$ will have larger radius in directions where there is more mass of the distribution and where the mass is farther away from the origin. More precisely, $\rho_P^{d+1}$ is the density of the measure
$\mu_P(\cdot) = \E_P\left[\|x\|_{\ell_2} \mathbf{1}_{\{x/\|x\|_{\ell_2} \in \cdot\}}\right]$ on the unit sphere.
Additionally, $\rho_P$ captures linear transformations of the data. For a random vector with distribution $P$ and its transformation under a linear map $\phi$, the radial function corresponding to the distribution of the transformed vector will define a dilate of the linear image $\phi(L_P)$. For example, 
rotations of the data distribution will be captured by a rotation of $L_P$. To illustrate what the radial function statistic does not capture, the following example describes a class of probability distributions that all correspond to a constant radial function.
\begin{ex}Let $f : \mathbb{S}^{d-1} \to [0, \infty)$ be a bounded function on the unit sphere, and define another function
\[g_f(u) := \left(1 + f(u)^{d+1}\right)^{\frac{1}{d+1}}, \quad u \in \mathbb{S}^{d-1}.\]
Note that $g_f > f$. Consider the compact set
\[A_f = \{x \in \R^d : f(x/\|x\|_{\ell_2}) \leqslant \|x\|_{\ell_2} \leqslant g_f(x/\|x\|_{\ell_2})\},\]
and let $P_f$ be the uniform distribution over $A_f$. For any $f$, the radial function associated with $P_f$ satisfies 
\[\rho_{P_f}(u)^{d+1} = \frac{1}{\vol_d(A_f)}\int_{f(u)}^{g_f(u)} r^d \mathrm{d} r = \frac{1}{(d+1)\vol_d(A_f)}\left(g_f(u)^{d+1} - f(u)^{d+1}\right) = \frac{1}{(d+1)\vol_d(A_f)},\]
and the corresponding star body $L_{P_f}$ is a dilate of $B^d$.
Thus, the entire class of distributions $\{P_f : f : \mathbb{S}^{d-1} \to [0, \infty) \text{ is bounded}\}$ correspond to constant radial functions and under the volume normalization will all correspond to the same optimal star body given by a unit volume $\ell_2$ ball. See Figure \ref{fig:l2norm-supports} for examples of the support of $P_f$ for different functions $f$. Intuitively, each distribution $P_f$ induces the same star body (up to scaling) because in each direction, the average mass of the distribution (weighted by its distance from the origin) is the same, regardless of the behavior of $f$.
\end{ex}
 \begin{figure}
    \centering
    \includegraphics[width=.27\textwidth]{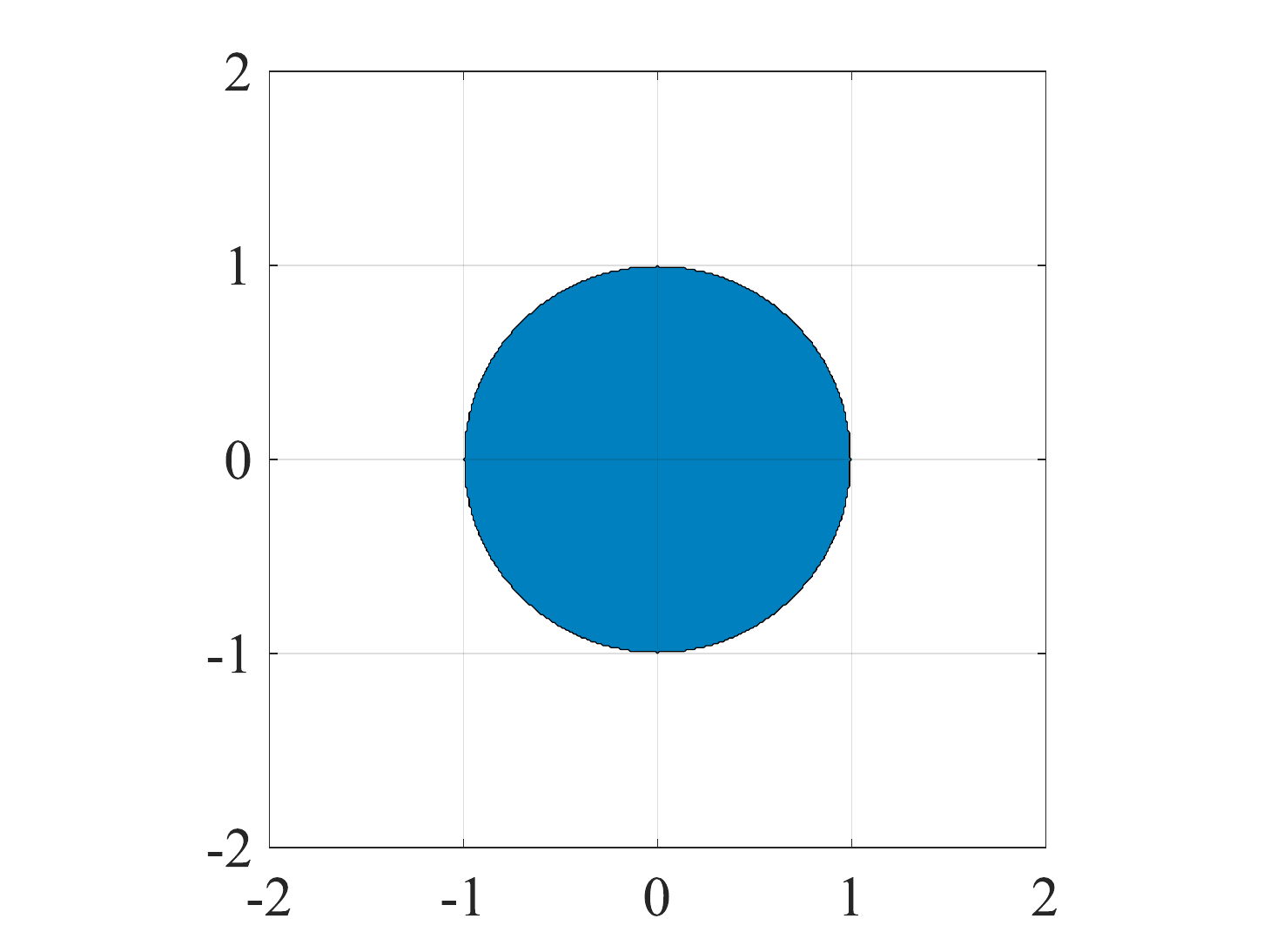}\hspace{-.75cm}
    \includegraphics[width=.27\textwidth]{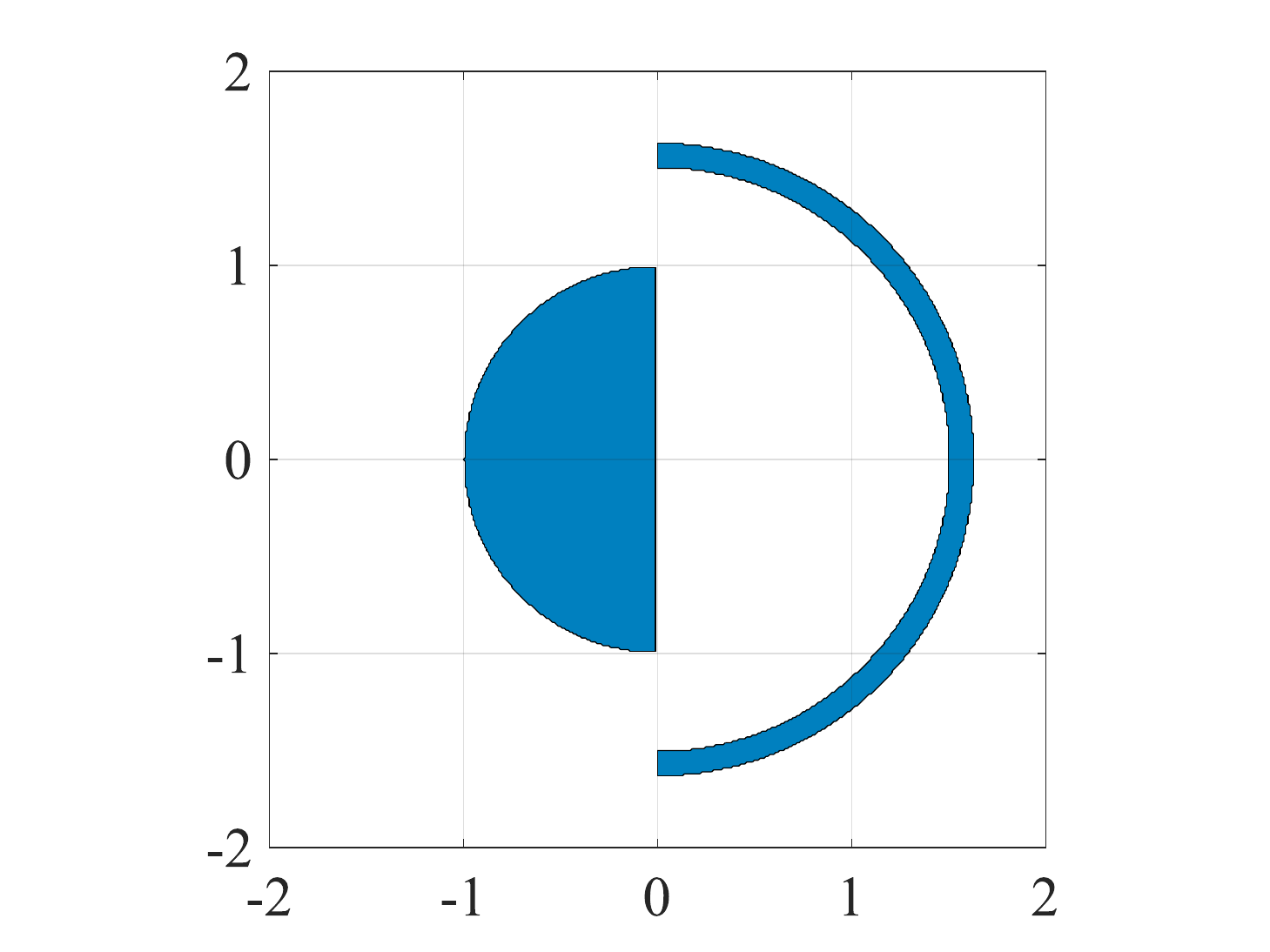}\hspace{-.75cm}
    \includegraphics[width=.27\textwidth]{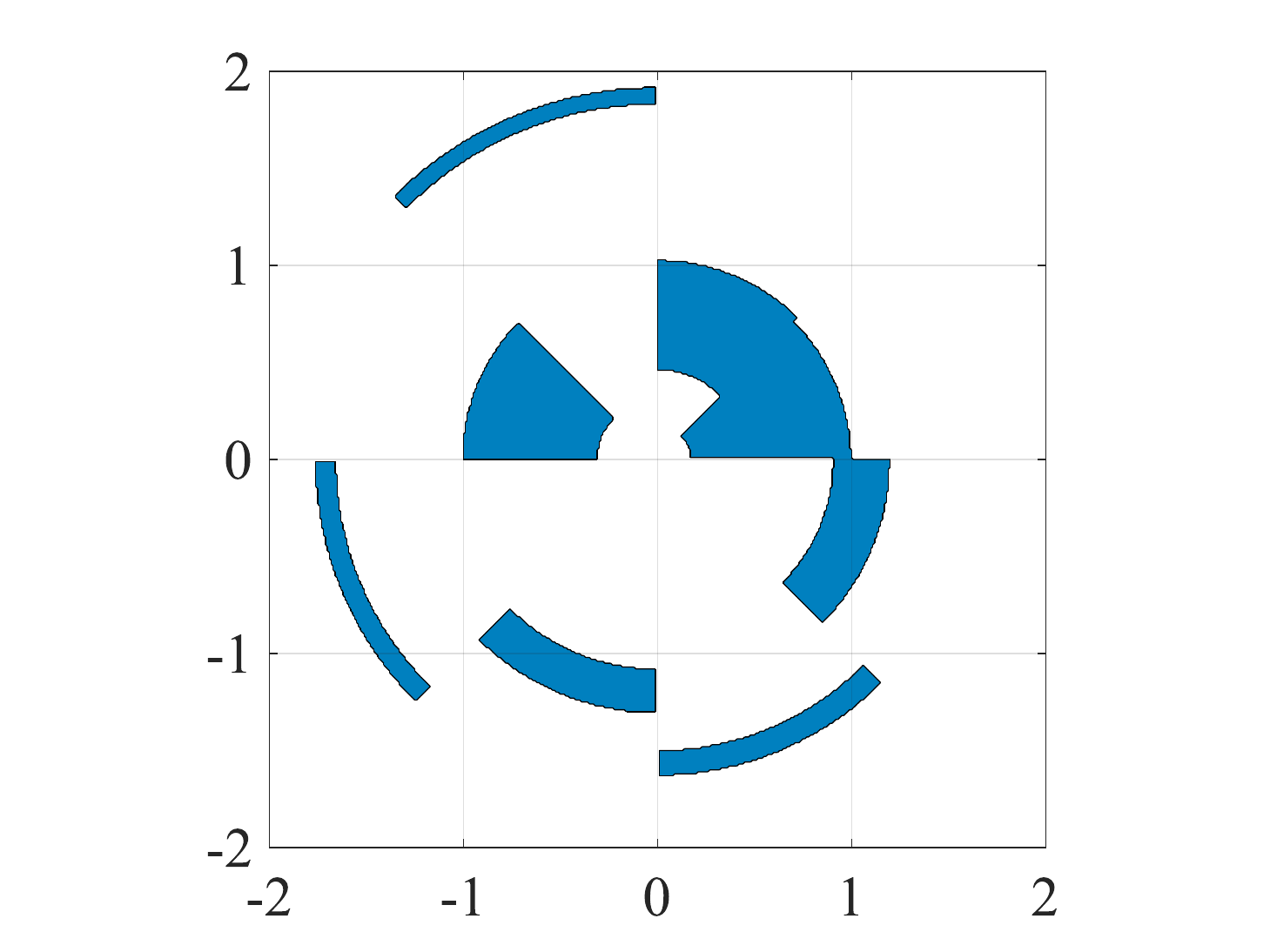}\hspace{-.75cm}
    \includegraphics[width=.27\textwidth]{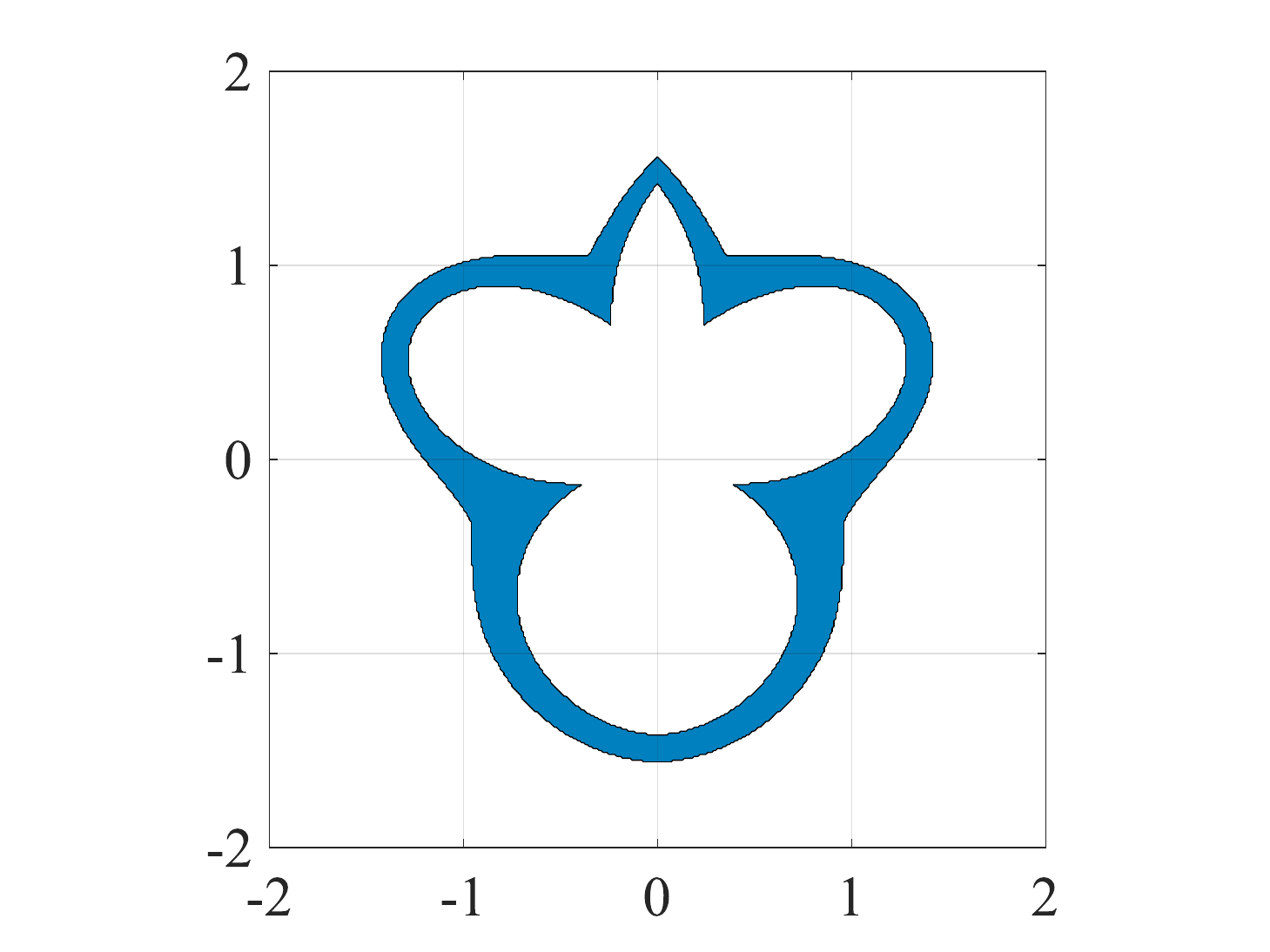}
    \caption{
    Plots of the support of four different compact sets $A_f$ for different functions $f$. The uniform distribution over each of these supports will correspond to the same constant radial function, and thus have the same optimal star body regularizer given by a scaled $\ell_2$-norm.}
    \label{fig:l2norm-supports}
\end{figure}

\subsection{Optimal Star Body Regularizers}\label{sec:Opt_starbody_regularizers}

We now state and prove our main result characterizing the optimal star body regularizer for \eqref{eq:star-body-min}. Given a data distribution $P$ of interest, our result establishes that there is a unique star body $K_*$ that achieves the minimum value of the objective over the space of star bodies of constant volume. 
The conditions on our distribution are fairly mild, and are satisfied for any distribution with a continuous density and finite second moments. They are also satisfied for more general cases, and we will explore examples of such distributions in the sequel.  Our main result is stated as follows:

\begin{theorem}\label{t:starbody_min}
Let $P$ be a distribution on $\R^d$ with density $p$ with respect to Lebesgue measure, and assume $\E_P[\|x\|_{\ell_2}] < \infty$. Suppose the function $\rho_P$ as defined in \eqref{e:rhoP} is positive and continuous over the unit sphere. Then there exists a unique star body $L_P \in \mathcal{S}^d$ whose radial function is $\rho_P$, and the set $K_* := \mathrm{vol}_d(L_P)^{-1/d}L_P$ is the unique solution to the minimization problem \eqref{e:opt_1}.

\end{theorem}

\begin{proof}[Proof of Theorem \ref{t:starbody_min}]
The main idea of our proof is that the population objective can be written as a dual mixed volume with respect to the star body $K$ we are optimizing and the distribution-dependent star body $L_P$. Once this is established, we can utilize the extremal inequality in Theorem \ref{thm:lutwak-dmv} to characterize the minimizers of our objective subject to a fixed volume constraint.

We first construct the star body $L_P$. Define the function 
\[f(x) := \left(\int_0^{\infty} r^d p(rx) \mathrm{d} r\right)^{-1/(d+1)}, \quad x \in \mathbb{R}^d.\]
Note that $f$ is positively homogeneous and non-negative. Define the subset $L_P$ of $\R^d$ by
\[L_P := \{ x \in \R^d : f(x) \leqslant 1\}.\]
The set $L_P$ is star-shaped, because for $x \in L_P$ we have for all $t \in [0,1]$, $f(tx) = tf(x) \leqslant t \leqslant 1$, and thus, $[0,x] \subseteq L_P$.
The gauge of $L_P$ is then
\begin{align*}
\|x\|_{L_P} &= \inf\{ t : x \in t L_P\} = \inf\{ t : x/t \in L_P\} \\
&= \inf\{ t : f(x/t) \leqslant 1\} = \inf\{ t : f(x) \leqslant t \} = f(x).   
\end{align*}
Thus, $f(x)$ is the gauge of $L_P$. In addition, note that $\rho_P(u) = 1/f(u)$ for $u \in \mathbb{S}^{d-1}$ where $\rho_P$ is defined in \eqref{e:rhoP}. Thus $\rho_P$ is the radial function of $L_P$, and $L_P$ is a star body because $\rho_P$ is positive and continuous.

We now establish that the objective in \eqref{e:opt_1} can be written as a dual mixed volume. Observe that by positive homogeneity of the gauge, we have via integrating in spherical coordinates that
\begin{align*}
\E_P[\|x\|_K] &= \int_{\R^d} \|x\|_K p(x) \mathrm{d}x = \int_{\mathbb{S}^{d-1}} \int_{0}^{\infty} r^d\|u\|_K p(ru) \mathrm{d}r\mathrm{d}u\\
&= \int_{\mathbb{S}^{d-1}} \|u\|_K \rho_{P}(u)^{d+1} \mathrm{d} u
\end{align*} where $\rho_P$ is defined in \eqref{e:rhoP}. Moreover, since $\|u\|_K = 1/\rho_K(u)$, the definition of the dual mixed volume gives

\begin{align*}
\E_P[\|x\|_K] &= \int_{\mathbb{S}^{d-1}} \|u\|_K \rho_{P}(u)^{d+1} \mathrm{d} u = \int_{\mathbb{S}^{d-1}} \rho_K(u)^{-1} \rho_{P}(u)^{d+1} \mathrm{d} u = d\tilde{V}_{-1}(K, L_P).
\end{align*} Hence, we can apply Theorem \ref{thm:lutwak-dmv} to show that our objective satisfies \begin{align*}
    \E_P[\|x\|_K] = d\tilde{V}_{-1}(K,L_P) \geqslant d\vol_d(K)^{-1/d}\vol_d(L_P)^{(d+1)/d}
\end{align*} with equality if and only if $K$ and $L_P$ are dilates. 
Thus,
$\E_P[\|x\|_K]$ is minimized over the collection $\{K \in \mathcal{S}^d : \vol_d(K) = 1\}$ by $K_* := \vol_d(L_P)^{-1/d}L_P$.
\end{proof}

\subsection{When is the Optimal Regularizer Convex?}\label{sec:opt_reg_convex}

Theorem \ref{t:starbody_min} explicitly characterizes the relationship between the optimal star body $K_*$ and the distribution $P$. As previously observed, if the star body $K_*$ is in fact a convex body, then the optimal regularizer is a convex function. This connection helps to justify intuition for when a data source modeled by $P$ is amenable to convex regularization. For example, consider the class of log-concave distributions which include many common distributions such as Gaussian, Laplace, and uniform distributions over convex bodies. Since the densities of such distributions have convex level sets, it is natural to guess that a convex regularizer would be most appropriate. If the logarithm of the density is both concave and positively homogeneous of some degree $k$, then the distribution $P$ falls into the class of probability distributions considered in Example \ref{ex:starbodies-density} below. In particular, the density is then a function, $\psi(r) = \exp(-r^k)$, of the gauge of a convex body $L$ and in Example \ref{ex:starbodies-density} we show that the optimal star body will be $L$ by computing the radial function $\rho_P$. Thus, since $L$ is convex, a convex regularizer is optimal for these log-concave distributions.

We can also use Theorem \ref{t:starbody_min} to characterize when a convex regularizer is optimal for general distributions. In particular, the density of $P$ defines the radial function of $K_*$, which determines whether or not the set $K_*$ is convex. This observation leads to an important contribution of this work: a condition on the data distribution $P$ such that the optimal regularizer is convex. 

\begin{corollary}\label{cor:convex_opt}
Let $P$ be a probability measure on $\R^d$ as in Theorem \ref{t:starbody_min}. Define the radial function $\rho_{P}$ as in \eqref{e:rhoP}. If $x \mapsto 1/\rho_P(x)$ is a convex function on $\R^d$, then the optimal $K_*$ as defined in Theorem $\ref{t:starbody_min}$ is a convex body, and the optimal star body regularizer is convex. 
\end{corollary}

This corollary gives a characterization of data distributions where convexity is the correct structure that should be sought among a class of both convex and nonconvex regularizers, and provides a tool for determining whether a given data source is amenable to convex regularization. We will illustrate in Example \ref{ex:gmm-sol} a setting where we can use this corollary to determine a parameter range for which the optimal regularizer for a parametric class of distributions is convex.

\subsection{Examples of Optimal Regularizers for a Given Distribution}\label{sec:examples}

In order to provide more intuition on the above results, we discuss here several additional examples of distributions $P$ and their associated stay body $L_P$ with radial function $\rho_P$ defined in \eqref{e:rhoP}. The two general classes of distributions we consider are distributions that depend on a compact body and general mixture distributions whose mixture components satisfy the assumptions of our Theorem. When the data distribution explicitly depends on the gauge of a star body $L$, then the optimal star body will be a dilate of $L$. For mixture distributions, the optimal star body will be a special sum of sets, akin to radial addition for star bodies. 
We now dive deeper into such examples here:

 \begin{ex}[Densities induced by star bodies]\label{ex:starbodies-density} Suppose $P$ has density $p(x) = \psi(\|x\|_L)$, i.e., the density only depends on the gauge function induced by a star body $L$. If $\psi$ is such that the integral $\int_0^{\infty} t^d \psi(t) dt < \infty$,  then we have that the corresponding radial function of $P$ satisfies
\begin{align*}
\rho_P(u) := \left(\int_{0}^{\infty} r^{d}\psi(r\|u\|_L) \mathrm{d}r\right)^{1/(d+1)} =  \rho_L(u) \left[\int_{0}^{\infty}t^d \psi(t) \mathrm{d}t\right]^{1/(d+1)} = \rho_{c(\psi)L}(u),
 \end{align*}
where $c(\psi) :=  \left[\int_{0}^{\infty}t^d \psi(t) \mathrm{d}t\right]^{1/(d+1)}$. Hence the optimal star body $K_*$ is the unit volume dilate of $L$:
 \[K_* = \vol_d(L)^{-1/d}L.\] 
There are several examples of distributions of interest that fall under this category, which we outline here.

\begin{itemize}
\item Suppose $L = \Bcal_{\ell_q} := \{x \in \R^d : \|x\|_{\ell_q} \leqslant 1\}$ is the $\ell_q$-function unit ball for $0 < q \leqslant \infty$ and the density of $P$ is given by $p(x) = \psi(\|x\|_{\ell_q})$. Then our optimal set is a dilate of the $\ell_q$-ball: $K_* = \vol_d(\Bcal_{\ell_q})^{-1/d}\Bcal_{\ell_q}.$ Note that this includes both convex norm balls $q \geqslant 1$ and nonconvex unit balls $q \in (0,1).$
 \item Another class of distributions $P$ that satisfy our assumptions are uniform distributions over a star body $L$. Let $\mathbf{1}_{E}$ denote the indicator function on the event $E$. In this case, the density is given by 
 \[p(x) = \frac{\mathbf{1}_{\{x \in L\}}}{\mathrm{vol}_d(L)} = \frac{\mathbf{1}_{\{\|x\|_L \leqslant 1\}}}{\mathrm{vol}_d(L)},\] and the radial function $\rho_P$ is
 \begin{align*}
     \rho_{P}(u) = \frac{1}{\mathrm{vol}_d(L)^{1/(d+1)}} \left(\int_0^{\infty} t^d 1_{\{t \leqslant1\}}\mathrm{d}t\right)^{1/(d+1)}\|u\|_L^{-1} =: c_{d,L}\rho_L(u)
 \end{align*} where $c_{d,L} := \left((d+1)\vol_d(L)\right)^{-1/(d+1)}.$
Hence $L_P := c_{d,L}L$, and 
 thus $K_*$ is given by
 \[K_* = \vol_d(L)^{-1/d}L.\]
 \item Suppose $P = \mathcal{N}(0,\Sigma)$ is a multivariate Gaussian distribution with mean zero and covariance matrix $\Sigma$ in $\R^d$. Then, a direct calculation shows that the radial function $\rho_P$ is given by
\begin{align*}
    \rho_P(u) = \det(2\pi\Sigma)^{-\frac{1}{2(d+1)}}\left(\int_0^{\infty} t^d e^{-t^2/2}\mathrm{d}t \right)^{\frac{1}{d+1}}\|\Sigma^{-1/2}u\|_{\ell_2}^{-1}.
\end{align*}
Note that the norm $\|\Sigma^{-1/2}u\|_{\ell_2}$ defines the gauge of an ellipsoid $E_{\Sigma}$ induced by $\Sigma$. Hence the minimizer $K_*$ is the ellipsoid
\begin{align*}
   K_* = \kappa_d^{1/d}\det(\Sigma)^{-1/2d} E_{\Sigma},
\end{align*}
where $\|u\|_{E_{\Sigma}} = \|\Sigma^{-1/2} u\|_{\ell_2}$ and $\kappa_d = \vol_d(B^d)$.
 \end{itemize}
\end{ex}

\begin{ex}[Gaussian mixture model] \label{ex:gmm-sol}
Consider the following Gaussian mixture model $P_{\varepsilon} = \frac{1}{2} \Ncal(0, \Sigma_{\varepsilon,1}) + \frac{1}{2} \Ncal(0,\Sigma_{\varepsilon,2})$ where $\Sigma_{\varepsilon,1} := [1, 0 ; 0 ,\varepsilon] \in \R^{2 \times 2}$ and $\Sigma_{\varepsilon,2} := [\varepsilon, 0 ; 0 , 1] \in \R^{2 \times 2}$ for $0 < \varepsilon < 1$ with density $p_{\varepsilon}(x)$ in $\R^2$. Then, we have that the radial function of $L_{P_{\varepsilon}}$ is given by \begin{align*}
    \rho_{P_{\varepsilon}}(u)^3  & = \int_0^{\infty} r^2 p_{\varepsilon}(ru)\mathrm{d}u = c_{\varepsilon, 1}\|\Sigma_{\varepsilon,1}^{-1/2}u\|_{\ell_2}^{-3} + c_{\varepsilon, 2} \|\Sigma_{\varepsilon,2}^{-1/2}u\|_{\ell_2}^{-3}
\end{align*} where $c_{\varepsilon, i} = \frac{1}{2}\det(2\pi\Sigma_{\varepsilon,i})^{-1/2}\left(\int_0^{\infty} t^2 e^{-t^2/2}\mathrm{d} t \right)$ for $i = 1,2$. The set $L_{P_{\varepsilon}}$ with radial function $\rho_{P_{\varepsilon}}$ is similar to a radial sum of two ellipsoids induced by $\Sigma_{\varepsilon,1}$ and $\Sigma_{\varepsilon,2}$, which is discussed in the next example. In Figure \ref{fig:gauss-mix-sol}, we plot the solutions $L_{P_{\varepsilon}}$ for different values of $\varepsilon$. We see in the plots that for small $\varepsilon$, $L_{P_{\varepsilon}}$ will be a nonconvex star body, but as $\varepsilon$ approaches 1, $L_{P_{\varepsilon}}$ approaches a convex $\ell_2$ ball. One can observe that $1/\rho_{P_{\varepsilon}}(x)$ will be convex for all $\varepsilon$ such that $\varepsilon \in (\varepsilon^*_c, 1]$ for some $\varepsilon^*_c \in (0,1)$. By Corollary \ref{cor:convergence-of-minimizers-conv}, the optimal regularizer will be convex for this range of the parameter $\varepsilon$. \begin{figure}
    \centering
    \includegraphics[width=.27\textwidth]{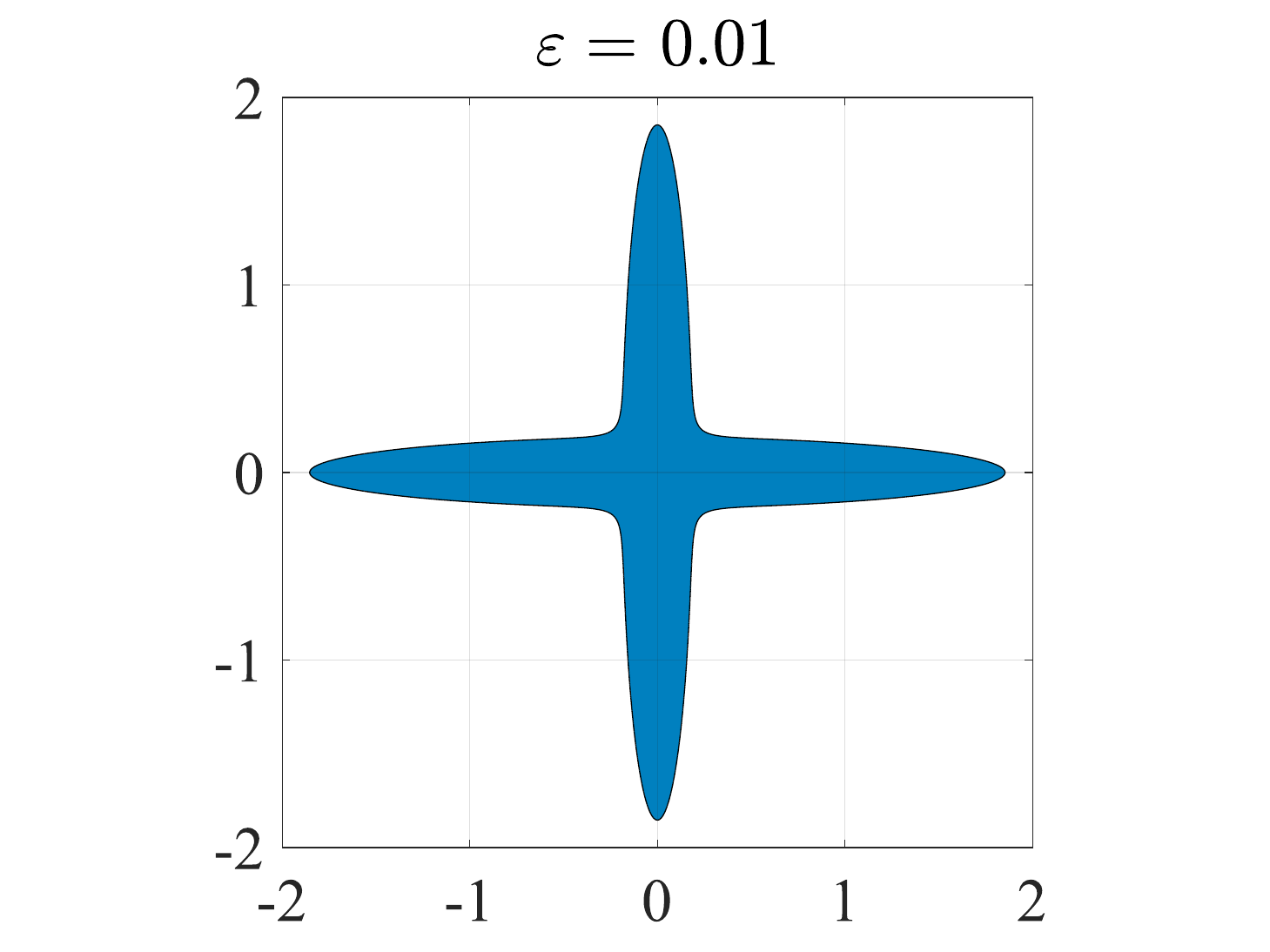}\hspace{-.75cm}
    \includegraphics[width=.27\textwidth]{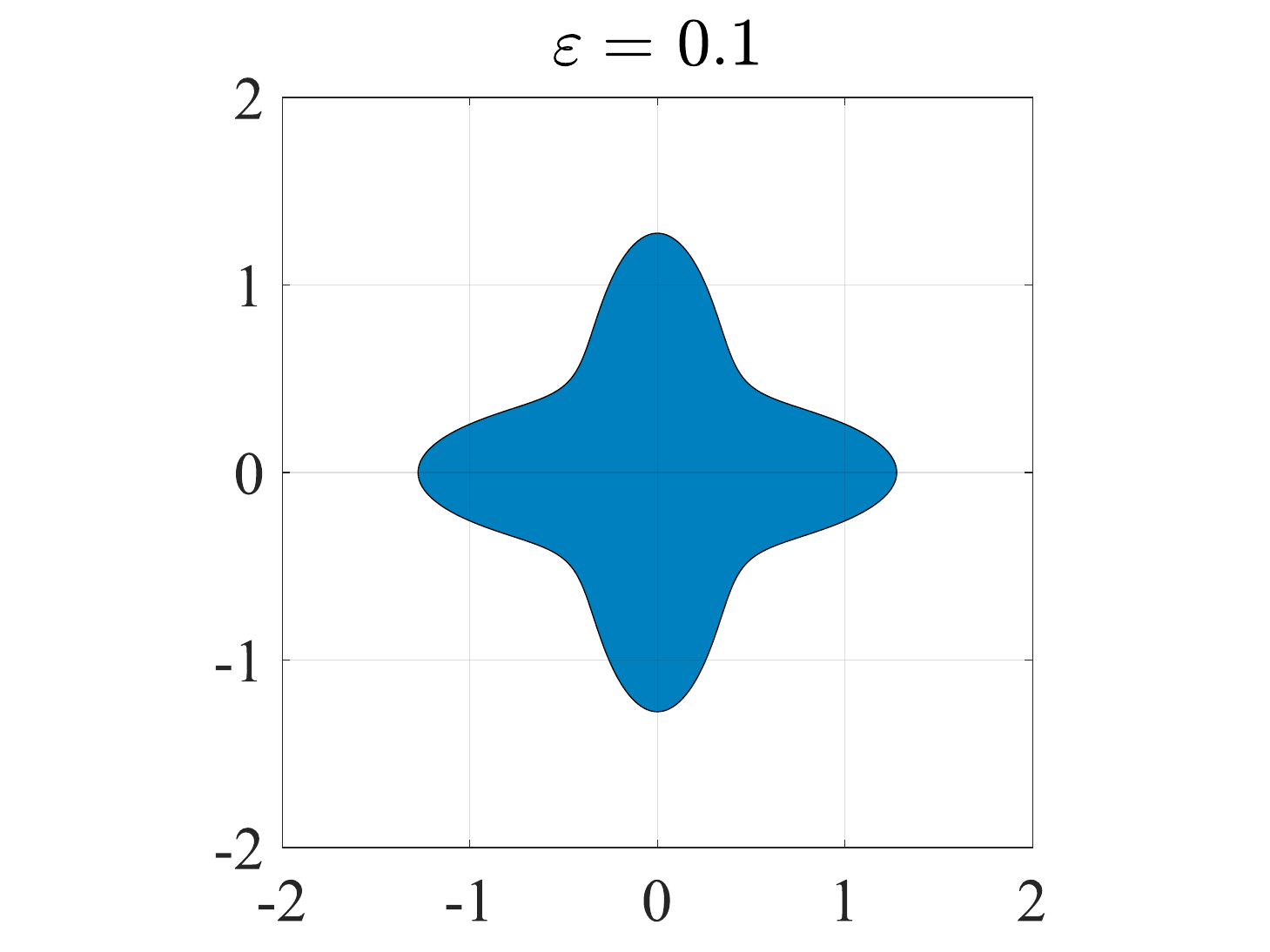}\hspace{-.75cm}
    \includegraphics[width=.27\textwidth]{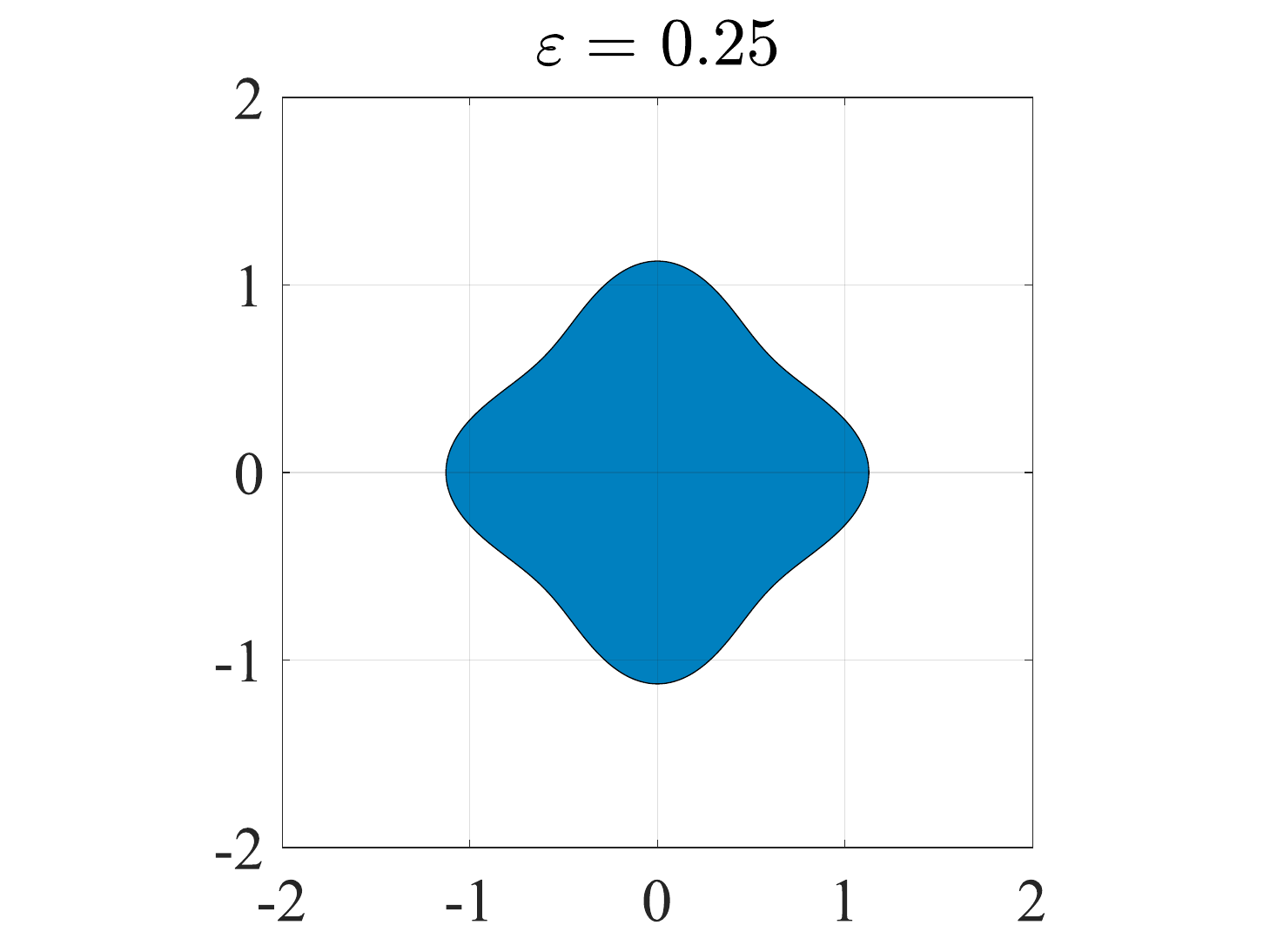}\hspace{-.75cm}
    \includegraphics[width=.27\textwidth]{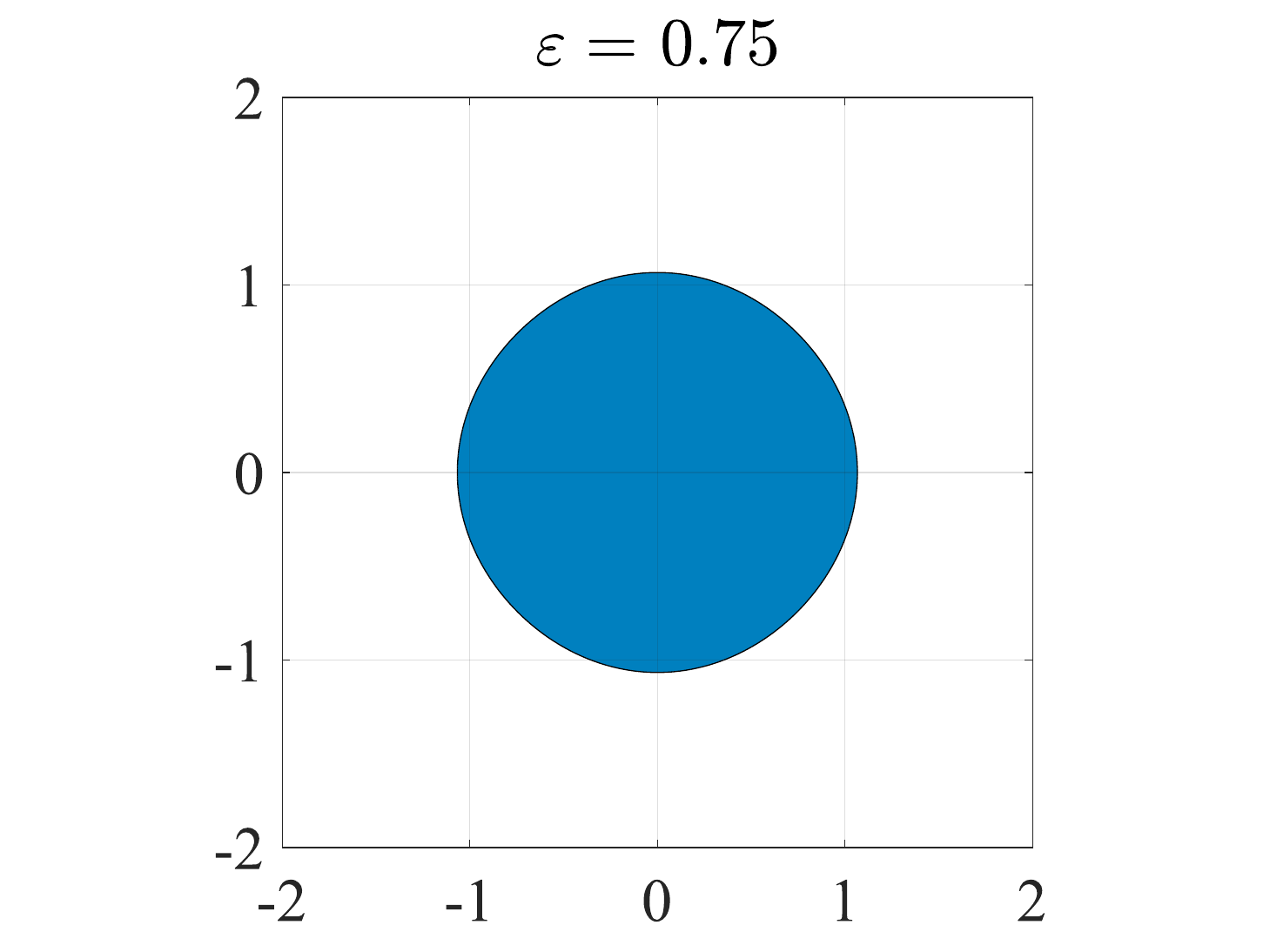}

    \caption{We compute the (unnormalized) optimal star body $L_{P_{\varepsilon}}$ based on the Gaussian mixture model in Example \ref{ex:gmm-sol} for $\varepsilon \in \{0.01, 0.1, 0.25, 0.75\}$. We note that as $\varepsilon$ decreases, the individual components of the Gaussian mixture model concentrate more heavily on a lower-dimensional subspace of $\R^2$. Similarly, the resulting star body also concentrates closer to the relevant subspaces. Finally, we see that as $\varepsilon$ increases, there is a critical value between at which $L_{P_{\varepsilon}}$ becomes convex. Numerically, this value is between $0.3$ and $0.45$. 
    }
    \label{fig:gauss-mix-sol}
\end{figure}

\end{ex}

\begin{ex}[General mixture distributions]
 More generally, suppose that our distribution $P$ is a mixture, i.e., let $P = \sum_{i=1}^m w_iP_i$ for weights $w_i > 0$ and distributions $P_i$ in $\R^d$. Then the associated convex body $K_*$ is a \emph{harmonic Blaschke linear combination}, as defined in \cite{Lutwak1990}, of star bodies $L_{P_i}$. Each $L_{P_i}$ is the associated star body with radial function $\rho_{P_i}$ in \eqref{e:rhoP}. The harmonic Blaschke linear combination $\hat{+}$ between star bodies defines a new star body $K \, \hat{+} \, L$ by adding the radial functions of $K$ and $L$ together in the following way:
\[\frac{\rho_{K \hat{+} L}(u)^{d+1}}{\vol_d(K \hat{+} L)} = \frac{\rho_K(u)^{d+1}}{\vol_d(K)} + \frac{\rho_L(u)^{d+1}}{\vol_d(L)}.\]

\end{ex}

{ 
\subsection{Extension to Positively Homogeneous Regularizers of General Degree} \label{sec:pos-hom-general-degree}

We conclude these results with an extension of our theory to the following more general class of regularizers that are positively homogeneous of degree $\alpha > 0$: $$x \mapsto \|x\|_K^{\alpha}, \qquad K \in \Scal^d.$$ The motivation for considering such regularizers is that certain powers of $\alpha$ are commonly encountered in practice. For example, popular Gaussian priors $P = \Ncal(0,\Sigma)$ induce regularizers of the form $\|\Sigma^{-1/2}x\|_{\ell_2}^2$, which are squared gauges induced by ellipsoids. Moreover, different powers of $\alpha$ can aid in obtaining certain properties beneficial for optimization, such as differentiability. Hence, we consider whether one can analogously characterize the solution to the following $\alpha$-homogeneous variant to problem \eqref{e:opt_1}: \begin{align}\label{e:alpha-hom-problem}
\argmin_{K \in \mathcal{S}^d: \mathrm{vol}_d(K) = 1}    \E_P[\|x\|_K^{\alpha}]. 
\end{align} The following Theorem shows that, by exploiting a more general dual mixed volume inequality of \cite{Lutwak1975}, we can obtain a general characterization of optimal $\alpha$-homogeneous star body regularizers, recovering the results of Theorem \ref{t:starbody_min} when $\alpha = 1$:

\begin{theorem} \label{thm:alpha-hom-general-result}
    Let $P$ be a distribution on $\R^d$ with density $p$ with respect to the Lebesgue measure, and assume $\E_P[\|x\|_{\ell_2}] < \infty$. Fix $\alpha > 0$. Suppose that the following function $\rho_{P,\alpha}$ is positive and continuous over the unit sphere: \begin{align}
        \rho_{P,\alpha}(u) := \left(\int_0^{\infty} r^{d+\alpha-1}p(ru)\mathrm{d}r\right)^{\frac{1}{d+\alpha}}, \qquad u \in \mathbb{S}^{d-1}. \label{eq:rho_P_alpha}
    \end{align} Then there exists a unique star body $L_{P,\alpha} \in \Scal^d$ whose radial function is $\rho_{P,\alpha}$, and the set $K_{*,\alpha} := \vol_d(L_{P,\alpha})^{-1/d}L_{P,\alpha}$ is the unique solution the minimization problem \eqref{e:alpha-hom-problem}.
\end{theorem}
\begin{proof}[Proof of Theorem \ref{thm:alpha-hom-general-result}]
    Analogously to the proof of Theorem \ref{t:starbody_min}, it is straightforward to show the existence of a star body $L_{P,\alpha}$ with radial function $\rho_{P,\alpha}$. We now establish that the objective in \eqref{e:alpha-hom-problem} can be written as a dual mixed volume with $i = -\alpha$ (see equation \eqref{eq:dual-mixed-vol-definition}). In particular, using positive homogenity of the gauge and integrating in spherical coordinates yields \begin{align*}
    \E_P[\|x\|_K^{\alpha}] & = \int_{\R^d} \|x\|_K^{\alpha} p(x) \mathrm{d}x  = \int_{\mathbb{S}^{d-1}} \|u\|_K^{\alpha} \left(\int_0^{\infty} r^{d+\alpha-1}p(ru) \mathrm{d}r\right) \mathrm{d} u \\
    & = \int_{\mathbb{S}^{d-1}} \rho_K(u)^{-\alpha} \rho_{P,\alpha}(u)^{d+\alpha} \mathrm{d}u = d \tilde{V}_{-\alpha}(K,L_{P,\alpha}).
\end{align*} Theorem 2 in \cite{Lutwak1975} states that for any two star bodies $K,L \in \Scal^d$ and any $i,j,k \in \R$ such that $i < j < k$, $$\tilde{V}_{j}^{k-i}(K,L) \leqslant \tilde{V}_i^{k-j}(K,L) \tilde{V}_{k}^{j-i}(K,L),$$ with equality if and only if $K$ and $L$ are dilates. Setting $i = -\alpha,  j = 0,$ and $k = d$ gives $$\E_P[\|x\|_K^{\alpha}] = d\tilde{V}_{-\alpha}(K,L_{P,\alpha}) \geqslant d\vol_d(L_{P,\alpha})^{(d+\alpha)/d}\vol_d(K)^{-\alpha/d}$$ with equality if and only if $K$ is a dilate of $L_P^{\alpha}.$ Thus, $\E_P[\|x\|_K^{\alpha}]$ is minimized over the collection $\{K \in \Scal^d : \vol_d(K) = 1\}$ by $K_{*,\alpha} := \vol_d(L_{P,\alpha})^{-1/d}L_{P,\alpha}.$
\end{proof}

\paragraph{Convexity of the optimal $\alpha$-homogeneous regularizer.} Similar to the results in Section \ref{sec:opt_reg_convex}, we can also explicitly characterize when the optimal $\alpha$-homogeneous regularizer will be convex. The following result is a direction consequence of the previous Theorem:

\begin{cor}
    Let $P$ be a probability measure on $\R^d$. Suppose $P$ satisfies the assumptions of Theorem \ref{thm:alpha-hom-general-result} for some $\alpha > 0$ and define $\rho_{P,\alpha}$ as in equation \eqref{eq:rho_P_alpha}. If $x \mapsto 1/\rho_{P,\alpha}(x)^{\alpha}$ is a convex function on $\R^d$, then the optimal $\alpha$-homogeneous regularizer $\|\cdot\|^{\alpha}_{K_{*,\alpha}}$ is convex.
\end{cor}

\paragraph{Optimal regularizer's dependence on $\alpha$.} In general, the geometry of the optimal regularizer's dependence on $\alpha$ can be complicated. For a broad family of distributions, however, there is a clearer dependence on $\alpha$. Suppose $P$ is a distribution with density $p(x) = \psi(\|x\|_L)$ for some function $\psi(\cdot)$ such that $\int_0^{\infty} r^{d+\alpha-1} \psi(r)\mathrm{d}r < \infty$ and star body $L$ with unit volume. Then, using a simple change-of-variables, we see that the radial function $\rho_{P,\alpha}$ in equation \eqref{eq:rho_P_alpha} satisfies the following: there exists a constant $c_{\psi,\alpha}$ that depends on $d,\psi$, and $\alpha$ such that for any $u \in \mathbb{S}^{d-1}$, \begin{align*}
    \rho_{P,\alpha}(u) = \left(\int_0^{\infty} r^{d+\alpha - 1}\psi(r\|u\|_L) \mathrm{d}r\right)^{\frac{1}{d+\alpha}} = \left(\|u\|_L^{-(d+\alpha)}\int_0^{\infty} r^{d+\alpha - 1}\psi(r) \mathrm{d}r\right)^{\frac{1}{d+\alpha}}= c_{\psi,\alpha} \|u\|_L^{-1} = c_{\psi,\alpha} \rho_L(u).
\end{align*} The power $\alpha$ only governs the value of $c_{\psi,\alpha}$ that scales $L$, and consequently the optimal unit-volume star body is simply $L$.  Hence, we obtain the conclusion that the optimal star body is $L$, regardless of the value of $\alpha > 0$, guaranteeing that the optimal regularizer is then $\|x\|_L^\alpha$. 
}


{ 
\paragraph{Remark on compactly supported measures.} Our results thus far demonstrate how one can identify an optimal positively homogeneous regularizer that promotes the structure in a given data distribution $P$. In some cases, however, the data distribution may be compactly supported and it is more natural to enforce this bounded support in downstream tasks via a constrained formulation rather than a regularized one, where one minimizes a data fidelity term subject to a constraint that the solution lies in a compact domain. It is then of interest in this setting to identify the optimal constraint set that promotes the structure in a given (compactly supported) distribution $P$. In analogy with our previous development, we consider the following set of distributions that are compactly supported:

$$\Dcal := \left\{p_K : \R^d \rightarrow \R\ |\ p_K(x) = \mathbf{1}_{\{x \in K\}}/ \mathrm{vol}_d(K),\ \mathrm{supp}(P) \subseteq K\ \text{is a star body}\right\}.$$ Identifying the best element of this class yields a constraint set that can be employed in a constrained formulation to solve downstream tasks. As before, we consider the following information-projection / maximum-likelihood approach to obtain the best element: \begin{align*}
    \argmin_{p_K \in \mathcal{D}} D_{\mathrm{KL}}(P || p_K)
\end{align*} Let $p(x)$ denote the density of $P$ on $\mathrm{supp}(P)$. Then, a direct calculation shows there exists a constant $C_P$ independent of $K$ such that for any $p_K \in \mathcal{D}$, \begin{align*}
    D_{\mathrm{KL}}(P || p_K) = \int_{\mathrm{supp}(P)} p(x)\log\left(\frac{p(x)}{p_K(x)}\right)\mathrm{d}x 
    & = C_P + \log \mathrm{vol}_d(K).
\end{align*} Hence the optimal uniform distribution $p_K$ will be the one such that $K$ has minimal volume: \begin{align*}
    \argmin_{p_K \in \mathcal{D}} D_{\mathrm{KL}}(P || p_K) = \argmin_{p_K \in \mathcal{D}} \mathrm{vol}_d(K).
\end{align*} When $\mathrm{supp}(P)$ is a star body, by monotonicity of the volume functional, the optimal solution will be $p_{\mathrm{supp}(P)}$. This shows that the optimal regularizer would be the projection onto the support of the distribution. When $\mathrm{supp}(P)$ is not a star body, the projection instead will be onto a ``star hull'' of $\mathrm{supp}(P)$ of minimal volume.

}






\section{Existence of Optimal Regularizers for General Data Distributions} \label{sec:general-existence-results}

In the preceding section, we characterized a unique optimal regularizer under certain conditions on the distribution $P$. For general $P$, it remains an open problem to characterize a unique star body that minimizes \eqref{e:opt_1}. However, we can prove existence of a solution to this variational problem for general $P$ by restricting the hypothesis class to contain only \emph{well-conditioned} star bodies, defined for any $r > 0$ by:
$$\Scal^d(r) := \{K \in \Scal^d : r B^d\subseteq \ker(K)\}.$$
Our results on the existence of minimizers of \eqref{e:opt_1} rely on continuity properties of the objective and compactness of the constraint set contained in $\mathcal{S}^d(r)$. First, in Section \ref{sec:cont-properties} we will show that the objective is Lipschitz continuous with respect to the radial metric $\delta(\cdot,\cdot)$. Then, in Section \ref{sec:compact_constraint} we show that our constraint sets of interest are compact by establishing Blaschke's Selection Theorem for bounded subsets of $\Scal^d(r)$ with respect to the radial metric.  
We will then state and prove the existence results in Section \ref{sec:minimizers-existence-proof}.

\subsection{Continuity Properties of the Objective Functional} \label{sec:cont-properties}
We begin by establishing continuity properties of our objective with respect to the radial metric over the collection $\Scal^d(r)$. When our sets are convex and bounded, we can also establish a Lipschitz bound with respect to the Hausdorff metric.
\begin{proposition}\label{prop:lip-cont-radial}
Suppose $K,L \in \Scal^d(r)$. Then for any $x,y \in \mathbb{S}^{d-1}$, we have \begin{align*}
    |\|x\|_K - \|y\|_L| \leqslant \frac{1}{r^2}\delta(K,L) + \frac{1}{r}\|x - y\|_{\ell_2}.
\end{align*} If, additionally, $K$ and $L$ are convex and $K,L \subseteq RB^d$, we have the following bound: \begin{align*}
    |\|x\|_K - \|y\|_L| \leqslant \frac{R}{r^3}d_H(K,L) + \frac{1}{r}\|x - y\|_{\ell_2}
\end{align*}
\end{proposition}
\begin{proof} We first note by the triangle inequality that \begin{align*}
    |\|x\|_K - \|y\|_L| & \leqslant |\|x\|_K - \|x\|_L| + |\|x\|_L - \|y\|_L|.
\end{align*} For the second term, note that since $rB^d \subseteq \ker(L)$, we have that by Proposition 6.2 in \cite{Rubinov2000} that $L = \bigcup_{u \in L} L_u$ where each $L_u$ is closed, convex and contains $rB^d$. Moreover, following the first part of the proof of Theorem 5.1 in \cite{Rubinov2000}, since $rB^d \subseteq L_u$ and each $L_u$ is convex, $\|\cdot\|_{L_u}$ is sublinear and Lipschitz with Lipschitz constant bounded by $1/r$. Finally, $\|\cdot\|_L$ coincides with $\inf_{u \in L} \|\cdot\|_{L_u}$ and since the infimum of a family $1/r$-Lipschitz functions is $1/r$-Lipschitz, it follows that $\|\cdot\|_L$ is $1/r$-Lipschitz.

We now focus on establishing a bound on $|\|x\|_K - \|x\|_L|$. First, note that since $r B^d \subseteq K$, we have that for any $\|x\|_{\ell_2} = 1$, $\rho_K(x) \geqslant \rho_{rB^d}(x) = r\rho_{B^d}(x) = r/\|x\|_{\ell_2} = r$. The same holds for $L$. Moreover, since $\|x\|_K = \rho_K(x)^{-1}$, we get \begin{align*}
    |\|x\|_K - \|x\|_L| \leqslant \frac{1}{\rho_K(x)\rho_L(x)}|\rho_K(x) - \rho_L(x)| \leqslant \frac{1}{r^2}|\rho_K(x) - \rho_L(x)|. 
\end{align*} Taking the supremum over $x \in \mathbb{S}^{d-1}$ yields the desired result.

For the convex case, we use the shorthand notation $\gamma := d_H(K,L)$.  By the definition of the Hausdorff metric, we have $K \subseteq L + \gamma B^d$ and $L \subseteq K + \gamma B^d$. Consider the ray generated by $x$.  We let $\tilde{x} := x / \| x \|_{K}$ and $\bar{x} := x/ \| x \|_{L}$ denote the intersection of the ray with the boundaries $\partial K$ and $\partial L$.  Our immediate goal is to show that $\|\bar{x} - \tilde{x} \|_{\ell_2} \leqslant \gamma R/r$. This will complete the proof.

First, we show that $1/\|x\|_L \leqslant 1/\|x\|_K + \gamma R/r$.  Suppose that $\bar{x} \notin K$. If $\bar{x} \in K$ then $1/\|x\|_L \leqslant1/\|x\|_K$, in which case the statement we wish to prove immediately holds.  Let $u$ be the projection of $\bar{x}$ onto $K$.  Let $D$ be the cone generated by all lines through $\tilde{x}$ and all points in $r B^d$.  We claim that $u \notin \mathrm{int}(D)$.  Suppose on the contrary that $u \in \mathrm{int}(D) $.  Consider the line from $u$ to $x$ extended.  Using the fact that $u \in \mathrm{int}(D) $, the line intersects $\mathrm{int}(r B^d)$.  Pick such a point inside $\mathrm{int}(r B^d)$ and consider an open ball inside $\mathrm{int}(r B^d)$, which is also inside $K$.  By convexity of $K$, we conclude that there is an open ball containing $\tilde{x}$ also in $K$.  However, this contradicts the original assumption that $\tilde{x} \in \partial K$.  Therefore we conclude that $u \notin \mathrm{int}(D)$.

By combining a simple trigonometric argument (such as by similarity) and using the fact that $u \notin \mathrm{int}(D)$, we have $$\|\bar{x}-\tilde{x}\|_{\ell_2} \leqslant\|\tilde{x}\|_{\ell_2} \frac{\|\bar{x}-u\|_{\ell_2}}{r} \leqslant\gamma R/r.$$ By repeating the same argument but with $K$ and $L$ switched, we have $|1/\|x\|_K - 1/\|x\|_L| \leqslant\gamma R/r$. Since $|1/\|x\|_K - 1/\|x\|_L| = |\rho_K(x) - \rho_L(x)|$, this completes the proof.
\end{proof} As a corollary, an application of Jensen's inequality and Proposition \ref{prop:lip-cont-radial} shows that the population objective functional is Lipschitz continuous over star bodies with sufficiently large inner width and convex bodies with large inner and outer widths.
\begin{corollary} \label{cor:lipschitz-pop-loss}
    Let $P$ be a distribution over $\R^d$ such that $\E_P \|x\|_{\ell_2} < \infty$. Then for any $K,L \in \Scal^d(r)$, 
    we have \begin{align*}
        |\E_P[\|x\|_K] - \E_P[\|x\|_L]| \leqslant \frac{\E_P\|x\|_{\ell_2}}{r^2} \delta(K,L).
    \end{align*} If, additionally, $K$ and $L$ are convex and $K,L \subseteq RB^d$, then \begin{align*}
    |\E_P[\|x\|_K] - \E_P[\|x\|_L]| \leqslant \frac{R\cdot\E_P\|x\|_{\ell_2}}{r^3}d_H(K,L).
\end{align*}
\end{corollary}

\subsection{Compactness and Blaschke's Selection Theorem for Star Bodies}\label{sec:compact_constraint}

In this section we will establish that a constraint set $\Cfrak$ that is a closed and bounded subset of $\Scal^d(r)$ is compact. 
We will then prove that the set of volume normalized well-conditioned star bodies $$\Scal^d(r,1) := \{ K \in \Scal^d(r) : \vol_d(K) = 1\}$$ is compact by establishing that this constraint set is a closed and bounded subset of $\Scal^d(r)$.

To establish compactness of these constraint sets, we must show that sequences in $\frak{C}$ exhibit convergent subsequences with limit inside $\frak{C}$. In order to showcase such a result, we establish a version of Blaschke's Selection Theorem for a subclass of star bodies equipped with the radial metric. It has been shown that Blaschke's Selection Theorem holds for any bounded, infinite collection of star bodies \textit{with respect to the Hausdorff metric} $d_H$ \cite{Hirose1965}. However, since our Lipschitz condition only holds with respect to the radial metric, we need to establish a local compactness result that holds for this metric instead.

To prove such a result, we require the following result from \cite{Sojka13}, which establishes that on the space $\Scal^d(r)$, the radial and Hausdorff metrics are topologically equivalent to one another.
\begin{theorem}[Theorem 3.6 in \cite{Sojka13}] \label{thm:top-equiv}
For any $r > 0$, the radial metric $\delta$ and the Hausdorff metric $d_H$ are topologically equivalent to one another on $\Scal^d(r)$. That is, convergent sequences in $(\Scal^d(r), \delta)$ are the same as convergent sequences in $(\Scal^d(r), d_H)$.
\end{theorem} 

With this result in hand, we can then prove Blaschke's Selection Theorem also holds with respect to the radial metric, specifically over the set $\Scal^d(r)$:
\begin{theorem}[Blaschke's Selection Theorem for the Radial Metric] \label{thm:BST-radial}
Fix $0 < r < \infty$ and let $\Cfrak$ be a bounded and closed subset of $\Scal^d(r)$. Let $(K_i)$ be a sequence of star bodies in $\Cfrak$. Then $(K_i)$ has a subsequence that converges in the radial and Hausdorff metric to a star body $K \in \Cfrak$.
\end{theorem}
\begin{proof}
    From \cite{Hirose1965}, since $(K_i)$ is a sequence of star bodies contained in a ball, we know that there exists a subsequence, call it $(K_{i_n})$, and a star body $K$ such that $d_H(K_{i_n},K) \rightarrow 0$ as $i_n \rightarrow \infty$. Since $\Cfrak$ is closed, $K \in \Cfrak$. Finally, to complete the proof, we note that since $d_H$ and $\delta$ are topologically equivalent on $\Scal^d(r)$ via Theorem \ref{thm:top-equiv}, we have that convergence in $d_H$ is equivalent to convergence in $\delta$. Hence the subsequence $(K_{i_n})$ of $(K_i)$ satisfies $\delta(K_{i_n}, K) \rightarrow 0$ as $i_n \rightarrow \infty$ where $K \in \Cfrak$.   
\end{proof} This result aids in establishing favorable compactness properties of our constraint set $\Cfrak$. We collect such properties as a corollary here:
\begin{corollary} \label{cor:compactness-eps-net-SrR}
For $0 < r < \infty$, let $\Cfrak$ be a closed and bounded subset of $\Scal^d(r)$. Then, the metric space $(\Cfrak, \delta)$ is compact and for every $\varepsilon > 0$, there exists a finite set of star bodies $\Scal_{\varepsilon} \subset \Cfrak$ such that $$\sup_{K \in \Cfrak} \inf_{L \in \Scal_{\varepsilon}} \delta(K,L) \leqslant \varepsilon.$$ The same conclusion holds with respect to the Hausdorff metric $d_H$.
\end{corollary} 
\begin{proof}
By Theorem \ref{thm:BST-radial}, we have that the space $(\Cfrak, \delta)$ is sequentially compact, as every sequence has a convergent subsequence with limit in $\Cfrak$. On metric spaces, sequential compactness is equivalent to compactness. Since the space is compact, it is totally bounded, thus guaranteeing the existence of a finite $\varepsilon$-net for every $\varepsilon > 0$ as desired. Theorem \ref{thm:top-equiv} guarantees the extension to the Hausdorff metric holds.
\end{proof}
We now consider the particular case when $\frak{C} = \Scal^d(r,1)$.  In order to establish compactness using Corollary \ref{cor:compactness-eps-net-SrR}, we need to establish that this subset is closed and bounded. We first show that the volume constraint in $\Scal^d(r,1)$ is sufficient to establish that the set is bounded.
\begin{lemma} \label{lem:unif-bound}
   For any $r > 0$, the collection $\Scal^d(r,1)$ is a bounded subset of $\Scal^d(r)$. In particular, for $R_r := \frac{d+1}{r^{d-1}\kappa_{d-1}}$ where $\kappa_{d-1} := \vol_{d-1}(B^{d-1})$, $$\Scal^d(r,1)  \subseteq \{K \in \Scal^d(r) : K \subseteq R_rB^d\}.$$
\end{lemma}
\begin{proof}
Let $K \in \Scal^d(r,1)$ and define $R(K) := \min\{R \geqslant 0 : K \subseteq RB^d\}$. For each $K$, $R(K) < \infty$, however we need to obtain a uniform bound $R_r$ such that $R(K) \leqslant R_r$ for all $K \in \Scal^d(r, 1)$. Now, let $x \in \partial K$ be such that $\|x\|_{\ell_2} = R(K)$, and consider the cone $C(K) := \conv(\{r B^{d-1} \cap x^{\perp} , x\})$ where $\conv(\cdot)$ denotes the convex hull. Since $rB^{d-1} \subseteq \mathrm{ker}(K)$, $C(K) \subseteq K$. By the volume constraint $\vol_d(K) = 1$, we have that
$\vol_d(C(K)) \leqslant \vol_d(K) = 1$.  This then implies that $$\vol_d(C(K)) = R(K) \cdot \frac{\vol_{d-1}(rB^{d-1}) )}{d+1} \leqslant 1.$$ Since this bound holds for any $K \in \Scal^d(r, 1)$, we conclude by observing the uniform bound \begin{align*}
  \sup_{K \in \Scal^d(r, 1)} R(K) \leqslant \frac{d+1}{r^{d-1}\vol_{d-1}(B^{d-1})} =: R_r.
  \end{align*}
\end{proof}
\noindent Lastly we show that the constraint set $\Scal^d(r,1)$ is closed under the Hausdorff topology. 
\begin{lemma} \label{lem:Scald-r-R-closed}
For any $0 < r < \infty$, the set $\Scal^d(r,1)$ is closed under the Hausdorff topology.
\end{lemma}
\begin{proof}
Let $(L_i)$ be a sequence in $\Scal^d(r,1)$ such that $L_i \rightarrow L$ in the Hausdorff metric. We claim that $L \in \Scal^d(r,1)$. We first show $r B^d \subseteq \ker(L)$. Note that $r B^d \subseteq \ker(L_{i})$ for each $i$ so that $\ker(L_i)$ has non-empty interior. Moreover, since each $L_i$ is closed, so is $\ker(L_i)$. Lastly, the sequence $(L_i)$ is bounded uniformly by Lemma \ref{lem:unif-bound}, so the sequence $(\ker(L_i))$ is as well. Thus, by Blaschke's Selection Theorem for convex bodies (Theorem \ref{thm:Blaschke}), there exists a subsequence, call it $(\ker(L_{i_{n}}))$, that is convergent in the Hausdorff metric to some convex body $\tilde{L}$. We claim that $\tilde{L} \subseteq \ker(L)$. Fix $v \in \tilde{L}$ and let $x \in L$ be arbitrary. We will show that $[v,x] \subseteq L$. Since $\ker(L_{i_n}) \rightarrow \tilde{L}$ and $L_i \rightarrow L$ in the Hausdorff metric, there exists sequences $(v_{i_n})$ and $(x_{i_n})$ such that $v_{i_n} \in \ker(L_{i_n})$ and $x_{i_n} \in L_{i_n}$ for each $i_n$ and $v_{i_n} \rightarrow v$ and $x_{i_n} \rightarrow x$. Then the line segments $[v_{i_n}, x_{i_n}] \rightarrow [v,x]$ in the Hausdorff metric where $[v_{i_n}, x_{i_n}] \subseteq L_{i_n}$ since each $v_{i_n} \in \ker(L_{i_n})$. Hence we must have that $[v,x] \subseteq L$. Since $x \in L$ was arbitrary, we must have that $v \in \ker(L)$ as desired so $\tilde{L} \subseteq \ker(L)$. Thus, since $r B^d\subseteq \ker(L_{i_n})$ for all $i_n$ and $\ker(L_{i_n}) \rightarrow \tilde{L}$, we obtain $r B^d \subseteq \tilde{L} \subseteq \ker(L)$ as claimed. We finally note that since each $L_i$ satisfies $\vol_d(L_i) = 1$, continuity of the volume functional gives $\vol_d(L) =1$. We conclude that $L \in \Scal^d(r,1).$ 
\end{proof}

\subsection{Existence of Minimizers}\label{sec:minimizers-existence-proof}

We now state and prove our main results on the existence of minimizers using the results above. 
\begin{theorem}\label{thm:star_existence}
Suppose $P$ is such that $\E_P[\|x\|_{\ell_2}] < \infty$. If $\Cfrak \subseteq \Scal^d(r)$ is closed and bounded, then there exists a solution to the following optimization problem:
\begin{align}\label{e:wellconditioned_opt_subset}
    \argmin_{K \in \Cfrak} \E_P[\|x\|_K].
\end{align}
In particular, for fixed $0 < r < \infty$, a solution to the following optimization problem exists:
\begin{align}\label{e:wellconditioned_opt_1}
    \argmin_{K \in \Scal^d(r): \mathrm{vol}_d(K) = 1} \E_P[\|x\|_K].
\end{align}
\end{theorem}
\begin{proof}
The result is a direct consequence of the fact our objective is continuous and the constraint set is compact. Indeed, using Corollary \ref{cor:lipschitz-pop-loss}, we have that the objective functional $K \mapsto \E_P[\|x\|_K]$ is $\frac{\E_P[\|x\|_{\ell_2}]}{r^2}$-Lipschitz over $\Cfrak$. Moreover, the constraint set $\Cfrak$ is compact by Corollary \ref{cor:compactness-eps-net-SrR}. As we are minimizing a continuous functional over a compact domain, the existence of minimizers is immediate. Combining Lemma \ref{lem:unif-bound} and Lemma \ref{lem:Scald-r-R-closed} shows that $\Scal^d(r,1)$ is a closed and bounded subset of $\Scal^d(r)$, meaning that the existence of a minimizer holds in this case as well.
\end{proof}

The existence of a minimizer to \eqref{e:wellconditioned_opt_subset} can be applied to the case when $\Cfrak$ is restricted to well-conditioned convex bodies defined by
\begin{align*}
    \Ccal^d(r) := \{K \in \Ccal^d :  r B^d \subseteq K\}.
\end{align*} We highlight this case because for computational and modeling considerations, one may wish to restrict to searching over convex regularizers. For instance, enforcing convexity can help ensure that the optimization framework to solve downstream tasks such as inverse problems will be more computationally tractable. There may also be further domain knowledge suggesting that convexity is the correct structure for the data distribution of interest. We will also consider in Sections \ref{sec:convergence-of-minimizers} and \ref{sec:stat-learning} empirical and population risk minimizers over natural subclasses of convex bodies, for which the following result ensures existence.
\begin{corollary}\label{cor:convex-minimzers}
Fix $0 < r < \infty$ and suppose $P$ is such that $\E_P[\|x\|_{\ell_2}] < \infty$. Suppose $\Cfrak \subseteq \Ccal^d(r)$ is closed and bounded. Then a minimizer to the optimization problem exists: \begin{align*}
       \min_{K \in \Cfrak} \E_P[\|x\|_K].
   \end{align*}
\end{corollary}
\begin{proof}Corollary \ref{cor:convex-minimzers} follows from Proposition \ref{thm:star_existence} because $\Ccal^d(r)$ is a closed subset of $\Scal^d(r)$, but it can also be proved directly by appealing to the corresponding results for convex bodies equipped with the Hausdorff metric.\end{proof}
At this point, it is natural to ask how Theorem \ref{thm:star_existence} relates to the results in Section \ref{sec:minimizers-star-bodies}. In particular, if $P$ satisfies the conditions of Theorem \ref{t:starbody_min}, when is the solution to \eqref{e:opt_1} equal to the solution of \eqref{e:wellconditioned_opt_1} for some $r > 0$? The next result answers this question using properties of the radial function $\rho_P$. 
\begin{proposition}
Let $P$ satisfy the conditions of Theorem \ref{t:starbody_min}, and define $\rho_P$ as in \eqref{e:rhoP}. If $x \mapsto 1/\rho_P(x)$ is Lipschitz, then there exists an $r > 0$ such that the unique minimizer of \eqref{e:opt_1} is also the unique minimizer of \eqref{e:wellconditioned_opt_1}.
\end{proposition}
\begin{proof}
It suffices to prove that there exists $r > 0$ such that $L_P \in \mathcal{S}^d(r)$. Indeed, Theorem 6.1 in \cite{Rubinov2000} implies that there exists $r > 0$ such that $rB^d \subseteq \mathrm{ker}(L_P)$ if and only if the gauge $\|\cdot\|_{L_P} = 1/\rho_P(\cdot)$ is Lipschitz.
\end{proof}

\section{Convergence of Minimizers}\label{sec:convergence-of-minimizers}

In this section, we study the behavior of the minimizers with increasing data.  Our main result is to show that the empirical minimizers converge to the minimizer of the population risk. To state our main result, we introduce the following notation. For a distribution $P$ over $\R^d$, let $P_m$ denote the empirical distribution of $P$ over $m$ i.i.d. observations drawn from $P$. Let $F(K ;P) := \mathbb{E}_P[\|x\|_K]$. Then $F(K; P_m)$ and $F(K;P)$ denote the empirical risk functional and population risk functional, respectively. As in the previous section, we will restrict to the class of well-conditioned star bodies $\Scal^d(r)$ and the bounded subset $\Scal^d(r,1)$ in order to obtain the required continuity and compactness properties associated to our optimization problem for the desired convergence to hold.

We first state the main result here: \begin{theorem} \label{thm:convergence-of-minimizers}
Let $P$ be a distribution over $\R^d$ such that  $\E_P \|x\|_{\ell_2} < \infty$ and fix $0 < r < \infty$. Then the sequence of minimizers $(K_m^*) \subseteq \Scal^d(r, 1)$ of $F(K; P_m)$ over $\Scal^d(r,1)$ has the property that any convergent subsequence converges in the radial and Hausdorff metric to a minimizer of the population risk almost surely: $$\text{any convergent}\ (K_{m_{\ell}}^*)\ \text{satisfies}\ K_{m_{\ell}}^* \rightarrow K_* \in \argmin_{K \in \Scal^d(r,1)} F(K; P).$$ Moreover, a convergent subsequence of $(K_m^*)$ exists.
\end{theorem} \noindent Note that our result states that any limit point of $(K_m^*)$ is a minimizer of the population risk, which is vacuously true if $(K_m^*)$ does not have a convergent subsequence. However, since we have established $\Scal^d(r,1)$ is a closed and bounded subset of $\Scal^d(r)$ and $(K_m^*) \subseteq \Scal^d(r,1)$, an application of Theorem \ref{thm:BST-radial} proves that a convergent subsequence of $(K_m^*)$ must exist.

If we restrict ourselves to searching over convex bodies $\Ccal^d(r,1) := \{K \in \Ccal^d : rB^d \subseteq K, \vol_d(K) = 1\}$, we can obtain the following result focusing on convex bodies, whose proof is a direct consequence of Theorem \ref{thm:convergence-of-minimizers} and the fact that the radial topology and Hausdorff topology are equivalent for convex bodies:

\begin{theorem} \label{cor:convergence-of-minimizers-conv}
Let $P$ be a distribution over $\R^d$ such that  $\E_P \|x\|_{\ell_2} < \infty.$ Then the sequence of minimizers $(K_m^*) \subseteq \Ccal^d(r, 1)$ of $F(K; P_m)$ over $\Ccal^d(r, 1)$ has the property that any convergent subsequence converges in the radial and Hausdorff metric to a minimizer of the population risk almost surely: $$\text{any convergent}\ (K_{m_{\ell}}^*)\ \text{satisfies}\ K_{m_{\ell}}^* \rightarrow K_* \in \argmin_{K \in \Ccal^d(r, 1)} F(K; P).$$ Moreover, a convergent subsequence of $(K_m^*)$ exists.
\end{theorem}

An outline of this section is as follows. First, in Section \ref{sec:gamma-convergence-overview}, we describe the core mathematical tool needed to prove convergence of minimizers, namely the notion of $\Gamma$-convergence from variational analysis. Then, in Section \ref{sec:ULLN-for-objective}, we establish a result that will be helpful in establishing the requirements of $\Gamma$-convergence. In particular, we prove uniform convergence of the empirical risk objective to the population objective over $\Scal^d(r,1)$. This result requires our continuity and compactness results that were used to prove existence of minimizers in Section \ref{sec:general-existence-results}. Finally, combining results from both sections, we prove Theorem \ref{thm:convergence-of-minimizers} in Section \ref{sec:final-convergence-of-minimizers} 
 by arguing that the minimizers of the empirical risk must converge to a minimizer of the population risk. Such tools will also be shown to establish robustness guarantees in Section \ref{sec:robustness} for our problem, in the sense that if we obtain a regularizer on data convolved with Gaussian noise, the regularizer converges to the optimal regularizer for clean data as the amount of noise decreases.

\subsection{$\Gamma$-convergence} \label{sec:gamma-convergence-overview}

Our key technique for establishing Theorem \ref{thm:convergence-of-minimizers} is a tool from variational analysis known as $\Gamma$-convergence \cite{Braides-Gamma-Handbook}.  In essence, $\Gamma$-convergence allows one to conclude the convergence of minimizers of a sequence of functionals to a minimizer of a particular limiting functional.  In addition, because our objects of interest are \emph{sets} and they reside in a more general metric space rather than Euclidean space, greater care is needed in understanding the limits of minimizers. 

We first state the definition of $\Gamma$-convergence here and cite some useful results that will be needed in our proofs. 

\begin{definition}
Let $(F_i)$ be a sequence of functions $F_i : X \rightarrow \R$ on some topological space $X$. Then we say that $F_i$ $\Gamma$-converges to a limit $F$ and write $F_i \xrightarrow[]{\Gamma} F$ if  the following conditions hold:
\begin{itemize}
    \item For any $x \in X$ and any sequence $(x_i)$ such that $x_i \rightarrow x$, we have \begin{align*}
        F(x) \leqslant \liminf_{i \rightarrow \infty} F_i(x_i).
    \end{align*}
    \item For any $x \in X$, we can find a sequence $x_i \rightarrow x$ such that \begin{align*}
        F(x) \geqslant \limsup_{i \rightarrow \infty} F_i(x_i).
    \end{align*}
\end{itemize} In fact, if the first condition holds, then the second condition could be taken to be the following: for any $x \in X$, there exists a sequence $x_i\rightarrow x$ such that $\lim_{i\rightarrow\infty} F_i(x_i) = F(x).$
\end{definition}

In addition to $\Gamma$-convergence, we also require the notion of equi-coercivity, which states that minimizers of a sequence of functions are attained over a compact domain.
\begin{definition}
A family of functions $F_i : X \rightarrow \R$ is equi-coercive if for all $\alpha$, there exists a compact set $K_{\alpha} \subseteq X$ such that $\{x \in X : F_i(x) \leqslant \alpha\} \subseteq K_{\alpha}.$
\end{definition}

Finally, this notion combined with $\Gamma$-convergence guarantees convergence of minimizers, which is known as the Fundamental Theorem of $\Gamma$-convergence \cite{Braides-Gamma-Handbook}.

\begin{proposition}[Fundamental Theorem of $\Gamma$-Convergence] \label{prop:fund-thm-gamma-conv}
If $F_i \xrightarrow[]{\Gamma} F$ and the family $(F_i)$ is equi-coercive, then the every limit point of the sequence of minimizers $(x_i)$ where $x_i \in \argmin_{x \in X} F_i(x)$ converges to some $x \in \argmin_{x \in X}F(x)$.
\end{proposition} 
 \noindent The goal of the subsequent sections will be in deriving the necessary conditions to establish $\Gamma$-convergence of the empirical risk to the population risk. Then, we will show how our compactness results in Section \ref{sec:compact_constraint} in conjunction with the results in Section \ref{sec:ULLN-for-objective} can establish the requirements of Proposition \ref{prop:fund-thm-gamma-conv}.

\subsection{Uniform Convergence of Empirical Risk Objective} \label{sec:ULLN-for-objective}

In order to show that our empirical risk functional $\Gamma$-converges to the population risk, we require a 
result that shows the empirical risk uniformly converges to the population risk in the limit of increasing data. To obtain this, we appeal to the following Uniform Strong Law of Large Numbers (ULLN) \cite{Pollard84}, which shows that if there exists an $\varepsilon$-net over our hypothesis class, then we can establish uniform convergence of the empirical risk to the population risk.

\begin{theorem}[Theorem 3 in \cite{Pollard84}] \label{thm:USLLN}
Let $Q$ be a probability measure, and let $Q_m$ be the corresponding empirical measure. Let $\mathcal{G}$ be a collection of $Q$-integrable functions. Suppose that for every $\varepsilon > 0$
there exists a finite collection of functions $\mathcal{G}_{\varepsilon}$ such that for every $g \in \mathcal{G}$ there exists $\overline{g}, \underline{g} \in \mathcal{G}_{\varepsilon}$ satisfying (i) $\underline{g}  \leqslant g \leqslant  \overline{g}$, and (ii) $\mathbb{E}_Q[\overline{g} - \underline{g}]< \varepsilon$. Then $\sup_{g \in \mathcal{G}}|\mathbb{E}_{Q_m}[g] - \mathbb{E}_Q[g]| \rightarrow 0$ almost surely.
\end{theorem}

We now state and prove the uniform convergence result. 

\begin{theorem} \label{thm:strong-LLN}
Fix $0 < r < \infty$ and let $P$ be a distribution on $\R^d$ such that $\E_{P}\|x\|_{\ell_2} < \infty$. Then, we have strong consistency in the sense that \begin{align*}
    \sup_{K \in \Scal^d(r,1)} | F(K; P_m) - F(K; P) | \rightarrow 0\ \text{as}\ m \rightarrow \infty\ \text{almost surely}.
\end{align*}
\end{theorem}

\begin{proof}[Proof of Theorem \ref{thm:strong-LLN}]

By Corollary \ref{cor:lipschitz-pop-loss}, we have that the map $K \mapsto \E_P[\|x\|_K]$ is $C M_P$-Lipschitz over $\Scal^d(r,1)$ where $C = C(r) := 1/r^2$ and $M_P := \E_P\|x\|_{\ell_2} < \infty$. Now, consider the set of functions $\mathcal{G} := \{\|\cdot\|_K : K \in \Scal^d(r,1)\}$. We will show that we can construct a finite set of functions that approximate $\|\cdot\|_K$ for any $K \in \Scal^d(r,1)$ via an $\varepsilon$-covering argument. This will allow us to apply Theorem \ref{thm:USLLN}. By Corollary \ref{cor:compactness-eps-net-SrR}, there exists an $\varepsilon$-net $\Scal_{\varepsilon} \subseteq \Scal^d(r,1)$ of finite cardinality. For fixed $\varepsilon > 0$, construct a $\eta$-cover $\Scal_{\eta}$ of $\Scal^d(r,1)$ such that $\sup_{K \in \Scal^d(r,1)}\min_{L \in \Scal_{\eta}}\delta(K,L) \leqslant \eta$ where $\eta \leqslant \varepsilon/(2CM_P)$. Define the following sets of functions: \begin{align*}
    \mathcal{G}_{\eta,-} & := \{(\|\cdot\|_K - CM_P \eta)_+ : K \in \Scal_{\eta}\}\ \text{and}\ \mathcal{G}_{\eta,+} := \{\|\cdot\|_K + CM_P \eta: K \in \Scal_{\eta}\}.
\end{align*} Let $\|\cdot\|_K \in \mathcal{G}$ be arbitrary. Let $K_0 \in \Scal_{\eta}$ be such that $\delta(K,K_0) \leqslant \eta$. Define $\overline{f} = \|\cdot\|_{K_0} + CM_P \eta \in \mathcal{G}_{\eta,+}$ and $\underline{f} = (\|\cdot\|_{K_0}- CM_P \eta)_+ \in \mathcal{G}_{\eta,-}$. It follows that $\underline{f} \leqslant f \leqslant \overline{f}$. Moreover, by our choice of $\eta$, we have that $$\mathbb{E}_P[\overline{f} - \underline{f}] \leqslant 2CM_P \eta \leqslant \varepsilon.$$ Thus, the conditions of Theorem \ref{thm:USLLN} are met and we get $$\sup_{K \in \Scal^d(r,1)}|F(K; P_m) - F(K; P)| = \sup_{K \in \Scal^d(r,1)}|\mathbb{E}_{P_m}[\|x\|_K] - \mathbb{E}_{P}[\|x\|_K| \rightarrow 0\ \text{as}\ m\rightarrow \infty\ \text{a.s.} $$
\end{proof}

\subsection{Proof of Theorem \ref{thm:convergence-of-minimizers}}
\label{sec:final-convergence-of-minimizers}

We first establish the two requirements of $\Gamma$-convergence. For the first, fix $K \in \Scal^d(r,1)$ and consider a sequence $K_m \rightarrow K$. Then we have that \begin{align}
   F(K; P) & = F(K; P) - F(K; P_m) + F(K; P_m) - F(K_m; P_m) + F(K_m; P_m) \nonumber\\
   & \leqslant |F(K; P) - F(K; P_m)| + |F(K; P_m) - F(K_m; P_m)| + F(K_m; P_m) \nonumber \\
   & \leqslant \sup_{K \in \Scal^d(r,1)} |F(K; P) - F(K; P_m)| + |F(K; P_m) - F(K_m; P_m)| + F(K_m; P_m) \label{eq:liminf-bound}.
\end{align} By Theorem \ref{thm:strong-LLN}, we have that the first term goes to $0$ as $m$ goes to $\infty$ almost surely. We now show that $|F(K; P_m) - F(K_m; P_m)| \rightarrow 0$ as $m \rightarrow \infty$ almost surely. To see this, observe that by Corollary \ref{cor:lipschitz-pop-loss} we have \begin{align*}
    |F(K; P_m) - F(K_m; P_m)| & \leqslant  \frac{\E_{P_m}[\|x\|_{\ell_2}]}{r^2}\delta(K,K_m).
\end{align*} Since $\E_{P_m}[\|x\|_{\ell_2}] \rightarrow \E_P[\|x\|_{\ell_2}] < \infty$ and $\delta(K,K_m) \rightarrow 0$ as $m \rightarrow \infty$ almost surely, we attain $|F(K; P_m) - F(K_m; P_m)| \rightarrow 0$. Thus, taking the limit inferior of both sides in equation \eqref{eq:liminf-bound} yields \begin{align*}
    F(K; P) \leqslant \liminf_{m \rightarrow\infty} F(K_m; P_m).
\end{align*} For the second requirement, we exhibit a realizing sequence so let $K \in \Scal^d(r,1)$ be arbitrary. By Corollary \ref{cor:compactness-eps-net-SrR}, for any $m \geqslant 1$, there exists a finite $\frac{1}{m}$-net $\Scal_{1/m}$ of $\Scal^d(r,1)$ in the radial metric $\delta$. Construct a sequence $(K_m) \subset \Scal^d(r,1)$ such that for each $m$, $K_m \in \Scal_{1/m}$ and satisfies $\delta(K_m, K) \leqslant 1/m$. Hence this sequence satisfies $K_m \rightarrow K$ in the radial metric and $K_m \in \Scal^d(r,1)$ for all $m \geqslant 1$. Hence we can apply Theorem \ref{thm:strong-LLN} to get \begin{align*}
    |F(K_m; P_m) - F(K; P)| \leqslant |F(K_m; P_m) - F(K; P_m)| + |F(K; P_m) - F(K; P)| \rightarrow 0\ \text{as}\ m \rightarrow \infty\ \text{a.s.}
\end{align*} so $\lim_{m\rightarrow\infty} F(K_m; P_m) = F(K; P)$.

Now, we show that $(F(\cdot ; P_m))$ is equi-coercive on $\Scal^d(r,1)$. In fact, this follows directly from Theorem \ref{thm:BST-radial}, the variant of Blaschke's Selection Theorem we proved for the radial metric. Thus equi-coerciveness of the family $(F(\cdot ; P_m))$ trivially holds over $\Scal^d(r,1)$. As a result, applying Proposition \ref{prop:fund-thm-gamma-conv} to the family $F(\cdot ; P_m) : \Scal^d(r,1) \rightarrow \R$, if we define the sequence of minimizers \begin{align*}
    K_m^* \in \argmin_{K \in \Scal^d(r,1)} F(K; P_m)
\end{align*} we have that any limit point of $K_m^*$ converges to some \begin{align*}
    K_* \in \argmin_{K \in \Scal^d(r,1)} F(K; P)
\end{align*} almost surely, as desired. The existence of a convergent subsequence of $(K_m^*)$ follows from $\Scal^d(r,1)$ being a closed and bounded subset of $\Scal^d(r)$ and Theorem \ref{thm:BST-radial}.

{ 
\subsection{Robustness Guarantees}

\label{sec:robustness}

Our results utilizing $\Gamma$-convergence can also be used to establish robustness guarantees on optimal regularizers for uncorrupted and noisy data. In particular, we are interested in deviations between regularizers $K$ that are optimal over uncorrupted data $y = x$ for $x \sim P$ versus regularizers $K_{\sigma}$ that are optimal over noisy data $y = x + \sigma z$ where $x \sim P$ and $z \sim Q$ for some noise distribution $Q$, such as the Standard Normal distribution $\Ncal(0,I_d)$. A natural question to ask is whether $K_{\sigma}$ converges to $K$ as the amount of noise $\sigma \rightarrow 0$.
To understand this question, we first consider characterizing the deviation in the population risk under an uncorrupted distribution and a noisy distribution. 
}

{\color{black}To state our results, we first require the following metric on the space of probability distributions. Let $T(P,Q)$ denote all joint distributions $\gamma$ for $(X, Y)$ that have marginals $P$ and $Q$. Define the $p$-Wasserstein distance between two probability distributions $P$ and $Q$ as \begin{align*}
    W_p(P,Q) := \left(\inf_{\gamma \in T(P,Q)} \E_{\gamma}\|X - Y\|_{\ell_2}^p\right)^{1/p},\ p \geqslant 1.
\end{align*} A particularly useful result for our purposes is in the case $p = 1$, as the $1$-Wasserstein distance has the following dual characterization \cite{Villani-OT}:  \begin{align*}
    W_1(P, Q) := \sup_{f \in \mathrm{Lip}(1)} \left|\E_P[f(y)] - \E_Q [f(x)]\right|.
\end{align*} Here, $\mathrm{Lip}(R)$ denotes the set of all $R$-Lipschitz functions from $\R^d$ to $\R$: $$\mathrm{Lip}(R) := \left\{f : \R^d \rightarrow \R : \sup_{x \neq y}\frac{|f(x) - f(y)|}{\|x-y\|_{\ell_2}} \leqslant R\right\}.$$ }

{  Given two probability distributions $P$ and $Q$, we define their convolution $P * Q$ as the distribution that satisfies $$y \sim P * Q \Longleftrightarrow y = x + z,\ x \sim P,\ z \sim Q.$$ For a parameter $\sigma \geqslant 0$ and distribution $Q$, let $Q_{\sigma}$ denote the distribution such that $y \sim Q_{\sigma}$ if and only if $y = \sigma z$, $z \sim Q$. We first characterize the deviation in population risk between $P$ and $P * Q_{\sigma}$ as a function of $\sigma$ and then aim to understand properties of star bodies learned on noisy data as the noise vanishes.} {\color{black} We first establish the following result on the population risk, which shows that we can uniformly bound the deviation in population risk as a function of $\sigma$ if the collection of star bodies have sufficiently large inner width:} { \begin{theorem} \label{thm:noisy_pop_risk}
Let $P$ and $Q$ be two independent probability distributions on $\R^d$ and fix $\sigma \geqslant 0$. Suppose there exists an $0 < r < \infty$ such that $\mathfrak{C}  \subseteq \Scal^d(r)$. Then we have that \begin{align*}
    \sup_{K \in \mathfrak{C}} |F(K; P) - F(K; P * Q_{\sigma})| \leqslant \sigma \cdot \frac{\E_{Q}[\|z\|_{\ell_2}]}{r}.
\end{align*}
\end{theorem}}

{\color{black}To prove this, we first consider the general question of how far $F(\cdot ; P)$ deviates from $F(\cdot ; Q)$ for two distributions $P$ and $Q$. We prove the following:
\begin{proposition} \label{prop:gen_result_on_pop_risk}
Let $P$ and $Q$ be two probability distributions on $\R^d$.  Suppose there exists an $0 < r < \infty$ such that $\mathfrak{C}  \subseteq \Scal^d(r)$. Then the following bound holds: \begin{align*}
   \sup_{K \in \Cfrak} |F(K; P) - F(K; Q)| & \leqslant \frac{1}{r} W_1(P,Q).
\end{align*}
\end{proposition}
\begin{proof}
Fix $K \in \mathfrak{C}$. It was proven in Proposition \ref{prop:lip-cont-radial} that if $r B^d \subseteq \ker(K)$, then the mapping $x \mapsto \|x\|_K$ is $1/r$-Lipschitz so that $\|\cdot\|_K \in \mathrm{Lip}(1/r)$. Thus, using the dual characterization of $W_1(P,Q)$, we obtain the bound \begin{align*}
    \left|F(K ; P) - F(K; Q)\right| & = \left|\E_P[ \|y\|_K] - \E_Q[\|x\|_K]\right|  \\
    & \leqslant  \sup_{f \in \mathrm{Lip}(1/r)} \left|\E_P[ f(y)] - \E_Q[ f(x)]\right| \\
    & \leqslant \frac{1}{r} W_1(P, Q).
\end{align*} Taking the supremum over all $K \in \mathfrak{C}$ completes the proof.
\end{proof}

Now we can prove the main Theorem.}
{ 
\begin{proof}[Proof of Theorem \ref{thm:noisy_pop_risk}]
By our previous proposition, we simply need to bound $W_1(P, P * Q_{\sigma})$. Note that since $P$ and $Q$ are independent, then the joint distribution $\gamma$ given by $(x, x + \sigma z)$ where $x \sim P$ and $z \sim Q$ is a valid coupling between $P$ and $P*Q_{\sigma}$. By sub-optimality, we can upper bound the Wasserstein distance as \begin{align*}
    W_1(P,P * Q_{\sigma}) \leqslant \E_{\gamma}[\|x - (x + \sigma z)\|_{\ell_2}] = \sigma\E_{Q}[\|z\|_{\ell_2}].
\end{align*} Combining the uniform bound from Proposition \ref{prop:gen_result_on_pop_risk} and the above bound completes the proof.
\end{proof}

Based on our quantitative convergence bound, we can prove the following:

\begin{theorem}
Fix $0 < r < \infty$ and let $P$ and $Q$ be two independent distributions on $\R^d$ such that $\max\left\{\E_P[\|x\|_{\ell_2}], \E_Q[\|z\|_{\ell_2}]\right\} < \infty$. Let $(\sigma_t)$ be a non-negative, decreasing sequence converging to $0$. Then we have that the sequence of minimizers on noisy data \begin{align*}
    K_t \in \argmin_{K \in \mathcal{S}^d(r,1)} F(K; P * Q_{\sigma_t})
\end{align*} has the property that any convergent subsequence converges to a minimizer on uncorrupted data in the radial metric and Hausdorff metric: \begin{align*}
   \text{any convergent}\ (K_{t_{\ell}})\ \text{satisfies}\ K_{t_{\ell}} \rightarrow K_* \in \argmin_{K \in \Scal^d(r,1)} F(K ; P).
\end{align*} Moreover, a convergent subsequence of $(K_t)$ exists.
\end{theorem}

\begin{proof} The proof of this result follows similarly to the proof of Theorem \ref{thm:convergence-of-minimizers} by using Theorem \ref{thm:noisy_pop_risk} to help establish $\Gamma$-convergence of the family $(F(\cdot ; P * Q_{\sigma_t}))_{t}$ to $F(\cdot; P)$ over $\Scal^d(r,1)$. For the two requirements of $\Gamma$-convergence, fix $K \in \Scal^d(r,1)$ and consider a sequence $K_t \rightarrow K$ as $t \rightarrow \infty$. Then by a similar argument to equation \eqref{eq:liminf-bound}, we have that \begin{align*}
    F(K; P) \leqslant \sup_{K \in \Scal^d(r,1)}|F(K;P) - F(K; P * Q_{\sigma_t})| + |F(K; P*Q_{\sigma_t}) - F(K_t; P*Q_{\sigma_t})| + F(K_t ; P*Q_{\sigma_t}).
\end{align*} By Theorem \ref{thm:noisy_pop_risk}, we have that the first term in the above display goes to $0$ as $t\rightarrow \infty$. Since $\E_{P*Q_{\sigma_t}}[\|x\|_{\ell_2} ]\rightarrow \E_P[\|x\|_{\ell_2}] < \infty$ and $\delta(K, K_t) \rightarrow 0$ as $t \rightarrow \infty$, then an application of Corollary \ref{cor:lipschitz-pop-loss} shows that the second term also goes to $0$ as $t \rightarrow \infty$. Taking the limit inferior of both sides gives us the desired result. The second requirement of $\Gamma$-convergence and equi-coercivity follows exactly the same argument as in the proof of Theorem \ref{thm:convergence-of-minimizers}.
\end{proof}}
{ 
\paragraph{Remark on noise assumptions.} We note that our assumptions on the noise only require essentially a finite second moment, which allows the result to handle a wide-range of distributions. For example, any Sub-gaussian distribution will satisfy the above requirements, along with distributions with heavier tails such as Sub-exponential distributions. While we focused on additive, scaled noise from the same distribution here for simplicity, it is also possible to extend such results to more complicated sequences of noise distributions, as long as tight control of the Wasserstein distance between the clean distribution and contaminated distribution can be shown.
}

\section{Statistical Learning Guarantees}

\label{sec:stat-learning}

As a final application of our results, we showcase sample complexity guarantees to control the deviation from the empirical risk to the population risk. In many data science contexts, it is imperative to understand the tradeoff between the expressiveness of the hypothesis class and the number of samples needed to efficiently solve the problem. Specifically, one might wish to ask, how many samples does it take to construct a regularizer (from a structured family) that approximates the best possible choice from a hypothesis class to some target accuracy -- such questions are often phrased as \emph{sample complexity} problems. In our framework, this requires an analysis of the particular collection of star bodies $\Cfrak \subseteq \Scal^d$ being optimized over and the behavior of the objective $K \mapsto \|\cdot\|_K$ over a finite number of samples drawn from a data distribution $P$. We showcase uniform convergence guarantees over the collection $\Cfrak$, where the difference between the empirical risk and the population risk is controlled with overwhelming probability by a quantity $\eta(m, \Cfrak, P)$ that depends on the number of samples $m$, the distribution $P$, and the complexity of the collection of bodies $\Cfrak$: \begin{align}
    \sup_{K \in \Cfrak} |F(K; P_m) - F(K; P)| \leqslant \eta(m, \Cfrak, P) \label{eq:unif-conv-bound-intro}
\end{align} As the number of samples increases, this approximation improves where the rate depends on the data distribution and structure of the bodies. As a consequence, we can establish that empirical minimizers $K_m^* \in \argmin_{K \in \Cfrak} F(K;P_m)$ obtain population risk not far from the population minimizer $K_* \in \argmin_{K \in \Cfrak} F(K; P)$ in the sense that with high probability \begin{align*}
    F(K_m^*; P) \leqslant F(K_*; P) + 2 \eta(m,\Cfrak, P).
\end{align*}

At a high-level, our generalization results require the following key ingredients to specify a uniform convergence bound:
\begin{itemize}
\item First, identify a family of star/convex bodies $\Cfrak$ of interest.
\item Show that the functional $K \mapsto \|\cdot\|_K$ is well-behaved in the sense that it is $C_{\Cfrak}$-Lipschitz over $\Cfrak$ with respect to a metric $D(\cdot,\cdot)$. This will allow us to exploit some of our previous continuity results for our objective functional over the space of well-conditioned star bodies.
\item Characterize the ``complexity'' of $\Cfrak$. While there are many notions of complexity that have been studied in statistical learning theory, one notion that will be particularly useful for our purposes is the notion of a covering number of a subset of a metric space. For a metric space $(X,D)$, we define the $\varepsilon$-covering number of a subset $T \subset X$ with respect to the metric $D$ as the smallest number of balls of radius $\varepsilon$ that covers $T$: \begin{align*}
    N(T,D,\varepsilon) := \min\left|\left\{\Tcal \subset T : T \subset \bigcup_{q \in \Tcal} \{x \in T : D(x,q) \leqslant \varepsilon\}\right\}\right|.
\end{align*}
\end{itemize}

In Section \ref{sec:general-uniform-bound-stat-learn}, we describe a general uniform convergence bound of the form \eqref{eq:unif-conv-bound-intro} that quantifies the rate of convergence of the empirical risk to the population risk over a general class of bodies $\Cfrak$. We will then showcase the utility of this bound by showing how it can be applied to several classes of bodies of interest in applications. First, we establish a bound on the class of star bodies given as unions of convex bodies with large inner width in Section \ref{sec:stat-learn-star-bodies-union-convex}. To our knowledge, such metric entropy estimates for this class of star bodies is new. Then, we apply this result to parametrized families of sets. The first class we consider is ellipsoids in Section \ref{sec:stat-learn-ellipsoids}. Afterwards, in Section \ref{sec:stat-learn-polytopes-DL}, we discuss polytopes and how such a bound can be applied to dictionary learning. Finally, we showcase an application to semidefinite regularizers in Section \ref{sec:stat-learn-semidef}, which are regularizers induced by linear images of Schatten $p$-norm balls.

\subsection{General Uniform Convergence Result} \label{sec:general-uniform-bound-stat-learn}

Given the ingredients described previously, we can prove a general uniform convergence result that shows when the functional $K \mapsto \|\cdot\|_K$ is Lipschitz over a collection of star bodies $\Cfrak$, we can obtain generalization bounds based on the complexity of the collection $\Cfrak$ and the data distribution $P$. We will show that this general theorem can then be applied to many classes of bodies of interest to showcase the utility of our result. 

\begin{theorem} \label{thm:unif-generalization-bound}
Let $P$ be a distribution over $\R^d$ such that for $x \sim P$, $\|x\|_{\ell_2} \leqslant 1$ almost surely. Consider a subset $\Cfrak \subset \Scal^d$ such that $r_{\Cfrak} B^d \subseteq K$ for all $K \in \Cfrak$ for some $0 < r_{\Cfrak} < \infty$. Suppose $\Cfrak$ has a finite $\varepsilon$-covering number in a metric $D$. That is, there exists a range $(0, \tilde{\varepsilon}]$ for some $\tilde{\varepsilon} > 0$ such that $\log N(\Cfrak, D, \varepsilon) < \infty$ for all $0 < \varepsilon \leqslant \tilde{\varepsilon}$. Suppose the map $K \mapsto \|\cdot\|_K$ is $C_{\Cfrak}$-Lipschitz over $\Cfrak$ in the maximum norm with respect to $D$, i.e., for all $K,L \in \Cfrak$ 
$$\|\|\cdot\|_K - \|\cdot\|_L\|_{\infty} := \max_{\|x\|_{\ell_2}=1}|\|x\|_K - \|x\|_L| \leqslant C_{\Cfrak} D(K,L).$$ 
Then for any $\gamma \in (0,1)$, we have that with probability $1 - \gamma$ over $m$ samples drawn from $P$, the following bound holds for all $K \in \Cfrak$, \begin{align*}
\begin{split}
    F(K; P) & \leqslant F(K; P_m) + 2\inf_{\varepsilon \in (0, \tilde{\varepsilon}]}\left\{\varepsilon + c_{P,m}\sqrt{\frac{2}{m} \log N\left(\Cfrak, D, \frac{\varepsilon}{C_{\Cfrak}}\right)}\right\} + c_{\Cfrak}\sqrt{\frac{2\ln 4/\gamma}{m}}.
\end{split}
\end{align*} Here $c_{P,m}$ is a constant that depends on $\Cfrak$, $P$, and $m$ and $c_{\Cfrak} := 4/r_{\Cfrak}$.
\end{theorem}

\begin{proof}
  Let $X_m := \{x_1,\dots,x_m\}$ denote the $m$ samples drawn from $P$ and define the function class \begin{align*}
      \Fcal_{\Cfrak} := \{\|\cdot \|_K : \R^d \rightarrow \R : K \in \Cfrak\}.
  \end{align*} Note that our class of functions $\Fcal_{\Cfrak}$ is uniformly bounded. In particular, by our assumption on the class $\Cfrak$, there exists an inner radius $0 < r_{\Cfrak} < \infty$ such that $r_{\Cfrak} B^d \subseteq K$ for all $K \in \Cfrak$. Hence for any $\|x\|_{\ell_2} \leqslant 1$, we have $\|x\|_K \leqslant 1/r_{\Cfrak}$. By standard statistical learning results \cite[Theorem 26.5]{shalev-shwartz-ben-david-ML}, the deviation between the empirical and population risk is controlled by the \textit{(empirical) Rademacher complexity} $\mathcal{R}(\Fcal_{\Cfrak}; X_m)$ \cite{Bartlett_etal05} of the function class $\Fcal_{\Cfrak}$. That is, it holds that with probability $1 - \gamma$, we have that for all $K \in \Cfrak$, \begin{align*}
    \E_{P}[\|x\|_K] \leqslant \E_{P_m}[\|x\|_K] + 2\mathcal{R}(\mathcal{F}_{\Cfrak}; X_m) + c_{\Cfrak}\sqrt{\frac{2\ln 4/\gamma}{m}}.
\end{align*} Here, $\Rcal(\Fcal; X_m)$ is defined by \begin{align*}
    \mathcal{R}(\mathcal{F} ; X_m) := \mathbb{E}_{\sigma}\left[\sup_{f \in \mathcal{F}} \frac{1}{m}\sum_{i = 1}^m\sigma_i f(x_i)\right]
\end{align*} where $\sigma_i$ are Rademacher random variables for $i \in [m]$ and $c_{\Cfrak} := 4/r_{\Cfrak}$. Moreover, one can show \cite[Proposition 5.2]{Rebeschini-notes} that the Rademacher complexity is bounded by the covering number of the function class $\Fcal_{\Cfrak}$ with respect to the pseudo-norm $\|\cdot\|_{1,X_m}$: \begin{align*}
    \mathcal{R}(\mathcal{F}_{\Cfrak}; X_m) \leqslant \inf_{\varepsilon \in (0, \tilde{\varepsilon}]}\left\{\varepsilon + c_{P,m}\sqrt{\frac{2}{m} \log N(\mathcal{F}_{\Cfrak}, \|\cdot\|_{1,X_m}, \varepsilon)}\right\}
\end{align*} where \begin{align*}
    c_{P,m} := \inf\left\{c > 0 : \sup_{K \in \Cfrak} \|\|\cdot\|_K\|_{1,X_m} \leqslant c\right\} \leqslant 1/r_{\Cfrak}.
\end{align*} Here, $\|f\|_{1,X_m} := \frac{1}{m}\sum_{i=1}^m|f(x_i)|$.

Now, we can use the fact that $K \mapsto \|\cdot\|_K$ is Lipschitz over $K \in \Cfrak$ to obtain a bound on $N(\mathcal{F}_{\Cfrak}, \|\cdot\|_{1,X_m}, \varepsilon)$ based on $N(\Cfrak, D, \eta)$ for some $\eta:=\eta(\varepsilon)$ to be determined. First, we note that it suffices to obtain a bound on the covering number of $\mathcal{F}_{\Cfrak}$ with respect to the maximum norm $\|\cdot\|_{\infty}$ as opposed to $\|\cdot\|_{1,X_m}$ since by our assumption on the data, $\|x_i\|_{\ell_2}\leqslant 1$ almost surely so that \begin{align*}
    \|\|\cdot\|_K\|_{1,X_m} = \frac{1}{m}\sum_{i=1}^m\|x_i\|_K\leqslant \frac{1}{m}\sum_{i=1}^m\|\|\cdot\|_K\|_{\infty} = \|\|\cdot\|_K\|_{\infty}.
\end{align*} Hence $N(\mathcal{F}_{\Cfrak}, \|\cdot\|_{1,X_m}, \varepsilon) \leqslant N(\mathcal{F}_{\Cfrak}, \|\cdot\|_{\infty}, \varepsilon)$.

Fix $\varepsilon \in (0,\tilde{\varepsilon}]$ and set $\eta := \varepsilon/C_{\Cfrak}$. Select an $\eta$-covering of $\Cfrak$, denoted by $\Cfrak_{\eta}$, of cardinality $\log |\Cfrak_{\eta}| = \log N(\Cfrak, D,\varepsilon/C_{\Cfrak})$. Then define the finite set of functions $\mathcal{F}_{\Cfrak_{\eta}} := \{\|\cdot\|_K : K \in \Cfrak_{\eta}\}$. Then $\log|\Cfrak_{\eta}| = \log |\mathcal{F}_{\Cfrak_{\eta}}|$ and for any $K \in \Cfrak$, there exists a $K_{\eta} \in \Cfrak_{\eta}$ such that $D(K, K_{\eta}) \leqslant \eta$. By the assumption that $K \mapsto \|\cdot\|_K$ is $C_{\Cfrak}$-Lipschitz in the $D$ metric, we have \begin{align*}
    \|\|\cdot\|_K- \|\cdot\|_{K_{\eta}}\|_{\infty} \leqslant C_{\Cfrak} D(K,K_{\eta}) \leqslant C_{\Cfrak}\eta = \varepsilon.
\end{align*} Hence, we have obtained an $\varepsilon$-covering of $\mathcal{F}_{\Cfrak}$ with respect to $\|\cdot\|_{\infty}$ of cardinality $ N\left(\Cfrak, D, \varepsilon/C_{\Cfrak}\right).$ 
\end{proof}

The above result establishes a uniform generalization bound for any collection of star bodies $\Cfrak$ that satisfies two requirements: 1) the functional $K \mapsto \|\cdot\|_K$ is Lipschitz continuous over $\Cfrak$ with respect to a metric $D$ and 2) $\Cfrak$ enjoys finite entropy bounds. In this work, we established that our objective is Lipschitz continuous over the space of well-conditioned star bodies $\Scal^d(r,1)$ with respect to the radial metric $\delta$. Moreover, by Corollary \ref{cor:compactness-eps-net-SrR}, this space $(\Scal^d(r,1), \delta)$ enjoys finite covering numbers. While there are several prior works that have analyzed entropy bounds for collections of convex bodies with respect to the Hausdorff metric, we are unaware of prior work analyzing such entropy bounds for star bodies in either the radial or Hausdorff metric. Our general bound can be applied to $\Scal^d(r,1)$ under the radial metric, but we do not have quantitative upper bounds on the covering number $\log N(\Scal^d(r,1), \delta, \varepsilon)$. Nevertheless, we show that our generalization bound can be applied to several interesting classes of bodies, including star bodies given as unions of convex bodies along with other classes of convex bodies that satisfy the conditions of Theorem \ref{thm:unif-generalization-bound} with respect to the Hausdorff metric $d_H$. This will recover some known results in the literature on matrix factorization while generating novel statistical learning guarantees for several classes of bodies of interest, including unions of convex bodies, ellipsoids, polytopes, and linear images of Schatten $p$-norm balls.

We note that in our general bound, we assume that the data distribution is supported in the unit Euclidean ball. This assumption could be relaxed to allow for a larger radius, which would then be absorbed into the constants $c_{P,m}$ and $c_{\Cfrak}$ that appear in our upper bound. It would be interesting to derive a result that would allow for a broader range of distributions. We focus on this class of almost surely bounded distributions for simplicity here.

In our results, we use $\lesssim$ and $\gtrsim$ to denote inequality up to positive, absolute constants. A result that we will repeatedly use in the sequel is the following approximate monotonicity property of covering numbers:
\begin{proposition}[Exercise 4.2.10 in \cite{VershyninHDP}] \label{prop:monotonicity-cov-num}
For two subsets $T_1,T_2$ of a metric space $X$ with metric $D$, $T_1 \subseteq T_2$ implies $N(T_1,D,\varepsilon) \leqslant N(T_2,D,\varepsilon/2).$    
\end{proposition} \noindent {  Finally, we note that the bound in Theorem \ref{thm:unif-generalization-bound} requires optimizing over a parameter $\varepsilon$. In the subsequent applications, we will aim to present transparent generalization bounds with a specific choice of $\varepsilon$. This choice of $\varepsilon$ does not necessarily obtain the infimum in the statement of Theorem \ref{thm:unif-generalization-bound}, but aids in obtaining cleaner bounds with more explicit dependencies on the problem parameters.}

\subsection{Application: Star bodies as Unions of Well-Conditioned Convex Bodies} \label{sec:stat-learn-star-bodies-union-convex}

We first consider a special class of star bodies that are given by unions of convex bodies. In particular, consider the following class of convex bodies with unit volume and inner width at least $r$: $\Ccal^d(r,1) := \{K \in \Ccal^d : rB^d \subseteq K,\ \vol_d(K) = 1\}$, which is a subset of the space of well-conditioned star bodies we considered in Section \ref{sec:general-existence-results}. As shown in Lemma \ref{lem:unif-bound}, we also have $K \subseteq R_r B^d$ for all $K \in \Ccal^d(r,1)$, where $R_r := \frac{d+1}{r^{d-1}\kappa_{d-1}}$ and $\kappa_{d-1} := \vol_{d-1}(B^{d-1})$. Now, for any $L \in \mathbb{N}$, define the following subset of star bodies $\Scal_L^d(r) \subset \Scal^d$ by \begin{align*}
    \Scal_L^d(r) := \left\{K = \bigcup_{i=1}^L K_i : K_i \in \Ccal^d(r,1)\ \forall\ i \in [L]\right\}.
\end{align*} This class of star bodies contains sets that can be obtained as unions of convex bodies with sufficiently large inner width and unit volume. By exploiting our previous continuity results and using known metric entropy estimates for bounded convex bodies, we can apply Theorem \ref{thm:unif-generalization-bound} to obtain a generalization bound as follows:

\begin{corollary}
Let $P$ be a distribution over $\R^d$ such that for $x \sim P$, $\|x\|_{\ell_2} \leqslant 1$ almost surely. Fix $0 < r < \infty$. Then for any $\gamma \in (0,1)$, we have that with probability $1 - \gamma$ over { $m \gtrsim dL$} samples drawn from $P$, the following bound holds for all $K \in \Scal_L^d(r)$,
{ 
\begin{align*}
    F(K; P) \lesssim F(K; P_m) + c_{P,m}\sqrt{\frac{dL(\log(m/L) + \log(R_r/r))}{m}} + c_r\sqrt{\frac{\log(1/\gamma)}{m}}
\end{align*} where $c_r \lesssim 1/r$.}

\end{corollary}
\begin{proof}
First, recall that by Proposition \ref{prop:lip-cont-radial}, we established that $K \mapsto \|\cdot\|_K$ is $C_{\Cfrak}:= 1/r^2$-Lipschitz continuous in the radial Hausdorff metric over $\Scal_L^d(r)$. Note that this part of the Proposition only requires $rB^d \subseteq K$, which is satisfied for $K \in \Scal_L^d(r)$. Now, we must establish that the metric entropy $\log N(\Scal_L^d(r),\delta,\varepsilon)$ is finite under the radial metric. To prove this, we will build an $\varepsilon$-cover of $\Scal_L^d(r)$ using known covers of bounded convex bodies in the Hausdorff metric and show how the radial distance between sets in $\Scal_L^d(r)$ can be upper bounded by their distance in the Hausdorff metric.

We record some useful known results. It has been shown \cite{Guntuboyina-opt-supp-functions} that the space $\Cfrak(r) := \{K \in \Ccal^d : K \subseteq R_rB^d\}$ has an $\varepsilon^2$-covering number in the Hausdorff metric of size $\log N(\Cfrak(r), d_H, \varepsilon) \leqslant c\cdot \frac{d-1}{2}\log\left(\frac{R_r}{\varepsilon}\right) < \infty$ for all $\varepsilon \leqslant R_r\varepsilon_0$ for some $c$ and $\varepsilon_0$ depending on $d$. Call such a covering $\tilde{\Ccal}^d_{\varepsilon^2}$ and define the relevant subset $\Ccal^d_{\varepsilon^2} := \tilde{\Ccal}^d_{\varepsilon^2} \cap \Ccal^d(r,1)$. Then for any $K \in \Ccal^d(r,1)$, there exists a $\tilde{K} \in \Ccal^d_{\varepsilon^2}$ such that $d_H(K,\tilde{K}) \leqslant \varepsilon^2$. Using this covering, we now construct a cover of $\Scal_L^d(r)$. Define the following finite class of star bodies: \begin{align*}
    \Scal_{L,\varepsilon^2}^d := \left\{K = \bigcup_{i=1}^L K_i : K_i \in \Ccal^d_{\varepsilon^2}\ \forall\ i \in [L]\right\} \subset \Scal_L^d(r).
\end{align*} We claim that this is an $R_r\varepsilon^2/r^3$-covering of $\Scal_L^d(r)$ in the radial Hausdorff metric. Fix $K \in \Scal_L^d(r)$. Then $K = \cup_{i=1}^L K_i$ where $K_i \in \Ccal^d(r,1)$ for each $i \in [L]$. By the definition of $\Ccal^d_{\varepsilon^2}$, for each $K_i$, there exists a corresponding $\tilde{K}_i \in \Ccal^d_{\varepsilon^2}$ such that $d_H(K_i,\tilde{K}_i) \leqslant \varepsilon^2$. We claim that $\delta(K, \tilde{K}) \leqslant R_r\varepsilon^2/r^3$ where $\tilde{K} := \cup_{i=1}^L \tilde{K}_i \in \Scal_{L,\varepsilon^2}^d$. First, note that each $K_i,\tilde{K}_i$ satisfies $rB^d \subseteq K_i, \tilde{K}_i \subseteq R_rB^d$. For a star body $K$ that is the union of star bodies $K_1,\dots,K_L$, its radial function satisfies $\rho_K(x) = \min_{i \in [L]} \rho_{K_i}(x)$. Moreover, we have the relation $\rho_K(x) = 1/\|x\|_K$. Thus, for any $\|x\|_{\ell_2} = 1$, we have that \begin{align*}
    |\rho_K(x) - \rho_{\tilde{K}}(x)|  = \left|\min_{i \in [L]} \rho_{K_i}(x) - \min_{i \in [L]} \rho_{\tilde{K}_i}(x)\right| 
    & = \left|\min_{i \in [L]} \frac{1}{\|x\|_{K_i}} - \min_{i \in [L]} \frac{1}{\|x\|_{\tilde{K}_i}}  \right| \\
    & = \left|\max_{i \in [L]} \|x\|_{K_i} - \max_{i \in [L]} \|x\|_{\tilde{K}_i}\right| \\
    & \leqslant  \max_{i \in [L]} \left| \|x\|_{K_i} - \|x\|_{\tilde{K}_i} \right| \\
    & \leqslant \frac{R_r}{r^3}d_H(K_{i_*}, \tilde{K}_{i_*}) \\
    & \leqslant \frac{R_r}{r^3} \cdot \varepsilon^2.
\end{align*} where we set $i_* \in \argmax_{i \in [L]} | \|x\|_{K_i} - \|x\|_{\tilde{K}_i} |$, used Proposition \ref{prop:lip-cont-radial} in the second to last inequality, and the definition of $\Ccal^d_{\varepsilon^2}$ in the final inequality. Since this holds for any $\|x\|_{\ell_2} = 1$, we can take the maximum over the sphere to obtain \begin{align*}
    \delta(K, \tilde{K}) & = \max_{\|x\|_{\ell_2}=1} |\rho_K(x) - \rho_{\tilde{K}}(x)| \leqslant \frac{R_r}{r^3} \cdot \varepsilon^2.
\end{align*} Since $K \in \Scal_L^d(r)$ was arbitrary, this shows that $\Scal_{L,\varepsilon^2}^d$ is an $R_r\varepsilon^2/r^3$-covering of $\Scal_L^d(r)$ in the radial metric. Note that this covering is of cardinality $|\Scal_{L,\varepsilon^2}^d| = \binom{|\Ccal^d_{\varepsilon^2}|}{L}.$ Using the upper bound $\binom{n}{k} \leqslant (n \cdot e/ k)^k$, we can estimate the metric entropy as \begin{align*}
    \log N\left(\Scal_L^d(r), \delta, \varepsilon\right) \leqslant \log \binom{|\Ccal_{\varepsilon}^d|}{L} \lesssim L \log \left(\frac{|\Ccal_{\varepsilon}^d|}{L}\right)
\end{align*} where $\log |\Ccal_{\varepsilon}^d| \lesssim d \log(R_r/\varepsilon)$. {  Choosing $\varepsilon := \sqrt{dL/m}$ in the bound on Theorem \ref{thm:unif-generalization-bound} achieves the desired result.}
\end{proof}

\noindent \textbf{Remark on the choice of parameters.} This result focuses on the class of star bodies that can be expressed as unions of convex bodies with fixed volume and inner width at least $r > 0$, but a similar bound can also be obtained when the convex bodies are restricted to lie in $RB^d$ for a user-specified outer radius $R < \infty$. Note that the outer radius $R_r$ grows with the dimension if $r$ is kept constant. In high-dimensions, however, the volume of the unit ball concentrates on the boundary, so it is natural to instead choose a value of $r$ that also increases with the dimension, e.g., setting $r = \alpha \sqrt{d}$ for some constant $0 < \alpha < 1$. This would ensure that $R_r$ and $r$ are on the same order. 

\noindent \textbf{Simplification for well-conditioned convex regularizers.} Note that this above result also gives a generalization guarantee for well-conditioned convex regularizers in the case $L = 1$. One could improve this result in terms of powers of $r$ for the convex case by directly using relevant results solely in the Hausdorff topology. Note that this cannot be done for the case of unions of convex bodies, since the radial topology is more natural over the space of star bodies.

\subsection{Application: Ellipsoidal Regularizers} \label{sec:stat-learn-ellipsoids}

We additionally consider another class of convex bodies that are linear images of a norm ball. Specifically, we showcase a generalization bound for learning regularizers induced by ellipsoids.

\begin{corollary}[Generalization Bound for Ellipsoids] \label{cor:ellipsoids-gen-bound}
Let $P$ be a distribution over $\R^d$ such that for $x \sim P$, $\|x\|_{\ell_2} \leqslant 1$ almost surely. Fix $0 < r < R < \infty$. Then for any $\gamma \in (0,1)$, we have that with probability $1 - \gamma$ over { $m \gtrsim d(d+1)$} samples drawn from $P$, the following bound holds for all $K \in \{A(B^d) : A \in \R^{d\times d}, A = A^T, RI \succeq A \succeq rI\}$:
{ \begin{align*}
    F(K; P) \lesssim F(K; P_m) + c_{P,m}\sqrt{\frac{d(d+1)(\log(m) + \log(R/r))}{m}} + c_r\sqrt{\frac{\log(1/\gamma)}{m}}.
\end{align*}}
\end{corollary}
\begin{proof}[Proof of Corollary \ref{cor:ellipsoids-gen-bound}]
Let $\mathfrak{A} := \{A \in \R^{d\times d} : A=A^T, rI \preceq A \preceq RI\}$. Observe that for any $A \in \mathfrak{A}$ and $x \in A(B^d)$, we have $\|x\|_{\ell_2} \leqslant \|A\|\|z\|_{\ell_2} \leqslant R$ where $x = Az$ and $\|z\|_{\ell_2}\leqslant 1$, meaning $A(B^d) \subseteq R B^d$. Moreover, $\|Az\|_{\ell_2} \geqslant r\|z\|_{\ell_2}$ for any $z \in \R^d$ so $A(B^d) \supseteq rB^d$. Thus, $K \mapsto \|\cdot\|_K$ is Lipschitz over $\Cfrak = \{A(B^d) : A \in \mathfrak{A}\}$ with respect to the Hausdorff metric.

To complete the proof, we show that the covering number of $\Cfrak $ in the Hausdorff metric can be bounded by the covering number of $\mathfrak{A}$ with respect to the spectral norm $\|\cdot\|$. This follows from noting that for any $A,D \in \mathfrak{A}$, we have by definition of the Hausdorff metric that \begin{align*}
    d_H(A(B^d),D(B^d)) & = \max_{\|x\|_{\ell_2} = 1}|h_{A(B^d)}(x) - h_{D(B^d)}(x)| \\
    & = \max_{\|x\|_{\ell_2} = 1}|h_{B^d}(A^Tx) - h_{B^d}(D^Tx)| \\
    & = \max_{\|x\|_{\ell_2} = 1}|\|A^Tx\|_{\ell_2} - \|D^Tx\|_{\ell_2}| \\
    & \leqslant \max_{\|x\|_{\ell_2} = 1}\|(A - D)^Tx\|_{\ell_2} = : \|A-D\|.
\end{align*} But note that the covering number of $\mathfrak{A}$ with respect to the spectral norm $\|\cdot\|$ is bounded by \begin{align*}
    \log N(\mathfrak{A}, \|\cdot\|,\varepsilon) \leqslant \frac{d(d+1)}{2} \cdot \log\left(\frac{3R}{\varepsilon}\right)
\end{align*} since the set of PSD matrices is a Riemannian manifold of dimension $d(d+1)/2$ \cite{samp-comp-DL}. This gives the bound \begin{align*}
    \log N\left(\Cfrak, d_H, \frac{\varepsilon r^3}{  R}\right) \leqslant \log N\left(\mathfrak{A}, \|\cdot\|, \frac{\varepsilon r^3}{  R}\right) \leqslant \frac{d(d+1)}{2} \log \left(\frac{3  R^2}{\varepsilon r^3}\right).
\end{align*} {  Choosing $\varepsilon := \sqrt{d(d+1)/m}$ achieves the desired result.}
\end{proof}

\subsection{Application: Polyhedral Regularizers and Dictionary Learning} \label{sec:stat-learn-polytopes-DL}

We now apply our result to another important class of convex bodies: polytopes. This bound also characterizes a generalization bound for dictionary learning. Let $\Pcal_T \subset \Ccal^d$ denote the space of polytopes with $T$ vertices: \begin{align*}
    \Pcal_T := \{\mathrm{conv}(S) : S \subset \R^d\ \text{has cardinality at most}\ T\}.
\end{align*} For $0 < r < R < \infty$, let $\Pcal_T(r,R) := \{K \in \Pcal_T : r B^d \subseteq K \subseteq R B^d\}$. Note that by Lemma \ref{lem:unif-bound}, $\{K \in \Pcal_T : r B^d \subseteq K,\ \vol_d(K) = 1\} \subseteq \Pcal_T(r,R_r)$, so the following bound can be applied to this class of volume-normalized polytopes as well. 

\begin{corollary} \label{cor:polytope-gen-bound}
Let $P$ be a distribution over $\R^d$ such that for $x \sim P$, $\|x\|_{\ell_2} \leqslant 1$ almost surely. Fix $0 < r < R < \infty$. Then for any $\gamma \in (0,1)$, we have that with probability $1 - \gamma$ over { $m \gtrsim Td$} samples drawn from $P$, the following bound holds for all $K \in \Pcal_T(r,R)$,
{ 
\begin{align*}
    F(K; P) \lesssim F(K; P_m) + c_{P,m}\sqrt{\frac{Td(\log(m) + \log(R/r))}{m}} + c_r\sqrt{\frac{\log(1/\gamma)}{m}}.
\end{align*}}

\end{corollary} 

\begin{proof}
First, note that due to our assumptions on $\Pcal_T(r,R)$, we have again by Proposition \ref{prop:lip-cont-radial} that $K \mapsto \|x\|_K$ is $R/r^3$-Lipschitz with respect to the Hausdorff metric over this class of convex bodies. It has additionally been shown  \cite{Guntuboyina-opt-supp-functions} (proof of Lemma C.1) that $\Pcal_T(R) : = \{K \in \Pcal_T : K \subseteq R B^d\}$ enjoys a finite $\varepsilon$-covering number with respect to the Hausdorff metric, which scales as \begin{align*}
    \log N(\Pcal_T(R), d_H, \varepsilon) \lesssim Td \cdot \log \left( \frac{MR}{\varepsilon}\right)
\end{align*} for some positive absolute constant $M$. Using the monotonicity property (Proposition \ref{prop:monotonicity-cov-num}) to bound the $\varepsilon r^3/R$-covering number of $\Pcal_T(r,R)$ by that of $\Pcal_T(R)$ and { setting $\varepsilon := \sqrt{Td / m}$} achieves the desired result.
\end{proof}

The class of bounded polytopes of sufficiently large inner width covers the setting of dictionary learning under some conditions on the dictionaries of interest. In particular, recall that in dictionary learning, the goal is to learn a basis $A \in \R^{d \times p}$ such that a given dataset $\{x_i\}_{i=1}^N$ can be sparsely represented as $x_i \approx Az_i$ where $z_i$ is sparse. In most cases, sparsity is induced by solving for latent codes $z_i$ with small $\ell_1$-norm. In this case, the learned regularizer can be shown to be a gauge function induced by a linear image of the $\ell_1$-ball: 

$$
\| x \|_{A(\Bcal_{\ell_1})} = \inf \{ \| z \|_{\ell_1} : z = Ax \}.
$$ This follows by definition, since $\| x \|_{A(\Bcal_{\ell_1})}  =  \inf \{ t > 0 :  x \in t \cdot A (\Bcal_{\ell_1}) \} = \inf \{ t > 0 :  x = Az , z \in t \cdot  \Bcal_{\ell_1} \} = \inf \{ \| z \|_{\ell_1} : x = Az \}$. 
Thus, the regularizers of interest in dictionary learning are gauges induced by linear images of the $\ell_1$-ball, where the linear map $A$ is the object to be learned. That is, for $A \in \R^{d \times p}$, the convex body of interest is given by \begin{align*}
    A(\Bcal_{\ell_1}) := \{Az \in \R^d : \|z\|_{\ell_1} \leqslant 1\} = \mathrm{conv}(\{A_1,\dots,A_p\}).
\end{align*} 

To apply our result for polytopes, we need to show that a class of suitable dictionaries $\mathfrak{A} \subset \R^{d \times p}$ induces a collection of convex bodies $\{A(\Bcal_{\ell_1}) : A \in \mathfrak{A}\}$ that satisfies the geometric conditions $r B^d \subseteq A(\Bcal_{\ell_1}) \subseteq RB^d$ for all $A \in \mathfrak{A}$. For a parameter $\eta > 0$, consider the collection of dictionaries $$\mathfrak{A} := \{A \in \R^{d \times p} : \|A_i\|_{\ell_2} = 1,\ \rank(A) = d,\ \sqrt{\sigma_{\min}(AA^T)} \geqslant \eta > 0\}$$ and fix $A \in \mathfrak{A}$. We first show $A(\Bcal_{\ell_1}) \subseteq RB^d$ for some $R > 0$. Consider $x \in A(\Bcal_{\ell_1})$. Then $x = Az$ for $\|z\|_{\ell_1} \leqslant 1$. Since $\|A\|^2 \leqslant \|A\|_{\ell_2}^2 = p$ where $\|\cdot\|$ and $\|\cdot\|_{\ell_2}$ denote the spectral and Frobenius norm, respectively, we see that $$\|x\|_{\ell_2} = \|Az\|_{\ell_2} \leqslant \|A\|\|z\|_{\ell_2} \leqslant \sqrt{p}\|z\|_{\ell_1} \leqslant \sqrt{p}.$$ Thus $A(\Bcal_{\ell_1}) \subseteq \sqrt{p}B^d$. We now show $A(\Bcal_{\ell_1}) \supseteq rB^d$ for some $0 < r < R$. Note that this inclusion holds if and only if the support functions satisfy $h_{A(\Bcal_{\ell_1})} \geqslant h_{rB^d}$. Note that for any $\|u\|_{\ell_2} = 1$, we have that the support function of $A(\Bcal_{\ell_1})$ is given by \begin{align*}
    h_{A(\Bcal_{\ell_1})}(u) = h_{\Bcal_{\ell_1}}(A^Tu) = \|A^Tu\|_{\ell_{\infty}}.
\end{align*} But since $A \in \mathfrak{A}$, note that  \begin{align*}
    \min_{\|u\|_{\ell_2} = 1} \|A^Tu\|_{\ell_{\infty}} \geqslant \frac{1}{\sqrt{p}} \min_{\|u\|_{\ell_2} = 1} \|A^Tu\|_{\ell_2} = \frac{1}{\sqrt{p}} \cdot \sqrt{\sigma_{\min}(AA^T)} \geqslant \frac{ \eta}{\sqrt{p}}.
\end{align*} This guarantees $\eta/\sqrt{p}B^d\subseteq A(\Bcal_{\ell_1})$. 

Thus each $K \in \{A(\Bcal_{\ell_1}) : A \in \mathfrak{A}\}$ satisfies $\eta/\sqrt{p}B^d \subseteq K \subseteq \sqrt{p}B^d$. Hence the functional $\|x\|_K$ is Lipschitz over the collection $\Cfrak := \{A(\Bcal_{\ell_1}) : A \in \mathfrak{A}\}$. Moreover, the set $$\{A(\Bcal_{\ell_1}) : A \in \mathfrak{A}\} \subseteq \Pcal_p(\sqrt{p})$$ since each $A(\Bcal_{\ell_1})$ is a polytope with at most $p$ extreme points. Thus we achieve the same bound in Corollary \ref{cor:polytope-gen-bound} with $T = p$, $r = \eta/\sqrt{p}$,  and $R = \sqrt{p}.$ These bounds match previously achieved generalization bounds for dictionary learning \cite{Vainsencher_etal11, samp-comp-DL} up to logarithmic factors.

\subsection{Application: Semidefinite Regularizers} \label{sec:stat-learn-semidef}

Our final class of convex bodies consists of linear images of Schatten $p$-norm balls. Such classes have found wide interest in data science. For example, if one wants to learn an infinite collection of dictionary elements, a natural generalization  of dictionary learning would require learning a linear image of the nuclear norm ball \cite{SohChandrasekaran}. The set of dictionary elements in this case corresponds to the convex hull of the set of rank-$1$ matrices.

For a matrix $A \in \R^{q \times q}$, let $\sigma_i(A)$ denote the $i$-th singular value of $A$. For $p \in [1,\infty]$, we define the Schatten $p$-norm of $A$ as \begin{align*}
    \|A\|_{S_p} := \begin{cases}
    \left(\sum_{i=1}^{q} \sigma_i(A)^p\right)^{1/p} & \text{if}\ p < \infty, \\
    \sigma_{\max}(A) & \text{else.}
    \end{cases}
\end{align*} Let $\Bcal_{S_p} := \{A \in \R^{q \times q} : \|A\|_{S_p} \leqslant 1\}$ denote the Schatten $p$-norm ball and define the space of linear maps from $\R^{q \times q}$ to $\R^d$ as $L(\R^{q \times q} , \R^d)$. For a linear map $\Lcal \in L(\R^{q \times q} , \R^d)$, define the induced operator norm $\|\cdot\|_{S_p,2}$ as \begin{align*}
    \|\Lcal\|_{S_p,2} := \max_{Z \in \Bcal_{S_p}} \|\Lcal(Z)\|_{\ell_2}.
\end{align*} For a collection of linear maps $\mathfrak{L}_p \subset L(\R^{q \times q}, \R^d)$, let $\Cfrak_p$ denote the induced collection of convex bodies which are linear images of Schatten $p$-norm balls  \begin{align*}
    \Cfrak_p := \{\Lcal(\Bcal_{S_p}) : \Lcal \in \mathfrak{L}_p \subset L(\R^{q \times q}, \R^d)\}. 
\end{align*} We will prove the following result, which prescribes a set of conditions on the collection $\mathfrak{L}_p$ to obtain a generalization bound.

\begin{corollary}
Let $P$ be a distribution over $\R^d$ such that for $x \sim P$, $\|x\|_{\ell_2} \leqslant 1$ almost surely. For $p \geqslant 1$, let $\tilde{p}$ denote the conjugate\footnote{Recall that the conjugate $\tilde{p}$ of $p$ is $\tilde{p} = \infty$ if $p = 1$ and $1/p + 1/\tilde{p} = 1$ otherwise.} of $p$ and consider the set of bounded linear maps $$\mathfrak{L}_p(r,R) := \left\{\Lcal \in L(\R^{q \times q}, \R^d) : r \leqslant \min_{\|x\|_{\ell_2} = 1}\|\Lcal^Tx\|_{S_{\tilde{p}}} \leqslant \|\Lcal\|_{S_p,2} \leqslant R\right\}.$$ Then for any $\gamma \in (0,1)$, we have that with probability $1 - \gamma$ over { $m \gtrsim q^2 d$} samples drawn from $P$, the following bound holds for all $K \in \{\Lcal(\mathcal{B}_{S_p}) : \Lcal \in \mathfrak{L}_p(r,R)\}$,
{ 
\begin{align*}
    F(K; P) \lesssim F(K; P_m) + c_{P,m}\sqrt{\frac{q^2 d(\log(m) + \log(R/r))}{m}} + c_r \sqrt{\frac{\log(1/\gamma)}{m}}.
\end{align*}
}

\end{corollary}
\begin{proof}
We first show that the collection $\Cfrak_p := \{\Lcal(\mathcal{B}_{S_p}) : \Lcal \in \mathfrak{L}_p(r,R)\}$ is bounded and has a sufficiently large inner width. Clearly for $K \in \Cfrak_p$, we have $K \subseteq RB^d$. Moreover, note that the support function $h_{\Lcal(\Bcal_{S_p})}(u) = \|\Lcal^Tu\|_{S_{\tilde{p}}}$. Thus, for $K \in \Cfrak_p$, the condition $rB^d \subseteq K$ is equivalent to \begin{align*}
     h_{rB^d}(u) \leqslant h_{\Lcal(\Bcal_{S_p})}(u)\ \forall \|u\|_{\ell_2} = 1
    \Longleftrightarrow r \leqslant \min_{\|u\|_{\ell_2} = 1} \|\Lcal^Tu\|_{S_{\tilde{p}}}.
\end{align*} By assumption, this inequality holds.

Now we need to bound the covering number of $\Cfrak_p$ in the Hausdorff metric. We claim such a bound can be obtained by covering the space $\mathfrak{L}_p$ in the $\|\cdot\|_{S_p,2}$ norm. Arguing similarly to proof of Corollary \ref{cor:ellipsoids-gen-bound}, we can show that for $\Lcal, \Acal \in \mathfrak{L}_p$, we have \begin{align*}
    d_H(\Lcal(\Bcal_{S_p}), \Acal(\Bcal_{S_p})) & \leqslant \max_{\|x\|_{\ell_2} = 1}\|(\Lcal^T-\Acal^T)x\|_{S_{\tilde{p}}}  =: \|\Lcal^T - \Acal^T\|_{2,S_{\tilde{p}}}.
\end{align*} But by duality, note that $\|X\|_{S_p} = \max\{|\langle X,Z\rangle| : \|Z\|_{S_{\tilde{p}}} \leqslant 1\}$ so that for any linear map $\Lcal \in L(\R^{q \times q}, \R^d)$, \begin{align*}
     \|\Lcal^T\|_{2,S_{\tilde{p}}} = \max_{\|x\|_{\ell_2}=1}  \|\Lcal^Tx\|_{S_{\tilde{p}}} =  \max_{\|x\|_{\ell_2} = 1}\max_{\|Z\|_{S_p} \leqslant 1} |\langle Z, \Lcal^Tx\rangle|
     & = \max_{\|x\|_{\ell_2} = 1}\max_{\|Z\|_{S_p} \leqslant 1} |\langle \Lcal(Z),x\rangle| \\
     & = \max_{\|Z\|_{S_p} \leqslant 1} \max_{\|x\|_{\ell_2} = 1} |\langle \Lcal(Z),x\rangle| \\
     & = \max_{\|Z\|_{S_p} \leqslant 1}  \|\Lcal(Z)\|_{\ell_2} \\
     & = \|\Lcal\|_{S_p,2}
\end{align*} which gives $d_H(\Lcal(\Bcal_{S_p}), \Acal(\Bcal_{S_p})) \leqslant  \|\Lcal - \Acal\|_{S_p,2}.$ It suffices to then bound the covering number of $\mathfrak{L}_p(R):=\{\Lcal \in L(\R^{q \times q}, \R^d) : \|\Lcal\|_{S_p,2} \leqslant R\}$. But note that $\mathfrak{L}_p(R)$ is isomorphic to the space of $d \times q^2$ matrices of bounded $\|\cdot\|_{S_p,2}$ norm. This gives the estimate \begin{align*}
     \log N(\Cfrak_p, d_H,\varepsilon) \leqslant \log N(\mathfrak{L}_p, \|\cdot\|_{S_p,2},\varepsilon) \leqslant \log N(\mathfrak{L}_p(R), \|\cdot\|_{S_p,2},\varepsilon/2) \lesssim q^2d \log\left(1 + \frac{2R}{\varepsilon}\right).
\end{align*} Bounding the $\varepsilon r^3/R$-covering number of $\Cfrak_p$ in the Hausdorff metric by that of $\mathfrak{L}_p(R)$ in the metric induced by the $\|\cdot\|_{S_p,2}$-norm and { setting $\varepsilon := \sqrt{q^2d / m}$} achieves the claimed bound.
\end{proof}

\section{Consequences Beyond Optimization} \label{sec:posterior-persp}
In this section, we briefly illustrate the consequences of our approach to regularizer selection for solution methods of inverse problems that do not rely on optimization but instead on the closely related topic of sampling.  In the preceding sections of this paper, we advocate for summarizing a data distribution $P$ via a star body regularizer $\|\cdot\|_K$.  As discussed in Section \ref{sec:MLE-criterion}, a regularizer $\|\cdot\|_K$ defined by the gauge of a star body $K \subseteq \R^d$ is optimal for a data distribution $P$ if the Gibbs density with energy $\|\cdot\|_K$ is the projection of $P$ onto the set of regularizer-induced Gibbs densities $\overline{\mathcal{D}}$.  With this notion of optimality, our results in Section~\ref{sec:opt_reg_convex} on conditions for a data distribution under which the optimal regularizer is convex have computational implications beyond optimization.

Concretely, in the Bayesian paradigm for solving inverse problems, one aims to reconstruct an object from corrupted measurements by sampling from a suitable posterior distribution.  When the prior is log-concave (i.e., when it is proportional to $\exp(-V)$ for a convex function $V$), the induced posterior distribution is also log-concave for many common likelihood models for the observed data (e.g., in linear inverse problems with likelihood functions given by exponential family models).  In such situations, we are faced with the task of sampling from a log-concave posterior.  The area of log-concave sampling is rich with algorithms and results guaranteeing convergence, for a wide range of smooth and non-smooth convex energies $V$ \cite{VempalaSurvey, RobertsTweedie1, RobertsTweedie2, LovaszVempala1, LovaszVempala2, Bubecketal18, Durmusetal19, Dwivedietal19}.  Consequently, the identification of log-concave Gibbs densities that can approximate a given prior distribution $P$ by obtaining convex regularizers corresponding to $P$ can yield a computationally tractable sampling-based approach for solving inverse problems.

We now describe this Bayesian perspective in more detail. Suppose we collect measurements $y \in \R^m$ of an underlying signal $x_* \in \R^d$, which are modeled via the likelihood function of $y$ given $x$ of the form $e^{-\Phi(x;y)}$ for a known map $\Phi$.  (For example, if $y$ and $x_*$ are related via $y = f(x_*) + \eta$ for some known forward model $f:\R^d \rightarrow \R^m$ and $\eta \sim \mathcal{N}(0,\sigma^2 I_m)$, then $\Phi(x;y) := (2\sigma^2)^{-1}\|y - f(x)\|_{\ell_2}^2$.)  From the Bayesian perspective \cite{Stuart-BIP}, one views $x_*$ as being drawn from a prior distribution $P$ and the quantity of interest for estimating $x_*$ is the posterior distribution of $x_*$ given $y$, which we denote by $\mu_P$:
\begin{align*}
    \mu_{P}(\mathrm{d}x) := \frac{1}{Z_P(y)} e^{-\Phi(x;y)} \mathrm{d}P(x)
\end{align*} where $Z_P(y) := \int_{\R^d} e^{-\Phi(x;y)} \mathrm{d}P(x)$ is the normalizing constant.

If instead of using the posterior $\mu_P$ we use the posterior $\mu_{p_K^*}$ induced by the moment projection $p_K^*$ of $P$ onto $\overline{\mathcal{D}}$, it is natural to ask what the implications might be in terms of reconstruction. Using recent results in the Bayesian inference literature, one can show that the distance from the posterior $\mu_P$ induced by the true prior $P$ to the induced posterior distribution $\mu_{p_K^*}$ is upper bounded by the distance between $P$ and the moment projection $p_K^*$. To state the result, recall that the total variation distance between two measures $P_1,P_2$ is given by $d_{TV}(P_1,P_2) := \sup_{A}|P_1(A) - P_2(A)|$ where the supremum is over all measurable subsets. \begin{proposition}
    Fix $y \in \R^m$. Consider two prior measures $P_1,P_2$ on $\R^d$ with $P_1 \ll P_2$ and log-likelihood function $\Phi(\cdot;y) : \R^d \rightarrow \R$ such that $Z_{P_i}(y):= \int_{\R^d}e^{-\Phi(x;y)}\mathrm{d}P_i(x) < \infty$ for each $i=1,2$. Defining the posterior measures $\mu_{P_i}(\mathrm{d}x) := Z_{P_i}(y)^{-1}e^{-\Phi(x;y)}\mathrm{d}P_i(x)$, we have that the following bounds hold: \begin{align*}
        d_{TV}^2(\mu_{P_1}, \mu_{P_2}) \leqslant \frac{2}{Z_{P_1}^2(y)}D_{\mathrm{KL}}(P_1 || P_2).
    \end{align*}
\end{proposition} \begin{proof}
    This result follows from Theorem 8 in \cite{Sprungk20} along with an application of Pinsker's inequality \cite{Pinsker64}, which states that $d_{TV}(P_1,P_2) \leqslant \sqrt{D_{\mathrm{KL}}(P_1 || P_2)/2}$ for any two measures $P_1,P_2$.
\end{proof}
Hence, if the normalizing constants $Z_P(y), Z_{p_K}(y)$ are finite, the $m$-projection $p_K^*$ of $P$ onto $\overline{\mathcal{D}}$ satisfies:
\begin{align*}
    \min_{p_K \in \overline{\mathcal{D}}} d^2_{TV}(\mu_P, \mu_{p_K}) \leqslant \frac{2}{Z_P(y)^2} D_{\mathrm{KL}}(P || p_K^*).
\end{align*}
Thus, identifying an element $p_K^*$ of the class $\overline{\mathcal{D}}$ to best approximate a prior $P$ with a view to minimizing the total variation between the corresponding posteriors $\mu_{P}, \mu_{p_K^*}$ may be achieved by setting $p_K^*$ equal to the $m$-projection of $P$ onto $\overline{\mathcal{D}}$; in turn, this $m$-projection can be obtained by deriving the optimal regularizer for $P$.  If this optimal regularizer is convex so that $p_K^*$ is log-concave, then under appropriate conditions on the likelihood $e^{-\Phi(x;y)}$ the posterior $\mu_{p_K^*}$ is log-concave.  To summarize this line of reasoning, identifying an optimal regularizer for a data distribution $P$ and conditions under which it is convex yield methods for obtaining samples from a log-concave posterior whose distance to the true posterior is controlled in terms of the quality of the approximation of $P$.


\section{Conclusion and Discussion}

\label{sec:conclusions}

In this paper, we analyzed the question of determining the optimal regularizer for a given data distribution. We showed that when considering the class of continuous, positively homogeneous regularizers, this question is equivalent to solving a variational optimization problem over the space of star bodies. We showed that given a distribution of interest, there exists a unique star body of fixed volume that achieves the optimal risk by exploiting dual Brunn-Minkowski theory. The density of the distribution helps characterize the radial function of the optimal star body, which can have a wide array of geometries depending on the distribution under consideration. We further analyzed the convergence of empirical risk minimizers in the limit of infinite data and robustness of solutions to noisy perturbations. We then obtained generalization guarantees for various classes of star bodies and convex bodies, recovering known results while developing novel guarantees. There are many interesting directions for future work, and we describe a few here:

\paragraph{Performance of regularizers in subsequent tasks.} An important next step in the theory of optimal regularizers is to understand their performance in downstream tasks such as inverse problems. These types of questions are interesting from a generalization perspective, and would bring unique mathematical challenges. For example, previous works \cite{OymakHassibi, Oymaketal13} have established worst-case bounds on denoising performance using a convex regularizer, which depends on the statistical complexity of the regularizer's subdifferential at the point of interest. In our setting and the convex case, this would amount to analyzing the subdifferential of $\|\cdot\|_{K_*}$ the gauge induced by the optimal star body $K_*$. Additionally, this would also likely involve quantifying the Gaussian width or statistical dimension \cite{Chandrasekaranetal12a, ChandrasekaranJordan13, Amelunxenetal14} of the tangent cone of convex bodies at datapoints of interest, a quantity that has been well-studied in the context of inverse problems and statistical estimation tasks. Understanding what the analogous quantities would be in the general nonconvex star case could lead to new geometric questions regarding the boundary structure of such bodies. 

\paragraph{Approximation power of convex bodies.} The focus of our paper mainly lied in characterizing the optimal star body regularizer for a given distribution, but did not investigate questions related to modeling mismatch. It is of interest to investigate situations such as those discussed in Section \ref{sec:general-existence-results}, where the collection of sets we are searching over $\Cfrak$ does not contain the optimal regularizer $K_*$. For example, if the optimal regularizer $K_*$ for a given distribution is nonconvex, what is the best convex approximation to this regularizer? Along these lines, are there benefits to considering the best convex approximation to the optimal star body regularizer from a computational or analytic perspective? The answers to such questions will further advance our understanding of the powers and limitations of convexity.

\paragraph{Computing the optimal regularizer.}  
One interesting direction that stems from our work is to provide a practical method for learning optimal regularizers that are specified as the gauge of a star body for a given data set.  The dictionary learning problem provides a roadmap for describing structured families of convex regularizers that are extremely expressive while being amenable to training (via alternating minimization methods) -- these are regularizers specified as the gauge of some polytope parametrized as some linear image of the $\ell_1$-ball.  A natural family of \emph{nonconvex} regularizers that extends the ideas of dictionary learning are those specified as the gauge of unions of a \emph{few} convex sets, each of which could be, for instance, specified as some projection of the $\ell_1$-ball.  It would be interesting to develop numerical algorithms for learning regularizers with such structure. %

\paragraph{Incorporating computational considerations.}  An outstanding challenge in the field of inverse problems is to better understand how contemporary data-driven methods for learning regularizers perform, and explain why they seem to work well in practice.  An interesting direction based on our work is to view stylized instances of these practical methods as instances of our framework, but with the additional constraint that the star body arises from some computational model.  For instance, we may view the dictionary learning problem as an instance of our framework with the constraint that the star body is specified as the linear image of an $\ell_1$-ball.  It would be interesting to investigate the relationship between the optimal regularizers corresponding to the constrained (to being specified by some computational model) and unconstrained instances, and in the process, provide insight into the performance of practical methods for learning regularizers.

\section*{Acknowledgements}

We thank Shuhan Yang for notifying us of a sign error in a previous version of this manuscript. We also thank Andrew Stuart for helpful discussions. VC was supported in part by AFOSR grants FA9550-23-1-0204, FA9550-23-1-0070 and NSF grant DMS 2113724. YS acknowledges support from the Ministry of Education (Singapore) Academic Research Fund (Tier 1) R-146-000-329-133.

\bibliographystyle{plain}

\bibliography{main.bib}

\begin{thebibliography}{10}

\bibitem{Agarwaletal16}
Alekh Agarwal, Animashree Anandkumar, Prateek Jain, and Praneeth Netrapalli.
\newblock Learning sparsely used overcomplete dictionaries via alternating minimization.
\newblock {\em SIAM Journal on Optimization}, 26(4):2775--2799, 2016.

\bibitem{Agarwaletal17}
Alekh Agarwal, Animashree Anandkumar, and Praneeth Netrapalli.
\newblock A clustering approach to learn sparsely-used overcomplete dictionaries.
\newblock {\em IEEE Transactions on Information Theory}, 63(1):575--592, 2017.

\bibitem{Aharonetal2006}
Michal Aharon, Michael Elad, and Alfred Bruckstein.
\newblock K-svd: An algorithm for designing overcomplete dictionaries for sparse representation.
\newblock {\em IEEE Transactions on Signal Processing}, 54(11):4311--4322, 2006.

\bibitem{amari2000methods}
S.~Amari and H.~Nagaoka.
\newblock {\em Methods of Information Geometry}.
\newblock Translations of mathematical monographs. American Mathematical Society, 2000.

\bibitem{Amelunxenetal14}
Dennis Amelunxen, Martin Lotz, Michael~B. McCoy, and Joel~A. Tropp.
\newblock Living on the edge: Phase transitions in convex programs with random data.
\newblock {\em Information and Inference: A Journal of the IMA}, 224–294:3(3), 2014.

\bibitem{Aroraetal15}
Sanjeev Arora, Rong Ge, Tengyu Ma, and Ankur Moitra.
\newblock Simple, efficient, and neural algorithms for sparse coding.
\newblock {\em Conference on Learning Theory}, 2015.

\bibitem{Arridgeetal19}
Simon Arridge, Peter Maass, Ozan \"{O}ktem, and Carola-Bibiane Sch\"{o}nlieb.
\newblock Solving inverse problems using data-driven models.
\newblock {\em Acta Numerica}, 28:1--174, 2019.

\bibitem{Baraketal15}
Boaz Barak, Jonathan~A. Kelner, and David Steurer.
\newblock Dictionary learning and tensor decomposition via the sum-of-squares method.
\newblock {\em Proceedings of the Forty-seventh Annual ACM Symposium on Theory of Computing}, page 143–151, 2015.

\bibitem{Bartlett_etal05}
Peter~L. Bartlett, Olivier Bousquet, and Shahar Mendelson.
\newblock Local rademacher complexities.
\newblock {\em Annals of Statistics}, 33:1497–1537, 2005.

\bibitem{BenningBurger18}
Martin Benning and Martin Burger.
\newblock Modern regularization methods for inverse problems.
\newblock {\em Acta Numerica}, 27:1–111, 2018.

\bibitem{BhaskarRecht13}
Badri Bhaskar and Benjamin Recht.
\newblock Atomic norm denoising with applications to line spectral estimation.
\newblock {\em IEEE Transactions on Signal Processing}, 61(23):5987–5999, 2013.

\bibitem{Boraetal17}
Ashish Bora, Ajil Jalal, Eric Price, and Alexandros Dimakis.
\newblock Compressed sensing using generative models.
\newblock {\em International Conference on Machine Learning}, 2017.

\bibitem{Braides-Gamma-Handbook}
Andrea Braides.
\newblock A handbook of {$\Gamma$}-convergence.
\newblock {\em Handbook of Differential Equations: Stationary Partial Differential Equations}, Volume 3, 2007.

\bibitem{BrediesKunischPock}
Kristian Bredies, Karl Kunisch, and Thomas Pock.
\newblock Total generalized variation.
\newblock {\em SIAM Journal on Imaging Sciences}, 3:492 -- 526, 2011.

\bibitem{Bubecketal18}
S\'{e}bastien Bubeck, Ronen Eldan, and Joseph Lehec.
\newblock Sampling from a log-concave distribution with projected langevin monte carlo.
\newblock {\em Discrete and Computational Geometry}, 59:757--783, 2018.

\bibitem{CandesRecht09}
Emmanuel~J. Cand\`{e}s and Benjamin Recht.
\newblock Exact matrix completion via convex optimization.
\newblock {\em Foundations of Computational Mathematics}, 9(6):717–772, 2009.

\bibitem{Tao2006}
Emmanuel~J. Cand\`{e}s, Justin~K. Romberg, and Terence Tao.
\newblock Stable signal recovery from incomplete and inaccurate measurements.
\newblock {\em Communications on Pure and Applied Mathematics}, 59(8):1207--1223, 2006.

\bibitem{ChambolleLions}
Antonin Chambolle and Pierre-Louis Lions.
\newblock Image recovery via total variation minimization and related problems.
\newblock {\em Numer. Math.}, 76:167--188, 1997.

\bibitem{ChandrasekaranJordan13}
Venkat Chandrasekaran and Michael~I. Jordan.
\newblock Computational and statistical tradeoffs via convex relaxation.
\newblock {\em Proceedings of the National Academy of Sciences}, 110(13):E1181--E1190, 2013.

\bibitem{Chandrasekaranetal12a}
Venkat Chandrasekaran, Benjamin Recht, Pablo~A. Parrilo, and Alan~S. Willsky.
\newblock The convex geometry of linear inverse problems.
\newblock {\em Foundations of Computational Mathematics}, 12:805--849, 2012.

\bibitem{ChatterjiBartlett17}
Niladri~S. Chatterji and Peter~L. Bartlett.
\newblock Alternating minimization for dictionary learning: Local convergence guarantees.
\newblock {\em Advances in Neural Information Processing Systems (NeurIPS)}, 2017.

\bibitem{DashtiStuart}
Masoumeh Dashti and Andrew~M. Stuart.
\newblock Uncertainty quantification and weak approximation of an elliptic inverse problem.
\newblock {\em SIAM Journal on Numerical Analysis}, 49(6):2524--2542, 2011.

\bibitem{Daubechiesetal04}
Ingrid Daubechies, Michel Defrise, and Christine de~Mol.
\newblock An iterative thresholding algorithm for linear inverse problems with a sparsity constraint.
\newblock {\em Communications on Pure and Applied Mathematics}, 57(11):1413--1457, 2004.

\bibitem{Donoho2006}
David Donoho.
\newblock For most large underdetermined systems of linear equations the minimal l1-norm solution is also the sparsest solution.
\newblock {\em Communications on Pure and Applied Mathematics}, 59(6):797--829, 2006.

\bibitem{Durmusetal19}
Alain Durmus, Szymon Majewski, and B\l{}a\.{z}ej Miasojedow.
\newblock Analysis of langevin monte carlo via convex optimization.
\newblock {\em Journal of Machine Learning Research}, 20:1--46, 2019.

\bibitem{Dwivedietal19}
Raaz Dwivedi, Yuansi Chen, Martin~J. Wainwright, and Bin Yu.
\newblock Log-concave sampling: Metropolis-hastings algorithms are fast.
\newblock {\em Journal of Machine Learning Research}, 20:42, 2019.

\bibitem{Eggermont93}
P.~P.~B. Eggermont.
\newblock Maximum entropy regularization of fredholm integral equations of the first kind.
\newblock {\em SIAM Journal on Mathematical Analysis}, 24(6):1557–1576, 1993.

\bibitem{EladSurvey10}
Michael Elad.
\newblock Sparse and redundant representations: From theory to applications in signal and image processing.
\newblock {\em Springer}, 2010.

\bibitem{FanLi2001}
Jianqing Fan and Runze Li.
\newblock Variable selection via nonconcave penalized likelihood and its oracle properties.
\newblock {\em Journal of the American Statistical Association}, 96:1348 -- 1360, 2001.

\bibitem{FazelThesis}
Maryam Fazel.
\newblock Matrix rank minimization with applications.
\newblock {\em Ph.D. Thesis, Department of Electrical Engineering, Stanford University}, 2002.

\bibitem{FoucartLai09}
Simon Foucart and Ming-Jun Lai.
\newblock Sparsest solutions of underdetermined linear systems via $\ell_q$ -minimization for $0 < q \leqslant 1$.
\newblock {\em Applied and Computational Harmonic Analysis}, 26:395--407, 2009.

\bibitem{GarciaCardonaWolhberg18}
Cristina Garcia-Cardona and Brendt Wohlberg.
\newblock Convolutional dictionary learning: A comparative review and new algorithms.
\newblock {\em IEEE Transactions on Computational Imaging}, 4(3):366–381, 2018.

\bibitem{geom-tom-Gardner}
Richard~J. Gardner.
\newblock Geometric tomography.
\newblock {\em Cambridge: Cambridge University Press}, 2006.

\bibitem{GelfandSmith98}
Andrew Gelman and Xiao-Li Meng.
\newblock Simulating normalizing constants: from importance sampling to bridge sampling to path sampling.
\newblock {\em Statistical Science}, 13(2):163--185, 1998.

\bibitem{GiordanoNickl}
Matteo Giordano and Richard Nickl.
\newblock Consistency of bayesian inference with gaussian process priors in an elliptic inverse problem.
\newblock {\em Inverse Problems}, 36(8):085001, 2020.

\bibitem{samp-comp-DL}
R\'{e}mi Gribonval, Rodolphe Jenatton, Francis Bach, Martin Kleinsteuber, and Matthias Seibert.
\newblock Sample complexity of dictionary learning and other matrix factorizations.
\newblock {\em IEEE Transactions on Information Theory}, 61(6):3469--3486, 2015.

\bibitem{Guntuboyina-opt-supp-functions}
Adityanand Guntuboyina.
\newblock Optimal rates of convergence for convex set estimation from support functions.
\newblock {\em Annals of Statistics}, 40(1):385--411, 2012.

\bibitem{HabringHoller22}
Andreas Habring and Martin Holler.
\newblock A generative variational model for inverse problems in imaging.
\newblock {\em SIAM Journal on Mathematics of Data Science}, 4(1):306–335, 2022.

\bibitem{Hansen-starset-survey}
G.~Hansen, I.~Herburt, H.~Martini, and M.~Moszynska.
\newblock Starshaped sets.
\newblock {\em Aequationes Mathematicae}, 94:1001–1092, 2020.

\bibitem{StatLearnSparsityBook}
Trevor Hastie, Robert Tibshirani, and Martin Wainwright.
\newblock Statistical learning with sparsity: The lasso and generalizations.
\newblock {\em Chapman \& Hall/CRC}, 2015.

\bibitem{Hirose1965}
Tunehisa Hirose.
\newblock On the convergence theorem for star-shaped sets in $\mathrm{E}^n$.
\newblock {\em Proc. Japan Acad.}, 41(3):209--211, 1965.

\bibitem{Knapiketal11}
B.~T. Knapik, A.~W. van~der Vaart, and J.~H. van Zanten.
\newblock Bayesian inverse problems with gaussian priors.
\newblock {\em The Annals of Statistics}, 39(5):2626–2657, 2011.

\bibitem{Kobleretal20}
Erich Kobler, Alexander Effland, Karl Kunisch, and Thomas Pock.
\newblock Total deep variation: A stable regularizer for inverse problems.
\newblock {\em arXiv preprint arXiv:2006.08789}, 2020.

\bibitem{LovaszVempala2}
L\'{a}szl\'{o} Lov\'{a}sz and Santosh Vempala.
\newblock Hit-and-run from a corner.
\newblock {\em SIAM Journal on Computing}, 35(4):985--1005, 2006.

\bibitem{LovaszVempala1}
L\'{a}szl\'{o} Lov\'{a}sz and Santosh Vempala.
\newblock The geometry of logconcave functions and sampling algorithms.
\newblock {\em Random Structures \& Algorithms}, 30(3):307--358, 2007.

\bibitem{Lunzetal18}
Sebastian Lunz, Ozan Öktem, and Carola-Bibiane Schönlieb.
\newblock Adversarial regularizers in inverse problems.
\newblock {\em Advances in Neural Information Processing Systems (NeurIPS)}, 31, 2018.

\bibitem{Lutwak1975}
Erwin Lutwak.
\newblock Dual mixed volumes.
\newblock {\em Pacific Journal of Mathematics}, 58(2):531--538, 1975.

\bibitem{Lutwak1990}
Erwin Lutwak.
\newblock Centroid bodies and dual mixed volumes.
\newblock {\em Proceedings of the London Mathematical Society}, s3-60(2):365--391, 1990.

\bibitem{MairaletalSurvey14}
Julien Mairal, Francis Bach, and Jean Ponce.
\newblock Sparse modeling for image and vision processing.
\newblock {\em Foundations and Trends in Computer Graphics and Vision}, 8(2--3):85–283, 2014.

\bibitem{Mohimanietal08}
Hosein Mohimani, Massoud Babaie-Zadeh, and Christian Jutten.
\newblock A fast approach for overcomplete sparse decomposition based on smoothed $\ell_0$ norm.
\newblock {\em IEEE Transactions on Signal Processing}, 57(1):289 -- 301, 2008.

\bibitem{Monardetal21}
Fran\c{c}ois Monard, Richard Nickl, and Gabriel~P. Paternain.
\newblock Consistent inversion of noisy non-abelian x-ray transforms.
\newblock {\em Communications on Pure and Applied Mathematics}, 74(5):1045--1099, 2021.

\bibitem{NyquistShannonSampling}
Harry Nyquist.
\newblock Certain topics in telegraph transmission theory.
\newblock {\em Trans. AIEE.}, 47(2):617--644, 1928.

\bibitem{OlshausenField96}
Bruno~A. Olshausen and David~J. Field.
\newblock Emergence of simple-cell receptive field properties by learning a sparse code for natural images.
\newblock {\em Nature}, 381(6583):607–609, 1996.

\bibitem{OlshausenField97}
Bruno~A. Olshausen and David~J. Field.
\newblock Sparse coding with an overcomplete basis set: A strategy employed by v1?
\newblock {\em Vision in Research}, 37(23):3311--3325, 1997.

\bibitem{OymakHassibi}
Samet Oymak and Babak Hassibi.
\newblock Sharp mse bounds for proximal denoising.
\newblock {\em Foundations of Computational Mathematics}, 16:965–1029, 2016.

\bibitem{Oymaketal13}
Samet Oymak, Christos Thrampoulidis, and Babak Hassibi.
\newblock The squared-error of generalized lasso: A precise analysis.
\newblock {\em 51st Annual Allerton Conference on Communication, Control, and Computing (Allerton)}, 2013.

\bibitem{Papyanetal17}
Vardan Papyan, Yaniv Romano, Jeremias Sulam, and Michael Elad.
\newblock Convolutional dictionary learning via local processing.
\newblock {\em International Conference on Computer Vision (ICCV)}, page 5296–5304, 2017.

\bibitem{PieperPetrosyan22}
Konstantin Pieper and Armenak Petrosyan.
\newblock Nonconvex regularization for sparse neural networks.
\newblock {\em Applied and Computational Harmonic Analysis}, 61:25 -- 56, 2022.

\bibitem{Pinsker64}
Mark~S. Pinsker.
\newblock Information and information stability of random variables and processes.
\newblock {\em Holden-Day, Inc., San Francisco, Calif.-London, Amsterdam}, 1964.

\bibitem{Pollard84}
David Pollard.
\newblock Convergence of stochastic processes.
\newblock {\em Springer-Verlag}, 1984.

\bibitem{Rebeschini-notes}
Patrick Rebeschini.
\newblock Lecture 5: Covering numbers bounds for rademacher complexity. chaining.
\newblock 2021.

\bibitem{Rechtetal10}
Benjamin Recht, Maryam Fazel, and Pablo~A. Parrilo.
\newblock Guaranteed minimum-rank solutions of linear matrix equations via nuclear norm minimization.
\newblock {\em SIAM Review}, 52(3):471–501, 2010.

\bibitem{ReehorstSchniter19}
Edward~T. Reehorst and Philip Schniter.
\newblock Regularization by denoising: clarifications and new interpretations.
\newblock {\em IEEE Transactions on Computational Imaging}, 5(1):52 -- 67, 2019.

\bibitem{RobertsTweedie1}
Gareth~O. Roberts and Richard~L. Tweedie.
\newblock Exponential convergence of langevin distributions and their discrete approximations.
\newblock {\em Bernoulli}, pages 341--363, 1996.

\bibitem{RobertsTweedie2}
Gareth~O. Roberts and Richard~L. Tweedie.
\newblock Geometric convergence and central limit theorems for multidimensional hastings and metropolis algorithms.
\newblock {\em Biometrika}, 83(1):95--110, 1996.

\bibitem{Rockafellar93}
R.~Tyrell Rockafellar.
\newblock Lagrange multipliers and optimality.
\newblock {\em SIAM Review}, 35(2):183--238, 1993.

\bibitem{romano2017RED}
Yaniv Romano, Michael Elad, and Peyman Milanfar.
\newblock The little engine that could: Regularization by denoising (red).
\newblock {\em SIAM Journal on Imaging Sciences}, 10(4):1804--1844, 2017.

\bibitem{Rubinov2000}
A.M. Rubinov.
\newblock Radiant sets and their gauges.
\newblock {\em In: Quasidifferentiability and Related Topics}, 43:235--261, 2000.

\bibitem{Rudinetal92}
Leonid~I Rudin, Stanley Osher, and Emad Fatemi.
\newblock Nonlinear total variation based noise removal algorithms.
\newblock {\em Physica D: nonlinear phenomena}, 60(1-4):259--268, 1992.

\bibitem{Schnass14}
Karin Schnass.
\newblock On the identifiability of overcomplete dictionaries via the minimisation principle underlying k-svd.
\newblock {\em Applied and Computational Harmonic Analysis}, 37(3):464–491, 2014.

\bibitem{Schnass16}
Karin Schnass.
\newblock Convergence radius and sample complexity of itkm algorithms for dictionary learning.
\newblock {\em Applied and Computational Harmonic Analysis}, 45(1):22--58, 2016.

\bibitem{conv-bodies-Schneider}
Rolf Schneider.
\newblock Convex bodies: The brunn–minkowski theory.
\newblock {\em Cambridge: Cambridge University Press.}, 2013.

\bibitem{Selesnick17}
Ivan Selesnick.
\newblock Sparse regularization via convex analysis.
\newblock {\em IEEE Transactions on Signal Processing}, 65(17):4481--4494, 2017.

\bibitem{Shahetal12}
Parikshit Shah, Badri~Narayan Bhaskar, Gongguo Tang, and Benjamin Recht.
\newblock Linear system identification via atomic norm regularization.
\newblock {\em Proceedings of the 51st Annual Conference on Decision and Control}, 2012.

\bibitem{shalev-shwartz-ben-david-ML}
Shai Shalev-Shwartz and Shai Ben-David.
\newblock {\em Understanding Machine Learning: From Theory to Algorithms}.
\newblock Cambridge University Press, USA, 2014.

\bibitem{SohChandrasekaran}
Yong~Sheng Soh and Venkat Chandrasekaran.
\newblock Learning semidefinite regularizers.
\newblock {\em Foundations of Computational Mathematics}, 19:375--434, 2019.

\bibitem{Spielmanetal12}
Daniel~A. Spielman, Huan Wang, and John Wright.
\newblock Exact recovery of sparsely-used dictionaries.
\newblock {\em Conference on Learning Theory}, 23(37):1–18, 2012.

\bibitem{Sprungk20}
Bj\"{o}rn Sprungk.
\newblock On the local lipschitz stability of bayesian inverse problems.
\newblock {\em Inverse Problems}, 36:055015, 2020.

\bibitem{Stuart-BIP}
Andrew~M. Stuart.
\newblock Inverse problems: a bayesian perspective.
\newblock {\em Acta Numerica}, 19:451--559, 2010.

\bibitem{Sunetal16I}
Ju~Sun, Qing Qu, and John Wright.
\newblock Complete dictionary recovery over the sphere i: Overview and the geometric picture.
\newblock {\em IEEE Transactions on Information Theory}, 63(2):853--884, 2016.

\bibitem{Sunetal16II}
Ju~Sun, Qing Qu, and John Wright.
\newblock Complete dictionary recovery over the sphere ii: Recovery by riemannian trust-region method.
\newblock {\em IEEE Transactions on Information Theory}, 63(2):885--914, 2016.

\bibitem{Sun12}
Qiyu Sun.
\newblock Recovery of sparsest signals via $\ell^q$-minimization.
\newblock {\em Applied and Computational Harmonic Analysis}, 32(3):329--341, 2012.

\bibitem{Sojka13}
Grzegorz Sójka.
\newblock Metrics in the family of star bodies.
\newblock {\em Advances in Geometry}, 13:117–144, 2013.

\bibitem{Tangetal13}
Gongguo Tang, Badri~Narayan Bhaskar, Parikshit Shah, and Benjamin Recht.
\newblock Compressed sensing off the grid.
\newblock {\em IEEE Transactions on Information Theory}, 59(11):7465–7490, 2013.

\bibitem{Tibshirani94}
Ryan Tibshirani.
\newblock Regression shrinkage and selection via the lasso.
\newblock {\em Journal of the Royal Statistical Society, Series B}, 58:267–288, 1994.

\bibitem{TikReg}
Andrey~Nikolayevich Tikhonov.
\newblock On stability of inverse problems.
\newblock {\em Dokl. Akad. Nauk SSSR}, 39(5):176--179, 1943.

\bibitem{Traonmilinetal-21}
Yann Traonmilin, R\'{e}mi Gribonval, and Samuel Vaiter.
\newblock A theory of optimal convex regularization for low-dimensional recovery.
\newblock {\em arXiv preprint arXiv:2112.03540}, 2021.

\bibitem{Vainsencher_etal11}
Daniel Vainsencher, Shie Mannor, and Alfred~M. Bruckstein.
\newblock The sample complexity of dictionary learning.
\newblock {\em Journal of Machine Learning Research (JMLR)}, 12:3259--3281, 2011.

\bibitem{VempalaSurvey}
Santosh Vempala.
\newblock Geometric random walks: A survey.
\newblock {\em Combinatorial and Computational Geometry}, 52(2):573--612, 2005.

\bibitem{venkatakrishnan2013pnp}
Singanallur~V Venkatakrishnan, Charles~A Bouman, and Brendt Wohlberg.
\newblock Plug-and-play priors for model based reconstruction.
\newblock In {\em 2013 IEEE Global Conference on Signal and Information Processing}, pages 945--948. IEEE, 2013.

\bibitem{VershyninHDP}
Roman Vershynin.
\newblock High-dimensional probability: An introduction with applications in data science.
\newblock {\em Cambridge University Press}, 2020.

\bibitem{Villani-OT}
C\'{e}dric Villani.
\newblock Topics in optimal transport.
\newblock {\em Providence, RI: American Mathematical Society}, 2003.

\bibitem{Wangetal14}
Zhaoran Wang, Han Liu, and Tong Zhang.
\newblock Optimal computational and statistical rates of convergence for sparse nonconvex learning problems.
\newblock {\em Annals of Statistics}, 42(6):2164 -- 2201, 2014.

\bibitem{Wohlberg15}
Brendt Wohlberg.
\newblock Efficient algorithms for convolutional sparse representations.
\newblock {\em IEEE Transactions on Image Processing}, 25(1):301--315, 2015.

\bibitem{YaoKwok18}
Quanming Yao and James~T. Kwok.
\newblock Efficient learning with a family of nonconvex regularizers by redistributing nonconvexity.
\newblock {\em Journal of Machine Learning Research (JMLR)}, 18:1--52, 2018.

\bibitem{Zhang10}
Cun-Hui Zhang.
\newblock Nearly unbiased variable selection under minimax concave penalty.
\newblock {\em Annals of Statistics}, 38(2):894 -- 942, 2010.

\end{thebibliography}

\end{document}